\documentclass[reqno, 10 pt]{article}

\usepackage[small]{titlesec}
\titleformat{\subsection}[runin]
{\normalfont\bfseries}{\thesubsection}{.5em}{}
\titleformat{\subsubsection}[runin]
{\normalfont\bfseries}{\thesubsubsection}{.5em}{}

\usepackage{geometry}

\usepackage{mathdesign,eucal}
\usepackage{
  amsmath,
  amssymb,
  amsthm,
  hyperref,
  paralist,
  xypic,
  graphicx,
  thmtools,
}

\usepackage{mathrsfs}

\usepackage[center]{subfigure}
\usepackage{setspace}

\usepackage[all]{xy}

\title{Special Codimension One Loci in Hurwitz Spaces}
\author{Anand Patel}

\newtheorem{thm}{Theorem}
\newtheorem*{thm*}{Theorem}
\newtheorem{thm-defn}[thm]{Theorem/Definition}
\newtheorem{lem}[thm]{Lemma}
\newtheorem{prop}[thm]{Proposition}
\newtheorem{cor}[thm]{Corollary}

\newtheorem{question}[thm]{Question}

\declaretheorem[title=Theorem]{maintheorem}

\theoremstyle{definition}
\newtheorem{defn}[thm]{Definition}

\theoremstyle{remark}
\newtheorem{rmk}[thm]{Remark}

\numberwithin{equation}{section}
\numberwithin{thm}{section}
\numberwithin{prop}{section}
\numberwithin{lem}{section}
\numberwithin{cor}{section}
\numberwithin{conj}{section}
\numberwithin{defn}{section}

\newcommand{\M}{\mathcal M}
\renewcommand{\H}{\mathcal H}
\newcommand{\td}{\widetilde}
\newcommand{\F}{{\mathbf F}}

\newcommand{\Hdg}{\overline{\H}_{d,g}}
\newcommand{\Hdgn}{\overline{\H}_{d,g}}
\newcommand{\Hdgnt}{\overline{\H}_{d,g}^{(2)}}
\newcommand{\Hdgnth}{\overline{\H}_{d,g}^{(3)}}

\renewcommand{\to}{\longrightarrow}

\newcommand{\sfA}{{\mathsf A}}

\newcommand{\sfY}{{\mathsf Y}}
\newcommand{\sfX}{{\mathsf X}}
\newcommand{\sfT}{{\mathsf T}}
\newcommand{\sfD}{{\mathsf D}}
\newcommand{\sfB}{{\mathsf B}}
\newcommand{\sfC}{{\mathsf C}}
\newcommand{\sfCE}{{\mathsf {CE}}}
\newcommand{\sfM}{{\mathsf M}}

\newcommand{\Mg}{{\overline{\mathcal{M}}_{g}}}

\newcommand{\cA}{{\mathcal A}}
\newcommand{\cC}{{\mathcal C}}
\renewcommand{\cD}{{\mathcal D}}
\newcommand{\cE}{{\mathcal E}}
\newcommand{\cF}{{\mathcal F}}
\newcommand{\cG}{{\mathcal G}}

\newcommand{\cK}{{\mathcal K}}

\newcommand{\cM}{{\mathcal M}}
\newcommand{\cN}{{\mathcal N}}
\newcommand{\cO}{{\mathcal O}}
\newcommand{\cP}{{\mathcal P}}

\newcommand{\cS}{{\mathcal S}}

\newcommand{\cV}{{\mathcal V}}
\newcommand{\cW}{{\mathcal W}}

\newcommand{\cY}{{\mathcal Y}}

\newcommand{\bG}{{\bf G}}

\newcommand{\comment}[1]{}

\newcommand {\Q}{{\bf Q}}

\newcommand {\E}{{\bf E}}
\newcommand {\C}{{\bf C}}
\newcommand {\A}{{\bf A}}
\renewcommand{\P}{{\bf P}}
\renewcommand{\F}{{\bf F}}

\newcommand {\from}{{\colon}}
\newcommand{\into}{\,{\hookrightarrow}\,}
\newcommand{\onto}{\twoheadrightarrow}
\renewcommand {\o}[1]{\overline{#1}}    
\newcommand{\dual}[1]{#1^{\vee}}
\newcommand{\compl}[1]{\widehat{#1}}
\newcommand{\tw}[1]{\widetilde{#1}}

\newcommand{\Br}{{\rm Br}}
\newcommand {\gen}{{\rm gen}}
\renewcommand{\H}{{\mathcal{H}}}
\newcommand{\Spec}{{\text{\rm Spec}\,}}
\newcommand{\Spf}{{\text{\rm Spf}\,}}
\newcommand{\Proj}{{\text{\rm Proj}\,}}
\newcommand{\rk}{{\text{rk}\,}}
\newcommand {\Eff}{{\rm Eff}\,}
\newcommand {\ns}{{\rm ns}}
\newcommand {\Pic}{{\rm Pic}\,}
\newcommand {\Sym}{{\rm Sym}}
\newcommand {\irr}{{\rm irr}}
\renewcommand {\top}{{\rm top}}
\renewcommand {\gen}{{\rm gen}}
\newcommand {\spl}{{\rm split}}
\newcommand {\ram}{{\rm ram}}
\newcommand {\End}{{\rm End \,}}
\newcommand {\Ext}{{\rm Ext \,}}
\newcommand {\PGL}{{\rm PGL}}
\newcommand {\length}{{\rm length}}
\newcommand {\res}{{\rm res}}
\begin{document}

\maketitle
\abstract{We investigate two families of divisors which we expect to play a distinguished role in the global geometry of Hurwitz space. In particular, we show that they are extremal and rigid in the small degree regime $d \leq 5$. We further show their significance in the problem of computing the sweeping slope of Hurwitz space in these degrees. In the process, we prove various general results about the divisor theory of Hurwitz space, including a proof of the independence of the boundary components of the admissible covers compactification. Some basic open questions and further directions are discussed at the end.}

\subsection{Introduction.}

Hurwitz space, the parameter space of degree $d$ genus $g$ branched covers of $\P^{1}$, has been studied by many mathematicians for over a century.  One of the main threads in this long history concerns the ``Hurwitz number problem'' which asks to enumerate the number of branched covers with prescribed ramification behaviour above fixed points in $\P^{1}$. 

Along similar lines, much effort has gone into understanding topological properties of the ``boundary" of Hurwitz space -- by this we mean the locus of covers with non-simple branching or small Galois group.  This ``boundary'' is naturally stratified, and many have studied this stratification. One of the main motivating problems here is to classify irreducible components of the stratification.  (The list of results devoted to these topics is so long that the author would rather not cite anyone for fear of leaving someone out.) 

Parameter spaces of branched covers have also played various important auxiliary roles in the study of $\M_{g}$, the moduli space of curves. The first instance of this is the theorem typically attributed to Clebsch \cite{clebsch-irreducibility} which says that Hurwitz space is irreducible, implying the irreducibility of $\M_{g}$. Diaz's theorem \cite{diaz-complete}, which gives a nontrivial upper bound on the dimension of a proper subvariety of $\M_{g}$, makes essential use of Hurwitz spaces.  The celebrated theorem on the Kodaira dimension of $\M_{g}$ \cite{Harris82:_Kodair_Dimen_Of_Modul_Space_Of_Curves} uses the divisor class expression for the locus of $k$-gonal curves, with $k = \lfloor g/2 \rfloor$, $g$ odd. Many relations in the tautological ring of $\M_{g}$ have been found using an alternate compactification of the space of branched covers \cite{pand-pix}.  The point here is that spaces of branched covers occur as central tools in the study of curves.

What is lacking is a study of Hurwitz space from a more algebro-geometric and less combinatorial or topological point of view, especially a study of subvarieties parametrizing branched covers with special algebro-geometric properties; properties not captured by degenerate branching behaviour or Galois group.  The literature here is much more sparse, and our intention in this paper is to introduce the special role that certain subvarieties play in the algebraic geometry of Hurwitz space. 

Our primary focus in this paper is to study particular codimension one loci in the {\sl interior} of Hurwitz spaces.  These subvarieties are called the {\sl Maroni} and {\sl Casnati-Ekedahl} divisors, and our goal is to explain their distinguished role in the divisor theory of Hurwitz spaces of low degree. They arise as a ``cohomology jump'' phenomenon, and are similar in spirit to the ``Koszul divisors'' of G. Farkas \cite{farkas-koszul}.  It should be emphasized that the Maroni and Casnati-Ekedahl divisors have very little to do with branching behavior or Galois groups of covers -- they arise from considering finer algebraic invariants of a cover, generalizing the classical {\sl Maroni invariant} of trigonal curves first discovered by Maroni in \cite{maroni}.

\subsection{Outline and main results.}
In \autoref{FirstSection}, we define the Maroni and the Casnati-Ekedahl divisors, denoted by $\sfM$ and $\sfCE$ respectively, and briefly introduce the reader to a broader context in which they arise.  In particular, we call attention to a very ill-understood stratification of Hurwitz space arising from a well-known general structure theorem due to Casnati and Ekedahl \cite{casnati96:_cover_of_algeb_variet_i}. 

Next, in \autoref{SecondSection}, we abruptly change topic and prove a basic theorem about the boundary components of the admissible cover compactification $\Hdg^{\rm adm}$ and its ``normalization,'' the space of twisted stable maps $\Hdg$. The main result we show is: 
\begin{maintheorem}[Independence of the boundary]\label{independence} The divisor classes of boundary components are independent in $\Pic \Hdg \otimes \Q$.
\end{maintheorem}  We introduce the concept of a {\sl partial pencil}, a particular type of one parameter family of admissible covers which will be used heavily in the later parts of the paper. 

In section \autoref{ThirdSection}, we obtain partial expressions for the divisor class expressions of $\sfM$ and $\sfCE$ in terms of some natural divisor classes on Hurwitz space.  We also recall the classical unirational parametrizations of Hurwitz space when $d \leq 5$.  Using these parametrizations, we are able to show: 

\begin{maintheorem}[Rigidity and Extremality]\label{rigidityextremality} Apart from one exception, the Maroni and Casnati-Ekedahl divisors are rigid and extremal in $\Hdg$ when $d \leq 5$.  The exception is $\o{\H}_{4,3}$, where the Casnati-Ekedahl divisor has two irreducible components, each of which is rigid and extremal.
\end{maintheorem} 
We view this result as a first validation of our basic claim that these divisors play a distinguished role in the geometry of $\H_{d,g}$. We end the section by indicating how one might hope to tackle the question of rigidity or extremality for higher degrees $d$. 

Section \autoref{slopeproblem} is the most technical part of the paper. Its purpose is to demonstrate the role that $\sfM$ and $\sfCE$ play in the well-known problem of determining a sharp upper bound for the {\sl slope} of a ``sweeping''  complete one parameter family $B \subset \Hdg$. Recall that the slope $s(B)$ is the ratio $\delta \cdot B/ \lambda \cdot B$, and we define $$s(\Hdg) := \sup_{\text{$B \subset \Hdg$ sweeping}} s(B).$$

Cornalba and Harris \cite{cornalba88:_divis} show that $s(\o{\H}_{2,g}) = 8 + 4/g$, while Stankova \cite{stankova} shows that $s(\o{\H}_{3,g}) = 7 + 6/g$ when $g$ is even.  Deopurkar and the author \cite{deopurkar_patel:sharp_trigonal} completed the trigonal case by showing that $s(\o{\H}_{3,g}) = 7 + 20/(3g+1)$ when $g$ is odd. 
Beorchia and Zucconi \cite{beorchia12} consider slopes of a large class of one parameter families in $\o{\H}_{4,g}$ and produce slope bounds for families in this large class.  Their focus is not on sweeping families -- in fact they consider a larger class of curves and therefore their bounds are different from the number $s(\o{\H}_{4,g})$.  

The main result of this section is:
\begin{maintheorem}[Slope bounds] \label{slope45}

\begin{enumerate}
\item (Stankova \cite{stankova}) Let $g \geq 4$ be even. Then \[s(\o{\H}_{3,g}) = 7 + \frac{6}{g}.\]
\item (P --) Let $g = 3 \mod  6$. Then \[s(\o{\H}_{4,g}) = \frac{13}{2} + \frac{15}{2g}.\] 
\item (P --) Let $g = 16 \mod 20$.  Then \[s(\o{\H}_{5,g}) = \frac{31}{5} + \frac{44}{5g}.\] 
\end{enumerate}

In all three cases there exist sweeping families achieving the given bound.
\end{maintheorem}
As a consequence of this theorem, we also deduce the same slope bounds for sweeping families in the corresponding $d$-gonal locus $\cM^{1}_{d,g} \subset \Mg$.

\begin{rmk}
We prove \autoref{slope45} for the slightly different compactifications $\Hdgnt$. The same proof method applies for $\Hdg,$ as we explain in proof \ref{proofslope45}.
\end{rmk}

\subsection{Acknowledgements.} I am grateful to Joe Harris for introducing me to these topics. I thank Gabriel Bujokas and Anand Deopurkar for uncountably many useful conversations. 

\subsection{Notation and conventions.} We work over an algebraically closed field $k$ of characteristic $0$. All stacks will be Deligne-Mumford stacks. 

If $\cW$ is a locally free sheaf, then $\P \cW$ will denote the projective bundle $\Proj(\Sym^{\bullet} \cW)$, and $\zeta_{\cW}$ will denote the divisor class associated to the natural $\cO(1)$.

A subscript under $\P^{1}$, e.g. $\P^{1}_{t}$, will be used to distinguish it from other rational curves being discussed. $t$ will denote a coordinate on $\P^{1}_{t}$.

\section{Invariants of a branched cover.}\label{FirstSection}

\subsection{The Maroni invariants and Tschirnhausen bundle.}
A branched cover $\alpha \from C \to \P^{1}$ has a {\sl relative canonical embedding} $$i \from C \hookrightarrow \P(\cE) \to \P^{1}$$ where $\pi \from \P(\cE) \to \P^{1}$ is a $\P^{d-2}$-bundle. The scroll $\P(\cE)$ can be realized as the set of $(d-2)$-planes spanned by the degree $d$ divisors $\alpha^{-1}(t)$ in the canonical space of $C$. (This geometric realization makes sense when $d < g$.)

Alternatively, $\alpha$ is described by a linear family of degree $d$ divisors on the curve $C$, i.e. we have an inclusion $\P^{1} \subset \Sym^{d}(C)$.  The projectivized normal bundle $\P(N_{\P^{1}/\Sym^{d}(C)})$ is isomorphic to the scroll $\P(\cE)$ mentioned earlier. (This description makes sense for all $d, g$.) 

Finally, there is a purely algebraic definition of the vector bundle $\cE$ which was implicitly mentioned above. The map $\alpha$ gives an exact sequence in which the dual of $\cE$ sits as the cokernel:
\begin{equation}\label{structure sheaf sequence}
  0 \to \cO_{\P^{1}} \to \alpha_* \cO_{C} \to {\cE}^\vee \to 0.
\end{equation}
 Denote by $\omega_\alpha$ the dualizing sheaf of $\alpha$. Applying $Hom_{\P^{1}}(-,\cO_{\P^{1}})$ to \eqref{structure sheaf sequence}, we get
\begin{equation}\label{dual structure sheaf sequence}
  0 \to \cE \to \alpha_* \omega_\alpha \to \cO_{\P^{1}} \to 0.
\end{equation}
The map $\cE \to \alpha_* \omega_\alpha$ induces a map $\alpha^* \cE \to \omega_\alpha$ on the curve $C$, which by a theorem of Casnati and Ekedahl \cite{casnati96:_cover_of_algeb_variet_i}[Theorem~2.1] induces the relative canonical map $i \from C \to \P(\cE)$.  From this definition, it is easy to see that the degree of $\cE$ is $g+d-1$, while the rank of $\cE$ is clearly $d-1$.

The vector bundle $\cE$ is called the {\sl Tschirnhausen bundle} of the cover $\alpha$, and its splitting type can be viewed as a set of discrete invariants for a branched cover. These invariants are called the {\sl scrollar invariants} or {\sl Maroni invariants}, and it is shown in \cite{coppens_existence} and \cite{ballico-scrollar} that for a Zariski open subset of Hurwitz space, the splitting type of $\cE$ is {\sl balanced}, i.e. $\Ext^{1}(\cE,\cE) = 0$. In the case $d = 3$, the scrollar invariant is simply the classical Maroni invariant; this explains the following definition:

\begin{defn}\label{defn:Mspecial}
A branched cover $\alpha \from C \to \P^{1}$ is {\sl Maroni-special} if its Tschirnhausen bundle $\cE$ is not balanced. We let ${\mathsf M} \subset \H_{d,g}$ denote the closed substack parametrizing Maroni-special curves.
\end{defn}

Deopurkar and the author prove in \cite{dp:pic_345}[Theorem~2.10] that the locus $\sfM$ has codimension one only when the divisibility condition $(d-1) \mid g$ holds, in which case the locus is nonempty and irreducible. It is an interesting open problem, in these cases, to determine if $\sfM$ is rigid or extremal as an effective divisor.  \autoref{rigidityextremality} establishes these properties of $\sfM$ for branched covers of degrees $d \leq 5$.

The scrollar invariants have certainly appeared before in the literature, most notably in the works \cite{ballico-scrollar}, \cite{coppens_existence}, \cite{schreyer-scroll} and \cite{ohbuchi:relations}.  However, the study of the geometry of the locus of Maroni-special curves has only recently attracted attention -- see \cite{dp:pic_345} for more on this topic.

\begin{rmk}
The splitting type of the vector bundle $\cE$ encodes, and is determined by, the dimensions of the powers of the $g^{1}_{d}$ given by the map $\alpha$.
\end{rmk}

\subsection{The Casnati-Ekedahl invariants and the bundle of quadrics.}
There are more discrete invariants coming from the geometry of the relative canonical embedding which are far less understood, and they arise as the splitting type of a different vector bundle.  Part of our aim in this paper is to highlight their significance. 

For every $t \in \P^{1}$, we let $\P^{d-2}_{t}$ be the fiber of the projection $\pi \from \P(\cE) \to \P^{1}$.  The general structure theorem for branched covers, proved by Casnati and Ekedahl in \cite{casnati96:_cover_of_algeb_variet_i}, shows that the degree $d$ divisor $\alpha^{-1}(t) \subset \P^{d-2}_{t}$ is an arithmetically Gorenstein subscheme for all $t$.  In particular, this implies that the vector spaces $$\cF_{t} := \{\text{Quadrics in $\P^{d-2}_{t}$ containing $\alpha^{-1}(t)$}\}$$
have the same rank, independent of $t$, and therefore glue together to form a vector bundle $\cF$ on $\P^{1}$, which we call the {\sl bundle of quadrics}.  

Algebraically, the bundle $\cF$ can be realized easily via its relationship with the Tschirnhausen bundle $\cE$.  Indeed, from \ref{dual structure sheaf sequence}, we conclude that there is a map $$\Sym^{2}\cE \to \alpha_{*}(\omega_{\alpha}^{\otimes 2})$$ which is easily seen to be surjective. The bundle $\cF$ is the kernel of this map.

\begin{rmk}
We are not aware of a nice description of the bundle $\cF$ in terms of the geometry of the singularity $x \in W_{d} \subset \Pic^{d}(C)$, as we had for $\cE$.
\end{rmk}

The splitting type of $\cF$ provides an additional discrete invariant of a branched cover which is essentially ``independent'' of the scrollar invariants, in the sense that one set of invariants does not determine the other. The isomorphism class of $\cF$ for a generic curve is understood under the provision that $g \gg d$: The Main Theorem of \cite{bujokas-patel} states that $\cF$ is balanced for a Zariski open set of covers, i.e. $\Ext^{1}(\cF, \cF) = 0.$  In complete analogy with \autoref{defn:Mspecial}, we make the following definition:

\begin{defn}\label{def:CEspecial}
  We say a cover $\alpha \from C \to \P^{1}$ is {\sl $\sfCE$-special} if its bundle of quadrics $\cF$ is not balanced, i.e. $\Ext^{1}(\cF,\cF) \neq 0$. We let $\sfCE \subset \H_{d,g}$ denote the closed substack parametrizing $\sfCE$-special covers.  
\end{defn}

We have not fully classified the circumstances under which the locus $\sfCE$ is divisorial, but the natural expectation is that this occurs precisely when the rank $\rk \cF$ divides the degree $\deg \cF$. This translates into the divisibility condition $d(d-3)/2 \mid (d-3)(g+d-1)$, and it follows from the main theorem of \cite{bujokas-patel} that for such pairs $(g,d)$ the $\sfCE$-locus is an effective divisor. As in the case of $\sfM$, it is an interesting problem to establish in these instances whether the divisor $\sfCE$ is rigid or extremal -- we establish this when $d \leq 5$ in \autoref{rigidityextremality}. As we discuss in the last section, the Casnati-Ekedahl divisor may lose some rigidity properties when $d \geq 6$, although we cannot prove this.

\begin{rmk}
When we refer to the divisor $\sfM$ or $\sfCE$ in a compactification of $\H_{d,g}$ we always take the {\sl closure} of the respective loci in the interior.
\end{rmk}

\subsection{} The bundles $\cE$ and $\cF$ are two bundles among a whole host of bundles naturally associated to a branched cover, as explained by a theorem of Casnati and Ekedahl.  We review this fundamental theorem here. 

 Let $X$ and $Y$ be integral schemes and $\alpha \from X \to Y$ a finite flat Gorenstein morphism of degree $d \geq 3$. The map $\alpha$ gives an exact sequence
\begin{equation}\label{structure sheaf sequence}
  0 \to \cO_Y \to \alpha_* \cO_X \to {\cE_\alpha}^\vee \to 0,
\end{equation}
where $\cE = \cE_\alpha$ is a vector bundle of rank $(d-1)$ on $Y$, called the \emph{Tschirnhausen bundle} of $\alpha$. Denote by $\omega_\alpha$ the dualizing sheaf of $\alpha$. Applying $Hom_Y(-,\cO_Y)$ to \eqref{structure sheaf sequence}, we get
\begin{equation}\label{dual structure sheaf sequence}
  0 \to \cE \to \alpha_* \omega_\alpha \to \cO_Y \to 0.
\end{equation}
The map $\cE \to \alpha_* \omega_\alpha$ induces a map $\alpha^* \cE \to \omega_\alpha$. 

\begin{thm}\label{thm:CE}
  \cite[Theorem~2.1]{casnati96:_cover_of_algeb_variet_i}
  In the setting above, $\alpha^* \cE \to \omega_\alpha$ gives an embedding $\iota \from X \to \P \cE$ with $\alpha  = \pi \circ \iota$, where $\pi \from \P \cE \to Y$ is the projection. Moreover, the subscheme $X \subset \P \cE$ can be described as follows.
\begin{enumerate}
\item The resolution of $\cO_X$ as an $\cO_{\P \cE}$-module has the form
\begin{equation}\label{eqn:casnati_resolution}
  \begin{split}
    0 \to \pi^* \cN_{d-2} (-d) \to \pi^* \cN_{d-3}(-d+2) \to \pi^*\cN_{d-4}(-d+3) \to \dots \\
    \dots \to \pi^*\cN_2(-3) \to \pi^*\cN_1(-2) \to \cO_{\P \cE} \to \cO_X \to 0,
  \end{split}
\end{equation}
where $\cN_i$ are vector bundles on $Y$. Restricted to a point $y \in Y$, this sequence is the minimal free resolution of length $d$, zero dimensional scheme $X_y \subset \P \cE_y\cong \P^{d-2}$.
\item The ranks of the $\cN_i$ are given by
  \[ \rk \cN_i = \frac{i(d-2-i)}{d-1} {d \choose {i+1}},\]
\item We have $\cN_{d-2} \cong \pi^* \det \cE$. Furthermore, the resolution is symmetric, that is, isomorphic to the resolution obtained by applying $Hom_{\cO_{\P \cE}}(-, \cN_{d-2}(-d))$.
\end{enumerate}
\end{thm}

The branch divisor of $\alpha \from X \to Y$ is given by a section of $(\det \cE)^{\otimes 2}$. In particular, if $X$ is a curve of (arithmetic) genus $g$, $\alpha$ has degree $d$, and $Y = \P^1$, then 
\begin{equation}\label{eqn:rk_deg_E}
 \rk \cE = d-1 \,\, \text{and } \deg \cE = g+d-1.
\end{equation}

\subsection{} By considering a particular bundle $\cN_{i}$ occurring in the Casnati-Ekedahl resolution \eqref{eqn:casnati_resolution}, one can consider the loci in Hurwitz space parametrizing covers whose bundle $\cN_{i}$ has a prescribed splitting type.  In this way, Hurwitz space becomes stratified by many algebro-geometric invariants.  This stratification is currently very mysterious -- in this paper we focus only on the invariants associated to the bundles $\cE$ and $\cF := \cN_{1}$.

\subsection{The Maroni loci $\sfM(\cE)$.}  The Hurwitz space $\H_{d,g}$ contains natural subvarieties (or substacks) which we call the {\sl Maroni loci} $\sfM(\cE)$.  More precisely, we make the following definition:
\begin{defn}\label{def:maroni} Let $\cE$ be a rank $d-1$, degree $g+d-1$ vector bundle on $\P^1$.  Then the {\sl Maroni locus} $\sfM(\cE)$ is defined as
\[\sfM(\cE) := \overline {\{[\alpha \from C \to \P^1] \in \H_{d,g} \mid \cE_{\alpha} = \cE\}}.\] 
\end{defn}

Clearly if $\sfM(\cE)$ is nonempty then  $\deg \cE = g+d-1$ and $\rk \cE = d-1$.  However, more can be said about $\cE$ in this situation.  For instance, $\cE$ cannot have any nonpositive summands, since this would force $C$ to be disconnected.  Write \[\mathcal{E} = \cO(a_{1}) \oplus \cO(a_{2}) ... \oplus \cO(a_{d-1})\] with $0 < a_{1} \leq a_{2} \leq ... \leq a_{d-1}$.   and set $\lfloor \cE \rfloor : = a_{1}$. 
\begin{defn}\label{tame}
We maintain the notation above. A vector bundle $\cE$ of rank $d-1$ and degree $g+d-1$  on $\P^1$ is {\sl tame} if $a_{i} - a_{i-1} \leq \lfloor \cE \rfloor$ for all $i = 2, ... ,d-1$.
\end{defn}

Among all tame bundles $\cE$ with fixed $\lfloor \cE \rfloor = m$, there is a most generic one, which we call $\cE[m]$.  It is the unique tame bundle which maximizes the sum $(d-1)a_{1} + (d-2)a_{2} + ... + a_{d-1}$. 

The most comprehensive theorem on the geometry of the Maroni loci to date is found in \cite{dp:pic_345}, which relies heavily on the work of Ohbuchi \cite{ohbuchi:relations} and Coppens \cite{coppens_existence}:

\begin{thm}[Deopurkar, P- \cite{dp:pic_345}]\label{thm:maroniloci}
  Let $m$ be an integer satisfying $\frac{g+d-1}{{d \choose 2}} \leq m \leq  \frac{g+d-1}{d-1} $.
  \begin{enumerate}
  \item If $\sfM(\cE)$ is nonempty, then $\cE$ is a tame bundle.
  \item If $\lfloor \cE \rfloor \leq m$ then $\sfM(\cE) \subset \sfM(\cE[m])$.
  \item $\sfM(\cE[m]) \subset \sfM(\cE[m+1])$ for all $m$.   
  \item $\sfM(\cE[m])$ is an irreducible subvariety of $\H_{d,g}$ of codimension $g-(d-1)m + 1$ unless $m = \lfloor \frac{g+d-1}{d-1} \rfloor$, in which case $\sfM(\cE[m]) =  \H_{d,g}$.
  \end{enumerate}
\end{thm}

In particular, point $(4)$ says that a general cover $\alpha \from C \to \P^1$ has a balanced Tschirnhausen bundle $\cE_{\alpha}$. 

\subsubsection{The Maroni divisor $\sfM \subset \H_{d,g}$.}  Item $(4)$ in \autoref{thm:maroniloci} gives a complete description of all divisorial Maroni loci.   Setting the codimension $g-(d-1)m +1$ equal to $1$ gives the divisibility relation 
\begin{equation}\label{cong:maroni}
(d-1) \mid g. 
\end{equation}
When this congruence condition is met, the general Tschirnhausen bundle $\cE^{gen}$ is {\sl perfectly} balanced, i.e. 
\[\cE^{gen} = \cO(k)^{\oplus d-1}\] where $k = g/(d-1) +1$.
Consider the special bundle 
\[\cE^{sp} := \cO(k-1) \oplus \cO(k)^{d-3} \oplus \cO(k+1).\] Then the Maroni locus $\sfM(\cE^{sp})$ is precisely the locus $\sfM$ from \autoref{defn:Mspecial}.

\subsection{The Casnati-Ekedahl loci $\sfC(\cF)$.} We now shift our attention to the bundle of quadrics $\cF_{\alpha}$ associated to a cover $\alpha \from C \to \P^1.$  As in \autoref{def:maroni}, we can consider the loci in $\H_{d,g}$ associated to the bundle $\cF$: 
\begin{defn}\label{def:casnati}
Let $\cF$ be a rank $d(d-3)/2$, degree $(d-3)(g+d-1)$ vector bundle on $\P^1$.  Then the {\sl Casnati-Ekedahl locus} $\sfC(\cF)$ is defined as
\[\sfC(\cF) := \overline {\{[\alpha \from C \to \P^1] \in \H_{d,g} \mid \cF_{\alpha} = \cF\}}.\] 
\end{defn}

Very little is known about the Casnati-Ekedahl loci $\sfC(\cF)$ - a general theorem analogous to part $(4)$ of \autoref{thm:maroniloci} has only recently been shown in \cite{bujokas-patel} using a somewhat delicate degeneration argument. We provide here a simpler proof in the case $d \leq 5$ because we will need to refer to the method of proof later. We note that Bopp proves the case of degree $5$ (the only interesting case) using very different methods in \cite{bopp5}. 
\begin{prop}\label{gen:casnati}
If $d \leq 5$ then $\cF_{\alpha}$ is balanced for a general cover $\alpha \in \H_{d,g}$.  
\end{prop}

We will use a degeneration argument.  To begin, we establish the proposition for $g \leq 4$, $d \leq 5$ by a case by case analysis.  We also record the genus $g=0$ and $g=1$ cases for arbitrary $d$:
\begin{lem} \label{EF01}
\begin{enumerate}
\item Let $\alpha \colon R \to \P^1$ and $\beta \from X \to \P^1$ be degree $d$ covers where $R$ is a smooth rational curve and $X$ is a smooth elliptic curve. Then the following hold:
\begin{enumerate} 
\item $\cE_{\alpha} = \cO(1)^{\oplus d-1}$
\item $\cF_{\alpha} = \cO(1)^{\oplus d-3} \oplus \cO(2)^{\oplus {d-2 \choose 2}}$
\item $\cE_{\beta} = \cO(1)^{\oplus d-2} \oplus \cO(2)$
\item $\cF_{\beta} =  \cO(2)^{\oplus \frac{d(d-3)}{2}}$
\end{enumerate}

\item Let $\gamma \from C \to \P^1$ be a general degree $d \leq 5$ cover, with $g(C) \leq 4$.  Then $\cF_{\gamma}$ is balanced. 
\end{enumerate}
\end{lem}

\begin{proof}[Proof of \autoref{EF01}:] 
\begin{enumerate}
\item We simply check all cases:
\begin{enumerate}

\item All summands of $\cE_{\alpha}$ are positive and add up to $ d-1$.  Therefore, all summands have degree $1$.
\item Using the relative canonical factorization $\iota \from R \to \P \cE$, we may think of $R$ as lying inside $\P \cE$. The series $|\zeta - f|$ on $\P \cE$ restricts to the complete series $\cO_{R}(d-2)$, and the projection morphism $q \from \P \cE \to \P^{d-2}$ given by the series $|\zeta - f|$ restricts to the embedding of $R$ into $\P^{d-2}$ as a rational normal curve. 

The rational normal curve $R \subset \P^{d-2}$ is contained in a ${d-2 \choose 2}$ dimensional space of quadrics.  The class of these quadrics, when pulled back along the projection $q$, is $2\zeta - 2f$.  Recall the exact sequence on the target $\P^1$: 
\begin{equation}\label{FEexact}
0 \to \cF_{\alpha} \to S^{2}\cE_{\alpha} \to \alpha_{*}({\omega^{\otimes 2}_{\alpha}}) \to 0.
\end{equation}
We twist by $\cO(-2)$ and consider global sections.  The previous paragraph along with projective normality of $R \in \P^{d-2}$ imply that $h^{0}(\cF_{\alpha}(-2)) = {d-2 \choose 2}$.  We conclude by noting that \eqref{FEexact} shows that no summand of $\cF_{\alpha}$ may exceed $2$, and the degree of $\cF_{\alpha}$ must be $(d-3)(d-1)$.  This forces the splitting type of $\cF_{\alpha}$ to be $\cO(1)^{\oplus d-3} \oplus \cO(2)^{\oplus {d-2 \choose 2}}$.

\item All summands of $\cE_{\beta}$ are positive, and their degrees sum to $d$.  Therefore $\cE_{\beta}$ must be as indicated.

\item Analogous to part $(b)$. However, we give a seperate proof which will be used later in \autoref{avoidCE}. The elliptic curve $X$ maps to $\P^{d-1}$ via the complete series $|\beta^{*}\cO(1)|$ as an elliptic normal curve, and the map $\beta \from X \to \P^{1}_{s}$ is given by a pencil of hyperplane sections $\beta^{-1}(s) = H_{s} \cap X$. For each hyperplane section $\beta^{-1}(s),$  it is well known that the restriction map 
\begin{equation}\label{restrictionelliptic}
\res_{s} \from \{\text{Quadrics containing $X$}\} \to \{\text{Quadrics in $H_s$ containing $\beta^{-1}(s)$}  \}
\end{equation}
 is an isomorphism. (Note that no quadric containing $X$ may contain a hyperplane $H_{s}$.)

Since the domain of this restriction map is independent of the parameter $s$, we conclude that $\cF_{\beta}$ must be {\sl perfectly} balanced.  Then, for degree reasons, we conclude $(d)$. 
\end{enumerate}
\item Again, we check all cases.  We need only consider $d=4$ and $d=5$, as these are the only instances when $\cF$ exists at all.  We label each case by the pair $(g,d)$, furthermore, we always assume $C$ is general in moduli.
\begin{enumerate}
\item[$(2, 4)$:] The degree of $\cF$ is $5$, and since it is a subbundle of $S^2 \cE = S^2 (\cO(1) \oplus \cO(2) \oplus \cO(2))$, we conclude that either $\cF = \cO(1) \oplus \cO(4)$ or $\cF = \cO(2) \oplus \cO(3)$.  The former cannot happen. Indeed, the curve $C \subset \P \cE$ would then be a complete intersection of two divisors $Q_{1}$ and $Q_{2}$ with $Q_{1} \in |2\zeta - f|$, and $Q_{2} \in |2\zeta - 4f|$.  The latter linear system consists of one non reduced divisor, so $C$ would not be reduced. Therefore, $\cF$ is balanced.

\item[$(2,5)$:] The degree of $\cF$ is $12$, and its rank is $5$.  Furthermore, $\cF$ is a subbundle of $S^2 \cE = S^2 (\cO(1)^{\oplus 2}  \oplus \cO(2)^{\oplus 2})$, which means the largest potential degree of a summand of $\cF$ is $4$.  However, it cannot happen that $\cO(4)$ is a summand of $\cF$: The sections of the linear system $|2\zeta - 4f|$ on $\P \cE$ are {\sl reducible}, and no such reducible section can contain $C \subset \P \cE$ by the non-degeneracy of $C \subset \P \cE$. ($C$ does not lie in a sub scroll.) Therefore, the largest degree of a summand of $\cF$ is $3$.

Now we argue that $\cF$ cannot split as $\cO(1) \oplus \cO(2) \oplus \cO(3)^{\oplus 3}$, which is the only potential non-balanced possibility.  This follows from the fact that the restriction map \[r \from H^0(\cO_{\P \cE}(2\zeta - 3f)) \to H^{0}(\cO_{C}(2\omega_{C}+D)) \] is surjective. Here $D$ is the divisor class of the $g^{1}_{5}$ on $C$.  The surjectivity of restriction map, in turn, follows from the following easy observations:

\begin{enumerate}
\item The restriction maps $$r_{1} \from H^0(\cO_{\P \cE}(\zeta - 2f)) \to H^{0}(\cO_{C}(\omega_{C}))$$ and $$r_{2} \from H^0(\cO_{\P \cE}(\zeta - f)) \to H^{0}(\cO_{C}(\omega_{C} + D))$$ are surjective.
\item The multiplication map $$H^{0}(\cO_{C}(\omega_{C})) \otimes H^{0}(\cO_{C}(\omega_{C} +D)) \to H^{0}(\cO_{C}(2\omega_{C}+D))$$ is surjective.
\end{enumerate}

\item[$(3,4)$:] The analysis is completely similar to the $(2,4)$ case.

\item[$(3,5)$:] The degree of $\cF$ is $14$, and its rank is $5$.  Furthermore, $\cF$ is a subbundle of $S^2 \cE = S^2 (\cO(1)  \oplus \cO(2)^{\oplus 3})$, which means the largest potential degree of a summand of $\cF$ is again $4$.  However, as before, it cannot happen that $\cO(4)$ is a summand of $\cF$: The sections of the linear system $|2\zeta - 4f|$ on $\P \cE$ are sums of products of sections of the linear system $|\zeta - 2f|$.  However, the same is true about the linear system $|2\omega_{C}|$: all sections are sums of products of sections of the canonical series.  This is to say that the restriction map \[r \from H^0(\cO_{\P \cE}(2\zeta - 4f)) \to H^{0}(\cO_{C}(2\omega_{C})\] is an isomorphism, which implies that $\cF$ cannot have $\cO(4)$ as a summand. The degree constraint on $\cF$ then forces it to be balanced.

\item[$(4,4)$:] This analysis is analogous to the $(3,4)$ case, so we skip it.

\item[$(4,5)$:] The degree of $\cF$ is $14$, and its rank is $5$.  Furthermore, $\cF$ is a subbundle of $S^2 \cE = S^2 (\cO(1)  \oplus \cO(2)^{\oplus 3})$, which means the largest potential degree of a summand of $\cF$ is again $4$. We now show that there is a unique $\cO(4)$ summand in $\cF$.  The reader can easily check that this follows from the fact that $C$ lies on a unique quadric in its canonical embedding.
\end{enumerate}
\end{enumerate}
\end{proof}

\begin{proof}[Proof of \autoref{gen:casnati}:]
\autoref{gen:casnati} follows from \autoref{EF01}, by a degeneration argument.  Indeed, suppose we begin with a general cover $\alpha_{1} \from C \to \P^1$ in $\H_{d,g}$ where $d \leq 5$, and suppose we know \autoref{gen:casnati} holds for $\H_{d,g}$. We will show then, that \autoref{gen:casnati} holds for $\H_{d,g+d}$ as well. Pick a general fiber $Z \subset C$ of $\alpha$ and attach an elliptic cover $\beta \from X \to \P^1$ to $C$ along $Z$.  The resulting admissible cover $\alpha \from C \to P$ is such that $h^1(\End \cF_{\alpha}) = 0$ - this follows from the fact that $\cF_{\beta}$ is perfectly balanced, as shown in \autoref{EF01}. Therefore, by smoothing $\alpha \from C \to P$ and appealing to upper semi-continuity, we obtain covers in $\H_{d,g+d}$ which have balanced bundles of quadrics.  

The second part of \autoref{EF01} then  provides the base cases of the induction.

\end{proof}

\begin{rmk}
The reader may wonder, given the existence of explicit classical constructions for all covers of degree $d\leq 5$, why we choose to prove \autoref{gen:casnati} via degeneration rather than constructively.  Certainly it is possible (and simpler) to prove \autoref{gen:casnati} constructively using Bertini-type theorems, however we will need to refer to the method of degeneration in the above proof later (e.g. \autoref{avoidCE}).
\end{rmk}

\subsection{Relating $c_{1}(\cE)$ and $c_{1}(\cF)$.}
We record the following fact conjectured by Casnati in \cite{Casnati-bielliptic}.

\begin{prop}[Conjecture $2.5$ in \cite{Casnati-bielliptic}]\label{relatecFcE}
For any finite, flat map between smooth varieties $\alpha \from X \to Y$, we have 
\[c_1(\cF) = (d-3)c_1(\cE).\]
\end{prop}
\begin{proof}
We establish the slightly stronger identity
\begin{equation}\label{eq_omegaN}
c_{1}(\alpha_{*}(\omega_{\alpha}^{\otimes N})) = (2N-1)c_{1}(\cE).
\end{equation} from which the proposition will follow when $N = 2$.

The identity \ref{eq_omegaN} is proved by a straightforward application of the Grothendieck-Riemann-Roch formula to the map $\alpha \from X \to Y$ and the sheaf $\omega^{\otimes N}_{\alpha}.$  Indeed, if we consider the degree one part of the equality 
\[ch[\alpha_{*}(\omega_{\alpha}^{\otimes N})]\cdot Td(T_{Y}) = \alpha_{*}[ch(\omega_{\alpha}^{\otimes N})\cdot Td(T_X)]\] and use the fact that $2c_{1}(\cE)$ is the class of the branch divisor of $\alpha$, we arrive at the equality 
\[c_1(\alpha_{*}(\omega_{\alpha}^{\otimes N})) - \frac{d}{2}K_{Y} = \alpha_{*}[Nc_{1}(\omega_{\alpha}) - \frac{K_{X}}{2}], \] which, when rearranged, gives \eqref{eq_omegaN} because $\alpha_{*}(\frac{K_{X}}{2}) - \frac{d}{2}K_{Y} = c_{1}(\cE)$.

\autoref{relatecFcE} follows by using \eqref{eq_omegaN} when $N = 2$, and by considering using the exact sequence \[0 \to \cF_{\alpha} \to \Sym^{2} \cE_{\alpha} \to \alpha_{*}(\omega_{\alpha}^{\otimes 2}) \to 0.\]

\end{proof}

\begin{rmk}
With more care, the argument above can be made to show that the first Chern classes of all syzygy bundles $\cN_{i}$ occuring in the Casnati-Ekedahl resolution \ref{eqn:casnati_resolution} are multiples of $c_{1}(\cE)$. With even more care, one can deduce many relationships among the total Chern characters of the bundles $\cN_{i}$.
\end{rmk}

\section{Compactifications.}\label{SecondSection}

\subsection{}
  We will work with a number of spaces and compactifications -- the definitions follow. Note that we do {\sl not} label ramification points in our compactifications of Hurwitz space.

  \begin{description}
  \item [$\M_{g}$] This is the moduli space parametrizing smooth, proper, genus $g$ curves.
  \\
  \item [$\Mg$] This is the Deligne-Mumford compactification of $\M_{g}$
  \\
  \item [$\H_{d,g}$] This is the stack parametrizing $[\alpha \from C \to \P^1]$, where $C$ is a smooth curve of genus $g$ and $\alpha$ a finite map of degree $d$ with simple branching (that is, the branch divisor of $\alpha$ is supported at $2g+2d-2$ distinct points). Two such covers $[\alpha_1 \from C_1 \to \P^1]$ and $[\alpha_2 \from C_2 \to \P^1]$ are considered isomorphic if there are isomorphisms $\phi \from C_1 \to C_2$ and $\psi \from \P^1 \to \P^1$ such that $\alpha_2 \circ \phi = \psi \circ \alpha_1$.
  \\

  \item [$\td \H_{d,g}$] This is the stack parametrizing $[\alpha \from C \to \P^1]$, where $C$ is an at-worst-nodal curve of arithmetic genus $g$, and $\alpha$ a finite map of degree $d$. The notion of equivalence is the same as that for $\H_{d,g}$. This space allows arbitrary branching behavior.
  \\
  \item [$\td \H_{d,g}^{\ns}$] This is the stack parametrizing $[\alpha \from C \to \P^1]$, where $C$ is an at-worst-nodal curve of arithmetic genus $g$ {\sl with only non-separating nodes}, and $\alpha$ a finite map of degree $d$. The notion of equivalence is the same as that for $\H_{d,g}$. This space allows arbitrary branching behavior.
  \\

  \item [$\overline{ \H}_{d,g}^{\rm{adm}}$] This is the Harris-Mumford admissible cover compactification of $\H_{d,g}$ -- see \cite{Harris82:_Kodair_Dimen_Of_Modul_Space_Of_Curves}. 
  \\

  \item[$\Hdgn$] This is the compactification of $\H_{d,g}$ by twisted stable maps, as defined by Abramovich, Corti, and Vistoli \cite{acv:03}. 
  \\
  \item[$\Hdgnt$] This is the alternate compactification of $\H_{d,g}$ by twisted stable maps where at most $2$ branch points are allowed to collide at a time. See \cite{deopurkar13:_compac_hurwit_spaces}. 
  \\
  \item[$\Hdgnth$] This is the alternate compactification of $\H_{d,g}$ by twisted stable maps where at most $3$ branch points are allowed to collide at a time. See \cite{deopurkar13:_compac_hurwit_spaces}. 
  \\

  \item[$\M^{1}_{d,g}$] This is the locus in $\Mg$ parametrizing $d$-gonal curves, i.e. the image of the forgetful map $F \from \Hdg \to \Mg$.

  \end{description}

\subsection{}
   This section will feature $\overline{ \H}_{d,g}^{\rm{adm}}$, $\Hdg$, $\Hdgnt$ and $\td \H_{d,g}$. The goal will be to prove the independence of the components of the boundary of $\Hdg$ (and consequently $\Hdgnt$). This will be done by a delicate use of test curves. 
  \begin{defn}
  A {\sl boundary divisor} is an irreducible component of $\Hdg \setminus \H_{d,g}$ or $\overline{ \H}_{d,g}^{\rm{adm}} \setminus \H_{d,g}$. 
  \end{defn}

\subsection{} 
  In practice, the difficulty in working with the boundary divisors of $\Hdg$ stems from three sources: $(1)$ There is a very large number of boundary divisors; $(2)$ for degrees $d > 5$, the construction of useful one-parameter families becomes nontrivial; and $(3)$ it is often difficult to maintain control over the intersection of test families with boundary divisors.  We overcome difficulties $(2)$ and $(3)$ by constructing  ``ramification reducing'' and ``hyperelliptic'' partial pencil families. 

    Before we continue, a technical point:  $\overline{\H}_{d,g}^{\rm{adm}}$ is not normal.  The normalization, which we denote simply by $\overline{\H}_{d,g},$ is  the smooth Deligne-Mumford stack of twisted stable maps \cite{acv:03}.  We will ultimately be concerned primarily with the normalization. 

\subsection{}
  Recall that the branch morphism \[\Br \colon \overline{\H}^{\rm{adm}}_{d,g} \to [\overline{\M}_{0,b}/\mathfrak{S}_{b}]\] is finite.  This means that every boundary divisor $\Delta \subset \overline{\H}^{\rm{adm}}_{d,g}$ will lie over a unique boundary divisor, $\Br(\Delta) \subset  [\overline{\M}_{0,b}/\mathfrak{S}_{b}]$.  The generic admissible cover $[\alpha \colon C \longrightarrow P]$ parametrized by $\Delta$ will map to the union of two marked curves: $B_L \subset {\P^1_{L}}$ and $B_R \subset {\bf P}^{1}_R$.   The curve $P$ is $\P^1_{L} $ glued to ${\bf P}^{1}_R$ at a point $x$ which is not any of the $b$ points $B_{L} \cup B_{R}$.   The domain curve $C$ breaks up into two halves $C_L$ and $C_R$ which are the preimages of $\P^1_{L}$ and ${\bf P}^{1}_R$.  Since $C$ is a nodal curve, it has a dual graph $\Gamma_{C}$ whose vertices $v$ are marked by the geometric genera $g_v$ of the corresponding components. We furthermore label every vertex  as either an $L$-vertex or an $R$-vertex depending on whether it parametrizes a component of $C_L$ or $C_R$ respectively. Furthermore, we label every vertex $v$ with its {\sl degree} $d_v$.  Since the nodes of $C$ must lie over the node $x$, we see that every edge in $\Gamma_{C}$ must join an $L$-vertex with an $R$-vertex. Finally, we label every edge $e$ with the local degree $d_e$  occurring at the corresponding node.

\subsection{} 
  We may arrange the vertices of the graph $\Gamma_C$ of a generic cover $[\alpha \colon C \longrightarrow P]$ parametrized by a boundary divisor $\Delta$ into two columns, the {\sl left}, and the {\sl right} side.  In this way, we associate to every boundary divisor $\Delta$ a decorated dual graph, which we call $\Gamma_{\Delta}$. We can now introduce some useful numerical quantities associated to a boundary divisor.

  \begin{defn}\label{ramificationdegree}
  Let $\Delta$ be a boundary component.  Then the {\sl ramification index} of $\Delta$ is the number $$r(\Delta) := \sum_{\text{all edges $e \in \Gamma_{\Delta}$}}(d_{e}-1).$$
  \end{defn}

  We say that a boundary divisor $\Delta$ is {\sl unramified} if $r(\Delta)=0$.  This simple means that the generic admissible cover parameterized by $\Delta$ does not have branching at any nodes.

  \begin{defn}\label{excessdefn}

  \end{defn}

\subsection{}
  Labeling every vertex and edge of $\Gamma_{\Delta}$ will often be notationally cumbersome and unnecessary.  Therefore, we will adopt the following convention: Genus $0$ vertices and unramified edges will usually be left undecorated. Furthermore, we note that since the degree of a vertex $v$ is determined by all ramification indices of edges incident to $v$, we need not specify {\sl both} the degrees of vertices and local degrees of edges -- the knowledge of one set of data determines the other.

\subsection{$\Hdg$ vs. $\o{\H}^{\rm{adm}}_{d,g}$. } 
  The boundary divisors in the space of twisted stable maps  $\Hdg$ lie above the boundary divisors of  $\overline{\H}_{d,g}^{\rm{adm}}$.  Indeed, there is a representable forgetful map \[\phi \from \Hdg \to \o{\H}^{\rm{adm}}_{d,g}\] which is finite and generically an isomorphism.

  We will not need a precise method for labeling the boundary divisors in $\overline{\H}_{d,g}$, but we will need the following simple lemma which will, in practice, allow us to ignore the distinction between boundary divisors in $\overline{\H}_{d,g}^{\rm{adm}}$ and boundary divisors in $\Hdg$:

  \begin{lem}\label{liftingtotwisted}
  Fix a boundary divisor $\Delta \subset \Hdg^{\rm{adm}}$ and a boundary divisor $\Delta' \subset  \Hdg$ lying over $\Delta$. Given a morphism $f \from S \to \Delta \subset \Hdg^{\rm{adm}}$ from $S$ a smooth complete curve, there is a finite cover $S' \to B$ which lifts the morphism $f$ to a morphism $f' \from S' \to \Delta'$.  
  \end{lem}

  \begin{proof}

  \end{proof}

  \autoref{liftingtotwisted} allows us to lift (after finite base change) test families in  $\overline{\H}_{d,g}^{\rm{adm}}$ to $\overline{\H}_{d,g}$ -- thus we will often construct a one parameter family in $\overline{\H}_{d,g}^{\rm{adm}}$ yet talk about a family in $\Hdg$ obtained possibly after a base change. 

\subsection{Some examples of boundary divisors.}

  We provide some examples for the reader's convenience.  These examples are not selected randomly: they all lie above $\Delta_2 \subset  [\overline{\M}_{0,b}/\mathfrak{S}_{b}]$, i.e. two branch points are colliding.  In enumerative settings, the  divisors lying over $\Delta_{2} \subset   [\overline{\M}_{0,b}/\mathfrak{S}_{b}]$ show up most frequently, so we refer to them as the {\sl enumeratively relevant}  boundary divisors. For the reader's convenience, we indicate the dual graphs of four of the most basic enumeratively relevant divisors.

  \begin{enumerate}
  \item $\Delta_{\rm irr}$, admissible covers having dual graph $\Gamma_{\Delta_{\rm irr}}$:  

  \[  \xygraph{
  !{<0cm,0cm>;<1cm,0cm>:<0cm,1cm>::}
  !{(0,-.5)}*{\bullet}="l"!{(-.5,-.3)}{\scriptscriptstyle (d,g-1)}
  !{(2,0) }*+{\bullet}="n"!{(2.2,.2) }{\scriptstyle w}
  !{(2,-.5) }*+{\bullet}="a_1"
  !{(2,-1.2) }*+{\bullet}="a_{d-2}"
  !{(0,-1.6)}{\scriptstyle ``{\rm left\,\, side}"}
  !{(2,-1.6)}{\scriptstyle ``{\rm right\,\, side}"}
   "l"-@/^.3cm/"n" 
  "l"-@/^.5cm/"n"
  "l"-@/^.1cm/"a_1"
  "l"-@/^.1cm/"a_{d-2}"^(.95){\scriptscriptstyle \vdots}
  }  \]
  \item $\sfT_{\rm adm}$, ``triple ramification''.  These are covers with dual graph $\Gamma_{T_{\rm adm}} $:

  \[ \xygraph{
  !{<0cm,0cm>;<1cm,0cm>:<0cm,1cm>::}
  !{(0,-.5)}*{\bullet}="l"!{(-.5,-.3)}{\scriptscriptstyle (d,g)}
  !{(2,0) }*+{\bullet}="t"
  !{(2,-.5) }*+{\bullet}="a_1"
  !{(2,-1.2) }*+{\bullet}="a_{d-2}"
   "l"-@/^.3cm/"t"^{\scriptscriptstyle 3} 
  "l"-@/^.1cm/"a_1"
  "l"-@/^.1cm/"a_{d-2}"^(.95){\vdots}
  }  \]

  \item $\sfD_{\rm adm}$, ``$(2,2)$ ramification''. These covers have dual graph $\Gamma_{D_{\rm adm}}$:

  \[  \xygraph{
  !{<0cm,0cm>;<1cm,0cm>:<0cm,1cm>::}
  !{(0,-.5)}*{\bullet}="l"!{(-.5,-.3)}{\scriptscriptstyle (d,g)}
  !{(2,0) }*+{\bullet}="t_1"
  !{(2,.25) }*+{\bullet}="t_2"
  !{(2,-.5) }*+{\bullet}="a_1"
  !{(2,-1.2) }*+{\bullet}="a_{d-2}"
   "l"-@/^.3cm/"t_1"_(.8){\scriptscriptstyle 2} 
    "l"-@/^.4cm/"t_2"^(.8){\scriptscriptstyle 2} 
  "l"-@/^.1cm/"a_1"
  "l"-@/^.1cm/"a_{d-2}"^(.95){\vdots}
  }  \]
  \item ${\sfB}_{\rm adm}$, ``basepoint''. These covers have dual graph $\Gamma_{{\sfB}_{\rm adm}} $: 

  \[  \xygraph{
  !{<0cm,0cm>;<1cm,0cm>:<0cm,1cm>::}
  !{(0,-.6)}*{\bullet}="l"!{(-.5,-.4)}{\scriptscriptstyle (d-1,g)}
  !{(0,0)}*{\bullet}="l'"
  !{(2,0) }*+{\bullet}="b"
  !{(2,-.5) }*+{\bullet}="a_1"
  !{(2,-1.2) }*+{\bullet}="a_{d-2}"
   "l'"-@/^.3cm/"b" 
   "l"-@/^.3cm/"b" 
  "l"-@/^.1cm/"a_1"
  "l"-@/^.1cm/"a_{d-2}"^(.95){\vdots}
  }  \]
  \end{enumerate}

  We will briefly explain the interpretation of $\Gamma_{\Delta_{\rm irr}}$ for the reader's convenience. The vertex $w \in \Gamma_{\Delta_{\rm irr}}$ is unlabeled, so implicitly it has genus $0$.  Furthermore, both edges emanating from $w$ are unramified, therefore the degree $d_w$ of $w$ is $2$.  So the rational curve associated to $w$ is attached to the curve $C_v$ at two points.  The stable model of the domain curve is therefore an irreducible nodal curve, hence the label ``$\Delta_{\rm irr}$''.   

\subsubsection{}\label{higherboundarydivisor}
  In practice, the boundary divisors $\Delta_{\rm irr}$, $\sfD_{\rm adm}$, and $\sfT_{\rm adm}$  occur most frequently in one parameter families.  We will often suppress the subscript and refer to the latter two as $\sfD$ and $\sfT$.  For simplicity, we adopt the following convention: {\sl The symbols $\sfD$ and $\sfT$ will denote the divisors parameterizing covers having $(2,2)$-ramification and triple ramification, respectively, regardless of the compactification we work in.}

  Because of the scarcity of interaction with other boundary divisors, we refer to any boundary divisor which is not any of the above three as a {\sl higher boundary divisor}.

\subsection{Partial pencil families.} 
  This section introduces the class of one parameter families central to the rest of the paper.  We call these families {\sl partial pencil} families, and they lie entirely within the boundary of $\Hdg$ (or $\Hdg^{\rm{adm}}$).   

  \begin{rmk}
  To aid the reader, we mention that our one parameter families will always begin as families in $\Hdg^{\rm{adm}}$, and then if necessary will be lifted to $\Hdg$. This is the main purpose of \autoref{liftingtotwisted}.
  \end{rmk} 
   
   First we give an informal description of a partial pencil. Consider a boundary component $\Delta \subset \Hdg^{\rm{adm}}$.  The general admissible cover parametrized by $\Delta$ can be written as \[\alpha \from C_{L} \cup C_{R} \to P_{L} \cup P_{R}\] where $\alpha_{L} \from C_{L} \to P_{L}$ and $\alpha_{R} \from C_{R} \to P_{R}$ are the left and right sides of the admissible cover.  The essential idea behind the partial pencil families is to vary one of the maps, which we typically take to be the right hand side $\alpha_{R}$, while keeping the left map $\alpha_{L}$ (including the points of attachment) fixed in moduli.  
   
  \begin{rmk}\label{analogy}
   The reader may find it useful to think of partial pencils as the $\Hdg$ or $\Hdg^{\rm{adm}}$ analogues of the standard one parameter families of curves found in the boundary of $\Mg$.  One major difference is that in the case of $\Mg$ there exist, conveniently, complete one parameter families lying entirely within a single boundary component while avoiding all others.  Unfortunately, partial pencil families typically intersect other boundary components; the management of their intersection with other boundary components is the ultimate technical obstacle in this paper. 
   \end{rmk}
   
     \subsubsection{} \label{markedadmissible}
     Now we give a formal description. Fix a partition $\overline{r} = r_{1} + r_{2} + ... + r_{k} = d$ by positive integers $r_{i}$. Suppose 
      \begin{equation}\label{right}
    \xymatrix @C -1.1pc  {
    \cC_{R}\ar[rr]^{\alpha_{R}} \ar[rd]_{f_{R}}&&
    \cP_{R} \ar[ld]^{p_{R}}\\
    & S}
     \end{equation} 
     is a one parameter family of {\sl $\overline{r}$-marked} admissible covers parameterized by a smooth, proper curve $S$.  This means 
    \begin{enumerate}
    \item $\alpha_{R}$ is a finite map from the nodal family of genus $g_{R}$ curves $\cC_{R}$ to the nodal family of genus $0$ curves $\cP_{R}$.

     \item  There is a marked section ${\bf 0} \subset \cP_{R}$, disjoint from the singular locus of the map $p_{R}$, such that \[\alpha^{-1}({\bf 0}) = r_{1}R_{1} \sqcup r_{2}R_{2} \sqcup ... \sqcup r_{k}R_{k}\] where $R_{i} \from S \to \cC_{R}$ are sections disjoint from ${\rm Sing}(f_{R})$. In other words, the family \eqref{right} is ramified over ${\bf 0}$ according to the partition given by $\overline{r}$.

     \item Away from ${\bf 0}$ and the singular locus ${\rm Sing}(p_{R})$, the map $\alpha_{R}$ is simply-branched. 
     
     \item Above ${\rm Sing}(p_{R})$, $\alpha_{R}$ must satisfy the usual admissible cover ``kissing'' condition.
     \end{enumerate}
     (Note: An $\overline{r}$-marked admissible cover is very similar to an admissible cover - we simply allow for higher ramification, given by the profile $\overline{r}$, above a specified marked point ${\bf 0}$ in the target.)
   
   \subsubsection{}
     Next, fix a single $\overline{r}$-marked cover $\alpha_{L} \from C_{L} \to (P_{L}, {\bf 0})$. Let $r_{1}T_{1}+ r_{2}T_{2}+... +r_{k}T_{k}$ be the preimage of ${\bf 0} \in P_{L}$, and let $\cC_{L} := C_{L} \times S$, and $\cP_{L} := P_{L} \times S$.  We let $L_{i} = T_{i} \times S \subset \cC_{L}$ denote the section along which $\alpha_{L}$ is ramified to order $r_{i}$. 

  \subsubsection{}   
    Finally, let $\cC$ be the surface obtained by identifying $L_{i}$ with $R_{i}$ for $i = 1, ... ,k$. Similarly, let $\cP$ be the surface obtained by glueing ${\bf 0} \in \cP_{L}$ with ${\bf 0} \in \cP_{R}$.  In this way, we construct a family of admissible covers \[\alpha \from \cC \to \cP \to S.\]
    \begin{defn}[Partial pencil]\label{ppencil}
    A {\sl partial pencil} is any one-parameter family of admissible covers $\alpha \from \cC \to \cP \to S$ as constructed above. We call $\alpha_R \from \cC_{R} \to \cP_{R}$ (resp. $\alpha_{L}$) the {\sl right side} (resp. {\sl left side}) of the partial pencil $\alpha$. We call the $R_{i} \subset \cC_{R}$ (resp. $ L_{i} \subset \cC_{L}$) the {\sl right gluing sections} (resp. {\sl left gluing sections}).
    \end{defn}

    Notice that the left side of a partial pencil is fixed in moduli. (Including the points of attachment.) 
  \subsubsection{}  

     Suppose $S \to \Hdg^{\rm{adm}}$ is a partial pencil family, lying entirely within a boundary divisor $\Delta$.  By \autoref{liftingtotwisted}, we may lift this family to a family $S' \to \Delta'$ for any $\Delta' \subset \Hdgn$ lying over $\Delta \subset \Hdg^{\rm{adm}}$. Such families $S' \to \Hdgn$ will be called {\sl twisted partial pencils}.  In particular, the coarse space of a twisted partial pencil is a partial pencil.

\subsection{Intersection multiplicities.}\label{intersectionmultiplicities}

  We now indicate how to compute intersection multiplicities of a one parameter family lying entirely in the boundary with boundary divisors.  For this, we review the description of the deformation spaces of admissible covers found in \cite{Harris82:_Kodair_Dimen_Of_Modul_Space_Of_Curves}. 

  \subsubsection{}
    Pick any boundary divisor $\Delta \subset \Hdg^{\rm{adm}}$, and choose a general admissible cover $[\alpha \colon C \longrightarrow P] \in \Delta$.  Let $p \in P$ be the unique node, and let $(r_1,r_2,...,r_k)$ be the local degrees of $\alpha$ occurring at the nodes $(q_1,...,q_k)$ above $p$.  A description of the complete local ring $\compl{\cO}_{\Hdg^{\rm{adm}}, \alpha}$ is given by Harris and Mumford in \cite{Harris82:_Kodair_Dimen_Of_Modul_Space_Of_Curves} as: 
    \[\compl{\cO}_{\Hdg^{\rm{adm}}, \alpha} = \Spec k[[z_{1}, ... , z_{b-4}, t_1,t_2, ... , t_k, s]]/ (t_1^{r_1} = t_2^{r_2} = ... = t_k^{r_k} = s).\] This presentation is such that  \[\compl{\cO}_{[\o{\M}_{0,b}/\mathfrak{S}_{b}], \Br(\alpha)} = k[[z_{1}, ... ,z_{b-4}, s]]\] and the obvious map $\compl{\cO}_{[\o{\M}_{0,b}/\mathfrak{S}_{b}], \Br(\alpha)} \to \compl{\cO}_{\Hdg^{\rm{adm}}, \alpha}$ gives the complete local description of the branch morphism \[\Br^{\rm{adm}} \from \Hdg^{\rm{adm}} \to [\overline{\M}_{0,b}/\mathfrak{S}_{b}].\]

  \subsubsection{} \label{localcoordinates} 
    Above the veresal deformation space ${\bf Def}\, \alpha := \Spf{\compl{\cO}_{\Hdg^{\rm{adm}}, \alpha}}$, there is a versal family: 
     \begin{equation*}   
    \xymatrix @C -1.5pc  {
    C \subset \mathcal{\cC} \ar[rr]^{\compl{\alpha}} \ar[rd]_\varphi&&
    {P} \subset \mathcal{\cP} \ar[ld]^\pi \\
    &{\bf Def}\, \alpha}
     \end{equation*} 
     Locally around the node $p \in {P} \subset \cP$, $\cP$ has local equation $uv=s$.  At the nodes $q_{i} \in C \subset \cC$, $i = 1,2,...,k$, the total space $\cC$ has local equation $x_i y_i = t_i$.  In these local coordinates, the admissible cover $\compl{\alpha}$ is locally given by $u \mapsto x_{i}^{r_{i}}$ and $v \mapsto y_{i}^{r_{i}}$. Furthermore, the equations $s = t_{1} =  ... = t_{k} = 0$ cuts out the divisor $\Delta$.

  \subsubsection{} 
     This local description of the boundary allows us to understand intersection multiplicities of partial pencil families with boundary components. So let 
     \[\alpha \from \cC \to \cP \to S\] be a general family contained in $\Delta$. Furthermore, let $R_{i} = L_{i}$, $i=1,...,k$ be the glued sections mapping to a marked section of nodes ${\bf 0} \subset \cP$, where, around the section $R_{i} = L_{i}$ the map $\alpha$ has local degree $r_{i}$ over ${\bf 0}$.  
     
      The admissibility condition at the nodes tells us that \[\deg (N_{L_{i}/\cC_{L}}\otimes N_{R_{i}/\cC_{R}})^{\otimes r_i} = \deg (N_{L_{j}/\cC_{L}}\otimes N_{R_{j}/\cC_{R}})^{\otimes r_j} = \deg (N_{{\bf 0}/\cP_{L}}\otimes N_{{\bf 0}/\cP_{R}}).\] Here the $N_{\bullet}$ denote normal bundles. 
   
  \subsubsection{}
    Suppose furthermore that $S \to \Hdg^{\rm{adm}}$ comes from a family $S \to  \Hdgn$, i.e. that $S \to \Hdg^{\rm{adm}}$ is the coarse space of a {\sl twisted} partial pencil. The family $S \to \Hdgn$ will have intersection number \[S \cdot \Delta' = {\rm deg}\, (N_{{\bf 0}/\cP_{L}}\otimes N_{{\bf 0}/\cP_{R}})/lcm(r_1,...,r_k)\] with the boundary divisor $\Delta' \subset \Hdg$ in which it lies.  (We assume that $S \to \Hdgn$ does not have additional isolated intersections with $\Delta'$, in which case we must add these contributions appropriately.) This follows from the observation that the order of ramification of the branch map from the normalization \[\Br^{\nu} \from ({\bf Def}\, \alpha)^{\nu} \to \Spf{\compl{\cO}_{[\o{\M}_{0,b}/\mathfrak{S}_{b}], \Br(\alpha)}}\] is $lcm(r_{1},...,r_{k})$. 

\subsection{Admissible reduction.}\label{admissiblereduction}
  At this point, it will be convenient to indicate how to perform ``admissible reduction'' in some commonly encountered situations. 

  \subsubsection{Setup.}
    Suppose we have a family of degree $d$  simply-branched covers $\alpha_{t} \from \cC_{t} \to \P^{1}$, with smooth total space $\cC_{t}$, parametrized by $\Delta_{t} := \Spec k[[t]]$, and $m$ disjoint multi-sections $\sigma_{1}, ... \sigma_{m}$ of degrees $d_{1}, ... , d_{m}$, each of which in some local coordinates $(z,t)$ on $\cC_{t}$ is given by an equation of the form $z^{d_{i}} - t = 0$.  Furthermore, suppose $\cup_{i} \sigma_{i} = \alpha_{t}^{-1}(0)$. In particular, note that $d_{1} + ... + d_{m} = d$.

   Our task is to attach a constant cover $\alpha_{L} \from C_{L} \to \P^{1}$ to the family $\alpha_{t}$ at the $d$ points of the multi-section $\cup_{i}\sigma_{i}$.  In other words, we would like $\alpha_{t}$ to be the right side of a partial pencil.

    There are two obstacles currently preventing us from accomplishing this: 1) The points of attachment for $t \neq 0$ are experiencing nontrivial monodromy, and 2) the points of attachment are not remaining distinct when $t=0$.

  \subsubsection{Base change.} 
    In order to kill monodromy, we first make an order $l := {\rm lcm}(d_{1}, ... ,d_{m})$ base change $t = s^{l}$, and let $\alpha_{s} \from \cC_{s} \to \P^{1} \times \Delta_{s}$ denote the base change family.  Letting $l_{i} := l/d_{i}$, we immediately see that we now have $d$ sections $\tau_{1}, ... ,\tau_{d}$ of the family $\cC_{s}$.  Moreover, these sections are naturally grouped according to which multi-section $\sigma_{i}$ they originally arose from. The first $d_{1}$ sections intersect at a point $x_{1} \in C_{0}$, the next $d_{2}$ intersect at a point $x_{2} \in C_{0}$, and so on.

  \subsubsection{Separating the sections $\tau_{i}$.}\label{separatingsections}
    Now we must separate the sections $\tau_{i}$ -- this requires some care, since we still need to produce an admissible cover.  We examine the local situation around each point $x_{i}$. 

    For this, let $X := \Spec k[[z,s]]$, $Y := \Spec k[[w,s]]$, $\Delta_{s} := \Spec k[[s]]$, and let $$\alpha_{\text{local}}: X \to Y$$ be the degree $d_{i}$ $\Delta_{s}$-morphism defined by $w = z^{d_{i}} - s^{l}$. This is a local model for the map $\alpha_{s}$ near the points $x_{i}$.

    We now blow up $X$ along the ideal $I := (z,s^{l_{i}})$ and $Y$ along the ideal $J := (w, s^{l})$. In other words, we set $$X' := \Proj k[[z,s]][U,V]/Uz-Vs^{l_{i}}$$ and $$Y' := \Proj k[[w,s]][A,B]/Aw-Bs^{l}.$$ Then the following map of graded rings provides the ``local admissible replacement'' $\alpha'_{\text{local}}$: 
    \begin{eqnarray} \nonumber
    &s \mapsto s\\ \nonumber
    &w \mapsto z^{d_{i}}-s^{l}\\ \nonumber
    &A \mapsto U^{d_{i}}\\ \nonumber
    &B \mapsto V^{d_{i}} - U^{d_{i}} \nonumber
    \end{eqnarray}

    The reader may easily check that the local ``kissing'' condition is met for $\alpha'_{\text{local}}$.  Furthermore, the local sections defined by $z - \zeta_{j}s^{l_{i}}$, where $\zeta_{j}$ are the ${l_{i}}$-th roots of unity, are separated in the blow up $X'$. 

    \begin{rmk}
    Notice that the ideal $I$ is not the inverse image of the ideal $J$.  In fact, an admissible cover with ramification occurring above a node is never a flat map.  Therefore, the ideal $I$ could not have been the inverse image of $J$. 
    \end{rmk}

  \subsubsection{Triple ramification.} 
    Now suppose $\alpha_{t}$ is a family of covers aquiring a triple ramification point $x \in C_{0} \subset \cC$. Then the reader may check that the admissible reduction of this family requires at least an order two base change.  

     \begin{rmk}
    Triple ramification is a divisorial phenomenon. Since we will be using many one parameter families, it will be a nuisance to have to perform admissible reduction each time it occurs.  This is why we eventually pass to the space $\Hdgnt$.
    \end{rmk}

\subsection{Constructions of partial pencils.} \label{construction} 
  We will frequently use one particular construction of partial pencil families. The basic idea will be introduced in this section.
 
  \subsubsection{Basic Setup.}  
    Suppose $\pi \from \F \to \P^1_{s}$ is a smooth surface fibered over $\P^1$ with $\pi$ the projection.  Let $C$ be a smooth curve on $\F$, and denote the degree of $\pi \from C \to \P^1$ by $d$.  We let $g$ denote the genus of $C$. 
 
    Next, we mark a general point ${\bf 0} \in \P^1_{s}$, label the corresponding fiber $F \subset \F$, and mark $d$ distinct points $p_{1}, ... ,p_{d} \subset F$.   We choose a pencil \[\P^1_{t} \subset |\cO_{\F}(C)|\] in the linear system of $C$ such that  the base locus $B$ of $\P^{1}_{t}$ contains the set $\{p_{1},...,p_{d}\}$.  

    Let $p \from X \to \P^1_{t}$ denote the total space of the pencil, i.e. $X = Bl_{B}\F$ is the blow up of $\F$ along $B$.  Since the pencil is forced to be constant on the marked fiber $F$, we see that at some point, which we take as $t = \infty \in \P^1_{t}$, the curve $p^{-1}(\infty) = C' \cup F$ becomes reducible.

    We make the following assumption, which will be true in all cases we encounter: {\sl We assume that $C'$ is at-worst-nodal and meets $F$ trasversely at smooth points.} 

    Our task is to realize the pencil $p$ as the right side of a partial pencil family.

  \subsubsection{}

    The induced map $\alpha_{R} \from X \to \P^{1}_{s} \times \P^{1}_{t}$ is finite everywhere except over the point $(0,\infty) \in \P^{1}_{s} \times \P^{1}_{t}$, where the preimage is  the curve $F$. If we let $\td {\P^1_{s} \times \P^{1}_{t}}$ denote the blow up at $(0, \infty)$, we see that $\alpha_{R}$ factors through $\td {\P^1_{s} \times \P^{1}_{t}}$, and $$\alpha_{R} \from X \to \td {\P^{1}_{s} \times \P^{1}_{t}}$$ is finite and flat. 

    The section $(0,t) \from \P^1_{t} \to  \td {\P^{1}_{s} \times \P^{1}_{t}}$ together with the sections of $p \from X \to \P^1_{t}$  induced by the basepoints $p_{i}$ provide the data of a family of marked admissible covers, as in section \ref{markedadmissible}.  The family $\alpha_{R}$ can now be attached (perhaps after a base change) to a constant family of covers on a constant left hand side $\alpha_{L} \from C_{L} \to \P^{1}_{R}$. This produces a partial pencil family, which we will refer to simply by $S$. 

  \subsubsection{}

    The family $S$ clearly lies entirely within a boundary divisor $\Delta$.  The admissible covers in $S$ originally associated to the point $t=\infty$ before base change, provide intersections with another boundary divisor, which we refer to by $\Delta_{\rm split}$. (Indeed, the target of these admissible covers splits off a third component.)

    \begin{defn}[Split Fiber]\label{split fiber} Maintain the notation and setting above. The admissible covers in the partial pencil family $S$ corresponding to the point $t=\infty$ are called {\sl split fibers}.  
    \end{defn} (This language is used to remind the reader that the component $F$ is ``splitting off'', creating an intersection with a new higher boundary divisor.)
    \begin{figure}\label{SplitFiber1}
     \centering
      \includegraphics[width=0.5\textwidth]{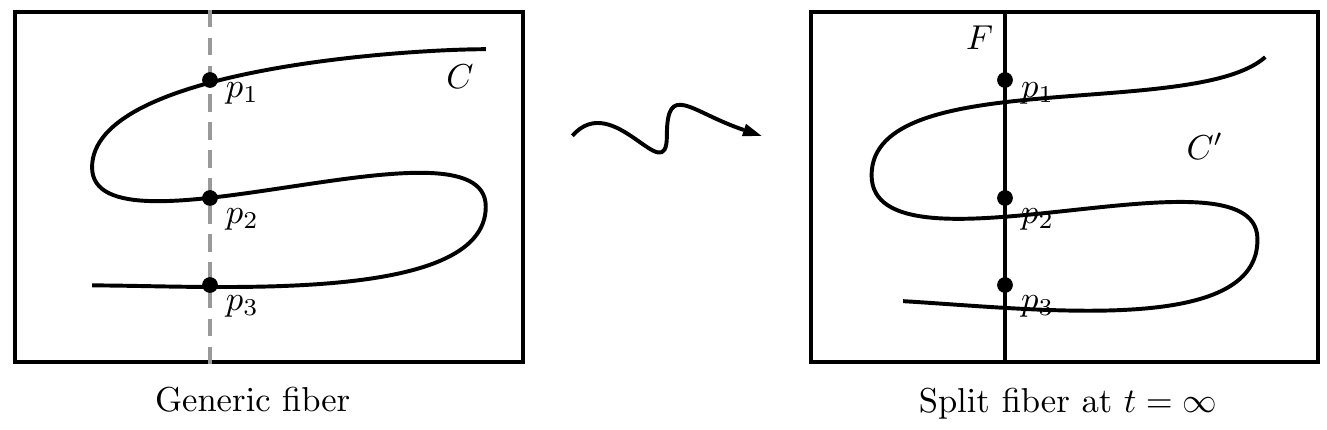}
     \caption{Acquiring a split fiber.}
     \end{figure}

  \subsubsection{}\label{genericpencil}   

    We will be interested in pencils satisfying the following genericity conditions: 
    \begin{enumerate}
    \item The curves parametrized by $\P^{1}_{t}$ are at-worst-nodal.
    \item The split fiber $F$ is simply branched under the map $\alpha_{R} \from X \to \td {\P^{1}_{s} \times \P^{1}_{t}}$.
    \item The  residual curve $C'$ in the split fiber is at-worst-nodal and intersects $F$ at $d$ distinct points.
    \end{enumerate}

    These conditions have to be checked in every situation.  However, this is usually a standard application of Bertini's theorem and dimension counts, so we usually omit this check.

    \begin{rmk} We should mention that there are essentially two ways of producing one parameter families of branched covers.  The first and conceptually simpler way is by ``varying the branch points''.  Effectively, this means to start with a curve $S \subset [\overline{\M}_{0,b}/\mathfrak{S}_{b}]$, and then lift it to Hurwitz space via the finite branch morphism.  The task of determining how many times the family meets various boundary components involves very difficult combinatorics, which we choose to avoid at all costs. See \cite{harris-monodromy}.  

    The other approach is to vary covers as divisors on fibered surfaces or as complete intersections in higher dimensional varieties.  This is the approach we have taken. It is much easier to understand how these families intersect various boundary components of Hurwitz space. The drawback is that the constructions are ad hoc, and do not ``cover generic moduli'' once the degree $d$ is large.  
    \end{rmk}
 
  \subsubsection{Variation on the basic construction I.}\label{modified} 

    We now modify the construction above by imposing nonreduced base points in the pencil $\P^{1}_{t} \subset |\cO_{\F}(C)|$. Maintaining the notation and setting above, let $\{p_{1}, p_{2}, ... ,p_{k}\}$  be a set of distinct points in $F$, with $k < d$.  Consider the closed subscheme $Z \subset F \subset \F$ isomorphic to $\Spec k[\epsilon]/(\epsilon^{m})$ supported on the point $p_1$, where $m = d-k-1$. 
 
    As before, pick a pencil $\P^{1}_{t} \subset |\cO_{\F}(C)|$ containing the scheme $Z \cup p_{2} \cup ... \cup p_{k}$ in its base locus $B$; choose the pencil to be general among such pencils.  We assume that the pencil $\P^1_{t}$ generically parameterizes smooth curves having order $m$ ramification at the point $p_{1}$. The point $t = \infty$ will denote the split fiber, where the component $F$ splits off.

  \subsubsection{}    
     There is one more element of the pencil $\P^{1}_{t}$ which is very important for our purposes. The total space $X = Bl_{B} \F$ has a unique $A_{m-1}$ singularity $s$ coming from blowing up the nonreduced scheme $Z$. As before, let $p \from X \to \P^{1}_{t}$ denote the total family of the pencil, and assume that the singular point $s \in X$ occurs over the point $0 \in \P^{1}_{t}$. Then the curve $C_{0} := p^{-1}(0)$ is a singular curve in $X$, possessing a node at the point $s$. We assume $C_{0}$ has no singularities outside of the node $s$. 

     The point $p_{1} \in C_{0} \subset \F$ is then a ``ramified node''; one of the branches is tangent to order $m-1$ along the fiber $F$, while the other branch meets $F$ transversely.  Let $s_{1}, ... , s_{k}$ denote the sections of $p \from X \to \P^{1}_{t}$ corresponding to the reduced base points points $p_{1}, ... ,p_{k}$. Observe that $s \in s_{1}$.

  \subsubsection{}   
     Now we produce the admissible reduction of the total family $p$.  Clearly, the singular point $s$ is problematic, as it lies on the section $s_{1}$,  so we desingularize the $A_{m-1}$ singularity $s \in X$, to obtain a new family $\tilde{p} \from \widetilde{X} \to \P^1_{t}$ with fiber $\tilde{p}^{-1}(0) = \widetilde{C}_{0} \cup E_{1} \cup ... \cup E_{m-1}$.  Here, $\widetilde{C}_{0}$ is the normalization of $C_{0}$, and the $E_{i}$ are the exceptional divisors of the standard desingularization of an $A_{m-1}$ singularity. 
     
     The normalization $\widetilde{C}_{0}$ intersects the exceptional curves $E_{1}$ and $E_{m-1}$, and we may assume that the section $\widetilde{s}_{1}$ passes through $E_{1}$ transversely at a point $a \in E_{1}$.  All other sections $s_{i}$, $i \geq 2$ are unaffected by the blow up. 
     
     Next contract all exceptional curves $E_{i}$, $i \geq 2$, introducing an $A_{m-2}$ singular point $b \in E_{1}$. (We continue to use ``$E_{1}$'' for this remaining exceptional divisor.) Finally we blow up the $k-1$ smooth points $s_{i} \cap \widetilde{C}_{0}$, $i \geq 2$.  In summary, we arrive at a family $\bar{p} \from \overline{X} \to \P^{1}_{t}$ satisfying the following properties: 
 
    \begin{enumerate}
    \item $\bar{p}$ has $k$ disjoint sections $\bar{s}_{i}$ all contained in the smooth locus of the map $\bar{p}$.
    \item The fiber $\bar{p}^{-1}(0)$ is $\widetilde{C}_{0} \cup E_{1} \cup P_{2} \cup ... \cup P_{k}$. These components intersect transversally as follows: $\widetilde{C}_{0}$ and $E_{1}$ intersect at two points $x$ and $b$, $\widetilde{C}_{0}$ and $P_{i}$ intersect at one point. All other components are mutually disjoint.
    \item  $E_{1}$ has two special points $a = \bar{s}_{1} \cap E_{1}$ and $b$, the $A_{m-2}$ singularity of $\overline{X}$. The point $a$ is a smooth point of $\overline{X}$. 
    \item There is a distinguished pencil of degree $m$ on the component $E_{1}$ spanned by the divisors $m\cdot a$ and $(m-1)\cdot x + b$.  Furthermore, there are distinguished pencils of degree one on each $P_{i}$ spanned by the points $\bar{s}_{i} \cap P_{i}$ and $\widetilde{C}_{0} \cap P_{i}$.
    \item We have: $\bar{s}_{1}^{2} = -1$.
    \end{enumerate}

  \subsubsection{}  
    After dealing with the split fiber $t=0$ in the usual way, (and after base changes from any necessary admissible reduction) we may attach  a fixed cover $\alpha_{L} \from C_{L} \to \P^{1}_{L}$ with ramification profile $m + 1+ ... +1$ above $0 \in \P^{1}_{L}$ along the sections $(\bar{s}_{1}, ... , \bar{s}_{k})$.  

    Let $S$ denote the resulting partial pencil family. Then $S$ has the following properties which are most important for us: 
    \begin{enumerate}
    \item $S$ is completely contained in a ramified boundary divisor $\Delta$.
    \item $S$ intersects the standard boundary divisors $\Delta_{\rm irr}$, $\sfT$, and $\sfD$ nonnegatively.
    \item $S$ intersects an unramified boundary divisor $\Delta_{\rm split}$. This corresponds to the split fiber $t=0$ in the original pencil $p$.
    \item $S$ intersects one other higher boundary divisor $\Delta_{\rm ram}$, corresponding to the element $t=0$ from the original pencil. The ramification index $r(\Delta_{\rm ram})$ is one less than the ramification index $r(\Delta)$.
    \end{enumerate}

    \begin{defn}[Reduced Ramification fibers] In any partial pencil family $S$ constructed as above, we call the fibers corresponding to $t=0$  {\sl reduced ramification fibers}.
    \end{defn}

    \begin{figure}[h!]\label{SplitFiber2}
     \centering
      \includegraphics[width=0.5\textwidth]{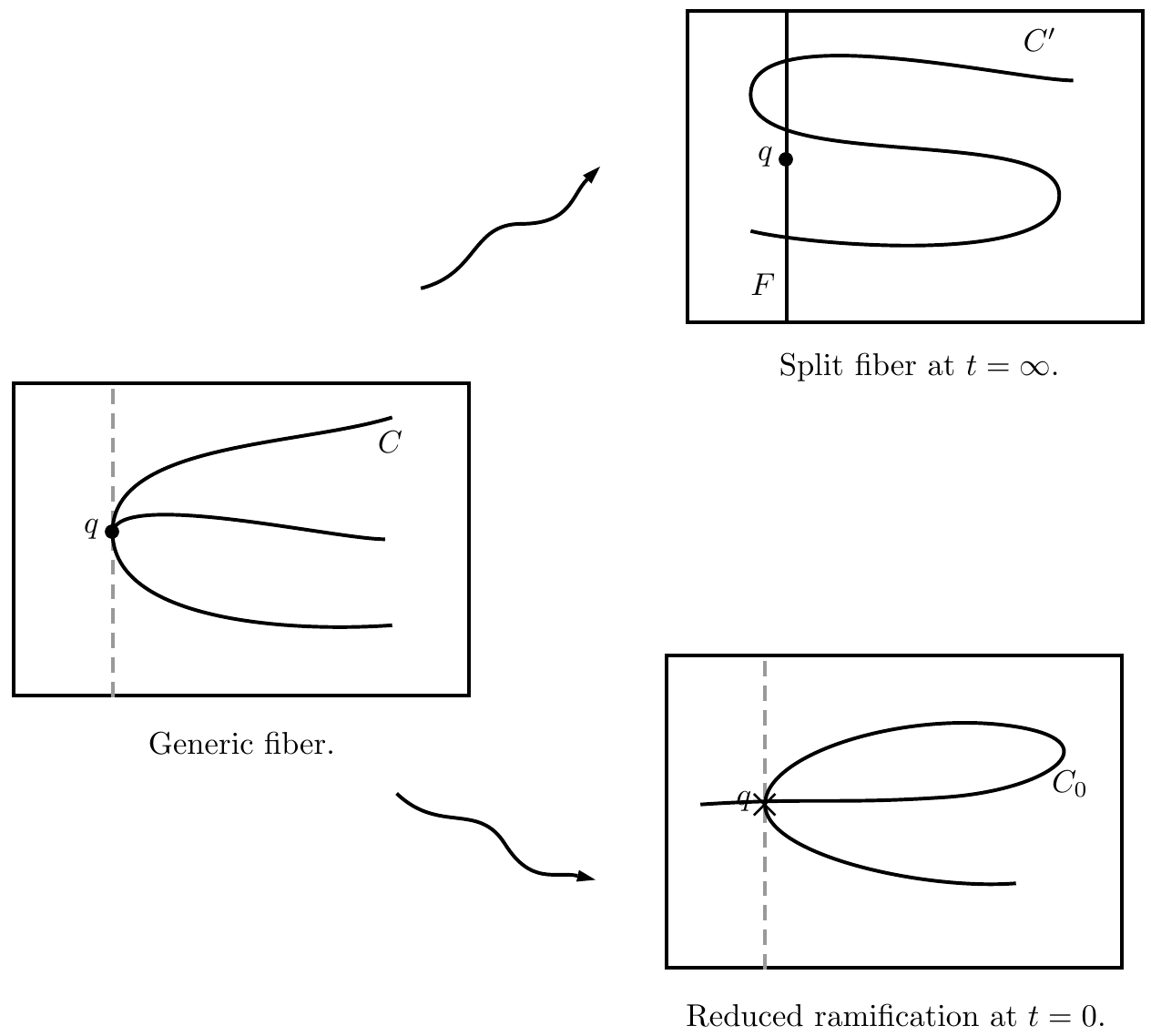}
     \caption{A picture of the initial pencil $p \from X \to \P^{1}_{t}$ acquiring a split fiber at $t= \infty$ and a reduced ramification fiber at $t = 0$. The ``$\times$'' indicates an $A_{2}$ singularity on $X$ in this example.}
     \end{figure}

     \begin{figure}[h!]\label{ReducedRam-partialpencil}
     \centering
      \includegraphics[width=0.5\textwidth]{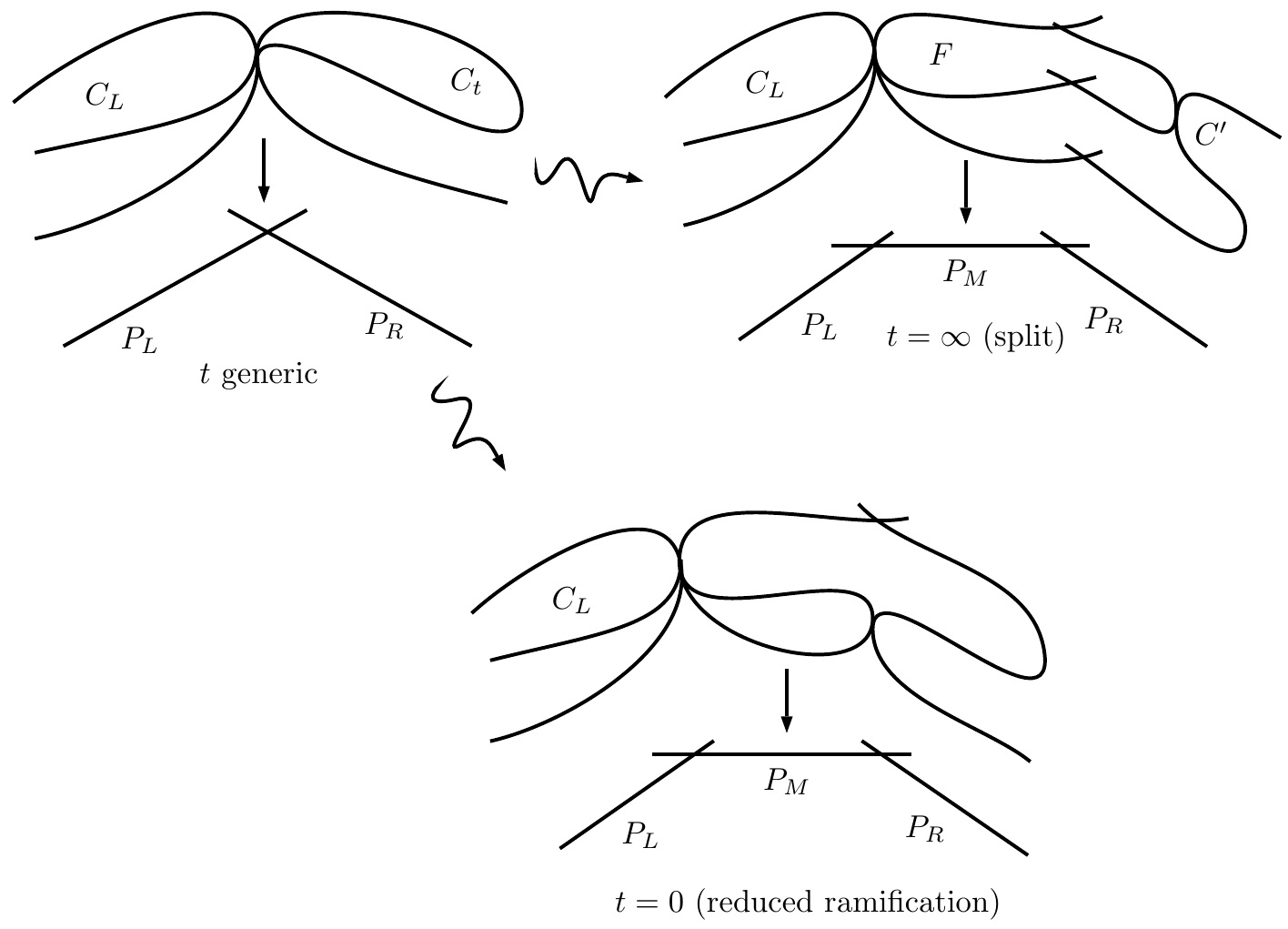}
     \caption{A schematic picture of a partial pencil acquiring a split fiber and ramification reduction.}
     \end{figure}

  \subsubsection{Variation on the basic construction II.}  \label{modified2}

    When constructing families of pentagonal curves, we will find ourselves in a situation not yet encountered in the previous two settings. This variation will not be needed in order to prove \autoref{independence}, so the reader may skip ahead and come back when needed. 

    We return to the general setting: $\pi \from \F \to \P^{1}_{s}$ is a fibered surface and $|\cO_{\F}(C)|$ an appropriate degree $d$, genus $g$ linear system.  We now pick a general pencil $\P_{t} \subset |\cO_{\F}(C)|$, {\sl without any assumption on its base locus.}  Our goal is to make this pencil the right side of a partial pencil.
    
  \subsubsection{}\label{branchingprofilepentagonalfamily}  
    We assume that the pencil $\P_{t}$, when restricted to the fiber curve $F := \pi^{-1}(0)$ is basepoint free, and that the induced branched cover $F \to \P^{1}_{t}$ is simply branched except possible at one point $t = 0 \in P^{1}_{t}$ where the branching profile is $(m_{1}, ... , m_{k})$, $\sum_{i}m_{i} = d$.

  \subsubsection{}

    Let $p \from X \to \P^{1}_{t}$ denote the total space of the pencil, where $X$ is the blow up of $\F$ along the base locus of the pencil.  Then $(p, \pi) \from X \to \P^{1}_{t} \times \P^{1}_{s}$ is a finite, degree $d$ map, and the curve $F$ is the preimage of the line $\P^{1}_{t} \times \{s = 0\}$, i.e. $F$ is a $d$-fold multi-section of the morphism $p$.  

    Of course we cannot glue a fixed family to the family $\phi$ along the multi-section $F$ for two reasons: $(1)$ the monodromy of the projection $F \to \P^{1}_{t}$ interchanges the $d$ points of attachment, and $(2)$ the ramification points of this projection pose an obvious obstruction to glueing.

    To overcome the first issue, we make the base change $\beta \from D \to \P^{1}$ of degree $d !$ which kills all monodromy of the cover $F \to \P^{1}$.   In this way we obtain the base changed family $p_{D} \from X_{D} \to D$ which now has $d$ sections $\sigma_{1}, ... ,\sigma_{d}$. Here $X_{D} := X \times_{\P^{1}_{t}}D$.  

    The local nature of the map $\beta$ is summarized as follows:
    \begin{enumerate}\label{localbasechangedescription}
    \item For each branch point $z\in \P^{1}_{t}$ of $F \to \P^{1}_{t}$ not equal to $0$, there are $d!/2$ preimages in $D$ which, as a set, we denote by $A_{z} \subset D$.  Each point in $A_{z}$ experiences $2:1$ branching under the map $\beta$. 
    \item Above the point $0 \in \P^{1}_{t}$, there are $d!/lcm(m_{1}, ... ,m_{k})$ preimages, which we denote by $B_{0} \subset \P^{1}_{t}$. Every point in $B_{0}$ experiences $lcm(m_{1}, ... ,m_{k}):1$ branching under the map $\beta$.

    \end{enumerate}

    The sections $\sigma_{i}$ are not disjoint in the surface $X_{D}$ -- they meet in two ways:
    \begin{enumerate}\label{sigmameeting}
    \item $\sigma_{i}$ and $\sigma_{j}$ meet above a point $x \in A_{z} \subset D$ for some set $A_{z}$ as defined above. 
    \item $\sigma_{i}$ and $\sigma_{j}$ meet above a point $y \in B_{0} \subset D$.
    \end{enumerate}

    In the first setting, the sections $\sigma_{i}$ and $\sigma_{j}$ meet transversely above $x \in D$, and the remaining sections remain disjoint.  In the latter case, $\sigma_{i}$ and $\sigma_{j}$ will meet with tangency order $lcm(m_{1}, ... ,m_{k})/m_{l}$ for some $m_{l}$ depending on the pair of sections.  This is exactly the setting achieved when performing admissible reduction in section \ref{admissiblereduction}. 

    By blowing up these intersections appropriately, as described in section \ref{admissiblereduction}, we obtain a partial pencil family $p_{D}$ obtained by attaching an unchanging cover on the left. The family $p_{D}$ is completely contained in a boundary divisor $\Delta$.  It interacts with two other boundary divisors: $\Delta_{{\rm simple}}$ at the points in $A_{z}$ and $\Delta_{{m_{1}, ... ,m_{k}}}$ at the points in $B_{0}$.

    \begin{lem}\label{lemma:basechangefamilyproperties}
    The resulting partial pencil $p_{D}$ intersects $\Delta_{{\rm simple}}$ and $\Delta_{{m_{1}, ... ,m_{k}}}$ transversely at the points in $A_{z}$ and $B_{0}$, respectively.
    \end{lem}

    \begin{proof}
    This follows easily from the local coordinates description of the versal deformation spaces found in section \ref{intersectionmultiplicities}.
    \end{proof}

\subsection{Some explicit partial pencils.} \label{sec:somepartialpencils}
  We now construct some useful partial pencil families.

 \subsubsection{Rational partial pencils.}\label{ppencils} 

    Consider the linear system of $(1,d)$ curves on the surface $\F = {\bf P}^1 \times {\bf P}^1$, and pick a $d$-tuple of distinct points $Z = \{p_{1}, ... ,p_{d}\}$ lying entirely in a $(1,0)$ ruling curve $F$.  Consider a general linear pencil ${\bf P}^1_{t} \subset |(1,d)|$  whose base locus $B$ contains $Z$. 

    The total space $X$ of the pencil $$p \colon X \longrightarrow {\bf P}^1_{t}$$ then has the following features: 
    \begin{itemize}
    \item The split fiber $t=\infty$ is $F \cup R_{1} \cup ... \cup R_{d}$ where $R_{i}$ are disjoint $(0,1)$ rulings. 
    \item There are $d$ fibers of the form $A_{i} \cup B_{i}$ where $A_{i}$ are $(0,1)$ lines and $B_{i}$ are smooth $(1,d-1)$ curves.
    \item The section $s_{i}$ of $p$
    corresponding to the basepoint $p_{i}$ obeys $s_{i}^{2} = -1$.
    \end{itemize}

    The curves $R_{i}$ can be contracted to create a family $\o{p} \to \o{X} \to \P^{1}_{t}$ -- we continue to use $s_{i}$ to denote the transformed section.  Since the $R_{i}$ did not pass through the points $p_{i}$, we still have $s_{i}^{2} = -1$.

    We may attach a fixed cover along the sections $s_{i}$ to obtain a partial pencil.    Away from $t= \infty$, we only obtain intersections with enumeratively relevant divisors: triple ramification, $(2,2)$-ramification, and simple nodes will appear.  Since these partial pencils are obtained by varying rational curves, we will call them {\sl rational partial pencils}.
 
 \subsubsection{Hyperelliptic families.}\label{hfamilies} 

    Every hyperelliptic curve $C$ of genus $h$ can be embedded in the Hirzebruch surface $\F := \F_{h}$ as a curve in the linear system $|2 \tau + f|$, where $\tau$ is the section class with $\tau^{2} = h$ and $f$ is the fiber class.  

    By fixing two general points $Z = \{p_{1}, p_{2}\}$ in a ruling line $F \subset \F$  to be lie the base locus of a pencil $\P^{1}_{t} \subset |2\tau + f|$ we obtain the family \[p \from X \to \P^{1}_{t},\] where $X$ is the total space of the pencil. 

    The family $p$ has a split fiber $F \cup C'$ where $C'$ is a hyperelliptic curve of genus $h-1$.  This family has two sections $\sigma_{1}, \sigma_{2}$ along which we may attach more components to create partial pencils. We call partial pencil families constructed in this way {\sl hyperelliptic partial pencils}. 

 \subsubsection{Ramification reducing families.}\label{reducefamilies}

    Start with a pencil $\P^{1}_{t} \subset |(1,d)|$ of bidegree $(1,d)$ curves on $\F := \P^{1} \times \P^{1}$ containing a scheme $Z = \Spec k[\epsilon]/(\epsilon^{d}) \subset F$ in its base locus. (As usual, we assume our pencil is generic among those satisfying this constraint.) 

    Then, as in \autoref{modified}, there is one distinguished element in the resulting partial pencil $p$ corresponding to a ramification reducing fiber. This pencil $p$ lies entirely within a unique boundary divisor $\Delta$ and intersects the ramification reduction $\Delta_{{\rm ram}}$   

    We need a slight generalization of the pencil above.  In particular we will need to fix a {\sl different} fiber $F'$ and create a partial pencil with two places to attach the rest of an admissible cover. Of course, in order for the fiber $F'$ to create a family of gluing sections, we must make an appropriate base change as in section \ref{modified2}. 

    The resulting partial pencil $p'$ lives in {\sl two} boundary divisors: $\Delta$, $\Delta_{{\rm unram}}$.  Furthermore, it intersects $\Delta_{{\rm ram}}$ as before, but also intersects a divisor $\Delta_{{\rm simple}}$, corresponding to the points ``$A_{z}$'' as in section \ref{modified2}.

\subsection{Independence of the boundary.}
  At this point, we have all necessary ingredients to prove \autoref{independence}.
    
  \subsubsection{Independence of the boundary of $\widetilde{\H}_{d,g}$.}
    First we need to understand the boundary of the partial compactification $\widetilde{\H}_{d,g}$:
    \begin{prop}\label{indHtilde}
    $\sfT$, $\sfD$, and the boundary divisors are independent in ${\rm Pic}_{\Q}\,\widetilde{\H}_{d,g}$.  
    \end{prop}

    \begin{proof}
    Let $\delta_{i,j}$ be the boundary divisor in $\widetilde{\H}_{d,g}$ generically parametrizing covers \[\alpha \from C_{1} \cup C_{2} \to \P^1\] where $g(C_{1}) = i$ and $\deg \alpha|_{C_{1}} = j$. (Note: $\delta_{i, j} = \delta_{g-i, d-j}$.) Furthermore, let $\delta_{\rm irr}$ denote the boundary divisor generically parameterizing covers $\alpha \from C \to \P^{1}$ where $C$ is irreducible and nodal.

    Begin with any relation:
    \begin{equation}\label{rel}
    0 = a\cdot \sfT + b\cdot \sfD  + c \cdot \delta_{\rm irr} + \sum c_{i,j} \cdot \delta_{i,j}
    \end{equation}

    We use test families to express the coefficients $c_{i,j}$ as linear combinations of $a,b$, and $c$.  The divisor classes $\sfT$, $\sfD$, and $\delta_{\rm irr}$ are known to be independent \cite{dp:pic_345}[Proposition 2.15].
    \begin{lem}\label{coeffcij}
    The coefficients $c_{i,j}$ are linear combinations of $a,b$, and $c$. 
    \end{lem}

    \begin{proof} Let $\Gamma_{k}$ be the family of covers obtained by varying a $(1,k)$ curve in a pencil on $\P^1 \times \P^1$ and attaching a fixed general genus $g$, degree $d-k$ cover $D$ of $\P^1$ at a base point of the pencil.  

    Then the family $\Gamma_{k}$ interacts nontrivially with the divisors $\delta_{0,k}$, $\delta_{0,k-1}$, $\delta_{0,1}$, $\sfT$ and $\sfD$.  

    In addition to the families $\Gamma_{k}$ consider the following families:
    \begin{enumerate}[$B_{1}$:]
     \item Begin with a general pencil $p$ of hyperelliptic  genus $g$ curves on $\F_{g}$ and attach a fixed general degree $d-2$ genus $0$ cover $R \to \P^{1}$ at a base point of the pencil.  $B_{1}$ has nontrivial intersection with $\delta_{0,d-2}$, $\sfT$, $\delta_{\rm irr}$, and $\sfD$.  
     \item[$B_{2}$:] Begin with a pencil $p$ of trigonal genus $g$ covers on a Hirzebruch surface, and attach an unchanging degree $d-3$ rational cover $R \to \P^{1}$ at a base point.  $B_{2}$ has nontrivial intersection with $\delta_{0,d-3}$, $\sfT$, $\delta_{\rm irr}$, and $\sfD$.  
    \end{enumerate}
    The families $\Gamma_{k}$, $B_{1}$, and $B_{2}$ provide enough relations to conclude \autoref{coeffcij} for  coefficients of the form $c_{0,k}$.  

    Now we argue that knowing  \autoref{coeffcij} for the coefficients $c_{0,k}$ allows one to deduce it for all coefficients $c_{i,j}$.  Indeed, suppose $i >0$.  Then it must follow that $j > 1$.  (The genus $i$ component must map with degree $j > 1$.)  

    Consider the following family:  
    \begin{enumerate}[$B_{3}$:]
    \item Begin with a general pencil $p$ of hyperelliptic curves of genus $i$ on the Hirzebruch surface $\F_{i}$.  At one base point, attach an unchanging genus $g-i$ cover $C_{1} \to \P^{1}$ of degree $d-j$.  Choose a second base point (which is not the hyperelliptic conjugate to the first),  and attach a second unchanging genus $0$ cover $R_{2} \to \P^{1}$ of degree $j-2$.  (If $j = 2$, simply omit $R_{2}$.) 
    \end{enumerate}
    The family $B_{3}$ only interacts nontrivially with $\sfT$, $\delta_{\rm irr}$, $\sfD$, $\delta_{i,j}$ and $\delta_{0,j-2}$.  The relations obtained from the test families of type $B_{3}$ allow us to finally conclude  \autoref{coeffcij}.

    \end{proof}

    \end{proof}

  \subsubsection{}
    
    Let's return to the proof of \autoref{independence}. We use the method of test families.  Begin with a dependence relation in ${\rm Pic}_{\Q}\, \Hdg$ \[ 0 = \sum_\Delta c(\Delta) \cdot \Delta\]    Let us rewrite the dependence relation as:

    \begin{equation}\label{hyprelation}
    a \sfT + b\Delta_{\rm irr} + c\sfD + \sum_{\Delta'\text{\,\, enum. relevant.}} c(\Delta') \cdot \Delta' + \sum_{\Delta}c(\Delta)\cdot \Delta = 0
    \end{equation}
     where the first sum is over all  enumeratively relevant divisors {\sl other than} $\sfT, \sfD$, and $\Delta_{\rm irr}$.  We will refer to the coefficients $a$, $b$, $c$, and $c(\Delta')$ as the {\sl enumeratively relevant} coefficients. 

  \subsubsection{}

    The first claim we make is

    \begin{prop}\label{reduction}
    The coefficients $c(\Delta)$ are linear combinations of enumeratively relevant coefficients.
    \end{prop}
      Before proving \autoref{reduction}, we will need a definition. 
     
     \begin{defn}
     The  {\sl excess}, ${\rm ex}(\Delta)$, of a boundary divisor $\Delta$ is the number  
     \[{\rm ex}(\Delta) : = \min \Big\{r(\Delta)+ \sum_{v \,\, \textrm{L-vertex}} g_v ,\,\, r(\Delta)+ \sum_{v \,\,\textrm{R-vertex}} g_v \Big\}.\] 
     \end{defn}
     
    Observe that the divisors $\Delta$ with ${\rm ex}(\Delta) = 0$ are precisely those which are unramified and have only rational components on one side.

    \begin{lem}\label{lem1}
    If ${\rm ex}(\Delta) =0$, then $c(\Delta)$ is a linear combination of the enumeratively relevant coefficients.
    \end{lem}
    \begin{proof}
       Pick a rational vertex $v \in \Gamma_{\Delta}$ which has degree $d_v \geq 2$.  Vary $v$ in a rational partial pencil as in \autoref{ppencils}.  The resulting partial pencil test family $S$ has negative intersection with $\Delta$ and otherwise only interacts with enumeratively relevant boundary divisors.  Therefore, we obtain a relation relating $c(\Delta)$ with enumeratively relevant coefficients, which is what we wanted to show.
    \end{proof}

    \begin{lem}\label{lem2}
    Suppose ${\rm ex}(\Delta) > 0$.  Then $c(\Delta)$ is a linear combination of the enumeratively relevant coefficients and various $c(\Delta_i)$,  where ${\rm ex}(\Delta_i) < {\rm ex}(\Delta)$.
    \end{lem}

    \begin{proof}
    If the divisor $\Delta$ is unramified, pick a vertex $v \in \Gamma_{\Delta}$ of genus $g_v > 0$ on the side with smaller total sum of genera.  Suppose the dual graph, locally around $v$, has the following form: \[   \xygraph{
    !{<0cm,0cm>;<1cm,0cm>:<0cm,1cm>::}
    !{(0,.9) }*{}="l_1"
    !{(0,.3) }*{}="l_2"
    !{(0,-.3) }*{}="l_3"
    !{(0,-.9) }*{}="l_4"
    !{(2,0) }*{\bullet}="r"
    !{(.6,.15) }*{\vdots}
    !{(-.2,.9) }*{\vdots}
    !{(-.2,.3) }*{\vdots}
    !{(-.2,-.3) }*{\vdots}
    !{(-.2,-.8) }*{\vdots}
    !{(2.2,.3) }*{\scriptstyle v }
    !{(2.5,-.3) }*{\scriptstyle (d_v, g_v) }
     "l_1"-@/^.3cm/"r" 
    "l_2"-@/^/"r"
    "l_3"-@/_/"r" 
    "l_4"-@/_.3cm/"r" 
    }  \]
     We will now ``concentrate'' all of the genus of the vertex $v$ into a hyperelliptic vertex, which allows us to vary the vertex as in \autoref{hfamilies}.  

     To accomplish this, we  replace this local part of the dual graph with the following picture. 
     \begin{equation}\label{hyp-mod}
       \xygraph{
    !{<0cm,0cm>;<1cm,0cm>:<0cm,1cm>::}
    !{(0,.9) }*{}="l_1"
    !{(0,.3) }*{}="l_2"
    !{(0,-.3) }*{}="l_3"
    !{(0,-.9) }*{}="l_4"
    !{(2,0) }*{\bullet}="r"
    !{(2,-1) }*{\bullet}="r_2"
    !{(1,.7) }*{\bullet}="m_1"
    !{(1,.3) }*{\bullet}="m_2"
    !{(1,-.5) }*{\bullet}="m"!{(1,-.7) }*{\scriptstyle v'}!{(1,-1) }*{\scriptstyle g_{v'} = g_v}!{(1,-1.3) }*{\scriptstyle d_{v'} = 2}
    !{(1,0) }*{\vdots}
    !{(-.2,.9) }*{\vdots}
    !{(-.2,.3) }*{\vdots}
    !{(-.2,-.3) }*{\vdots}
    !{(-.2,-.8) }*{\vdots}
    !{(2.2,.3) }*{\scriptstyle w }
    !{(2.9,-.3) }*{\scriptstyle d_w = d_v-1 , \,g_w = 0 }
     "m_1"-@/^.1cm/"r" 
    "m_2"-@/^.1cm/"r"
    "m"-@/^.1cm/"r" 
    "m"-@/^.1cm/"r_2" 
    "m"-@/_.1cm/"l_4" 
    "m"-@/_.1cm/"l_3" 
    "m_2"-@/_.1cm/"l_2"
    "m_1"-@/_.1cm/"l_1"
    }  
    \end{equation}

    Now we vary the genus $g_v$ hyperelliptic curve corresponding to $v'$ in a pencil of $|2\tau + f|$ curves on ${\bf F}_{g_v}$, with {\sl two pairs} of prescribed basepoints lying on two distinct fibers of the natural projection $\pi \colon {\bf F}_{g_v} \longrightarrow {\bf P}^1$.  (We need two pairs of points to glue this family with the rest of the admissible cover as dictated by the altered dual graph above.)  

    The resulting test family $S$  provides the relation we seek because it has negative intersection with $\Delta$ and with two other higher boundary divisors, each having smaller excess. ($S$  also interacts with enumeratively relevant divisors.)

    If the divisor $\Delta$ is ramified, we consider two cases.  First, suppose there is a vertex $v$, such that locally around $v$ the dual graph has the form: 
    \[  \xygraph{
    !{<0cm,0cm>;<1cm,0cm>:<0cm,1cm>::}
    !{(0,.9) }*{}="l_1"
    !{(0,.3) }*{}="l_2"
    !{(0,-.3) }*{}="l_3"
    !{(0,-.9) }*{}="l_4"
    !{(2,0) }*{\bullet}="r"
    !{(.7,.1) }*{\vdots}
    !{(-.2,.9) }*{\vdots}
    !{(-.2,.3) }*{\vdots}
    !{(-.2,-.3) }*{\vdots}
    !{(-.2,-.8) }*{\vdots}
    !{(2.2,.3) }*{\scriptstyle v }
    !{(2.5,-.3) }*{\scriptstyle (d_v, g_v) }
     "l_1"-@/^.3cm/"r"^(.4){m} 
    "l_2"-@/^/"r"
    "l_3"-@/_/"r" 
    "l_4"-@/_.3cm/"r" 
    }  \]
     (The top edge in the diagram has local degree $m$.) We replace this local picture with: 
    \[  \xygraph{
    !{<0cm,0cm>;<2cm,0cm>:<0cm,2cm>::}
    !{(0,.3) }*{}="l_2"
    !{(0,-.3) }*{}="l_3"
    !{(0,-.6) }*{}="l_{3.5}"
    !{(0,-.9) }*{}="l_4"
    !{(2,0) }*{\bullet}="r"
    !{(2,.3) }*{\bullet}="r+"
    !{(2,.5) }*{\bullet}="r++"
    !{(2,-.3) }*{\bullet}="r_2"
    !{(2,-.5) }*{\bullet}="r_3"!{(2,-.65) }*{\vdots}
    !{(2,-.9) }*{\bullet}="r_4"
    !{(1,.3) }*{\bullet}="m_2"!{(.8,.2) }*{\scriptstyle v'}!{(.8,0) }*{\scriptstyle g_{v'} = 0} !{(.8,-.2) }*{\scriptstyle d_{v'} = m+1}
    !{(1,-.5) }*{\bullet}="m"!{(1,-.7) }*{\scriptstyle w'}!{(1,-.9) }*{\scriptstyle g_{w'} = g_{v}}!{(1,-1.1) }*{\scriptstyle d_{w'} = d_{v} - (m+1)}
    !{(-.2,.3) }*{\vdots}
    !{(-.2,-.3) }*{\vdots}
    !{(.1,-.65) }*++++{\vdots}
    !{(-.2,-.8) }*{\vdots}
    !{(1.5,.2) }*{\vdots}
    "m_2"-@/^.3cm/"r++" 
    "m_2"-@/^.3cm/"r+"
    "m_2"-@/_.3cm/"r"
    "m_2"-@/_.3cm/"r_2" 
    "m"-@/_.1cm/"l_4" 
    "m"-@/_.1cm/"l_3" 
    "m"-@/_.1cm/"l_{3.5}"
    "m_2"-@/_.1cm/"l_2"_{m}
    "m"-@/_.1cm/"r_4" 
    "m"-@/_.1cm/"r_3" 
    "m"-@/_.1cm/"r_2" 
    }  \]

    We vary the totally ramified, degree $m$ rational vertex $v'$ in ${\bf P}^1 \times {\bf P}^1$ with an $(m)$-fold basepoint fixed in a fiber as in \autoref{reducefamilies}.  From section \ref{reducefamilies}, the resulting test family will intersect $\Delta$ negatively, and will intersect other higher boundary divisors which have strictly smaller total ramification index than $\Delta$, {\sl unless } $m=1$, in which case we may not assume strictness.  Therefore, we may assume that all nonzero ramification indices are $1$, i.e. when ramification occurs over the node, it is simple.

    Pick such a simply-ramified divisor $\Delta$. If all genera of all components on one side of $\Gamma_{\Delta}$ are $0$, then we may use a partial pencil family of $(1,k)$ curves in $\P^{1} \times \P^{1}$ with a nonreduced base locus as in \autoref{modified} to conclude.  

    So we may assume $v \in \Gamma_{\Delta}$ is a simply ramified vertex with $g(v) > 0$ on the side with smaller total genera.  We can vary $v$ in a hyperelliptic family as in \eqref{hyp-mod}. More precisely, suppose the local picture of $\Gamma_{\Delta}$ around the vertex $v$ is as follows: 

    \[   \xygraph{
    !{<0cm,0cm>;<1cm,0cm>:<0cm,1cm>::}
    !{(0,.9) }*{}="l_1"
    !{(0,.3) }*{}="l_2"
    !{(0,-.3) }*{}="l_3"
    !{(0,-.9) }*{}="l_4"
    !{(2,0) }*{\bullet}="r"
    !{(.6,.15) }*{\vdots}
    !{(-.2,.9) }*{\vdots}
    !{(-.2,.3) }*{\vdots}
    !{(-.2,-.3) }*{\vdots}
    !{(-.2,-.8) }*{\vdots}
    !{(2.2,.3) }*{\scriptstyle v }
    !{(2.5,-.3) }*{\scriptstyle (d_v, g_v) }
     "l_1"-@/^.3cm/"r"^2 
    "l_2"-@/^/"r"^{a_{1}}
    "l_3"-@/_/"r"^{a_{2}} 
    "l_4"-@/_.3cm/"r"^{a_{3}}
    }  \] 

    Then we may modify this picture with the following picture: 

    \[  \xygraph{
    !{<0cm,0cm>;<2cm,0cm>:<0cm,2cm>::}
    !{(0,.3) }*{}="l_2"
    !{(0,-.3) }*{}="l_3"
    !{(0,-.6) }*{}="l_{3.5}"
    !{(0,-.9) }*{}="l_4"
    !{(2,.3) }*{\bullet}="r+"
    !{(2,-.3) }*{\bullet}="r_2"
    !{(2,-.5) }*{\bullet}="r_3"!{(2,-.65) }*{\vdots}
    !{(2,-.9) }*{\bullet}="r_4"
    !{(1,.3) }*{\bullet}="m_2"!{(.8,.2) }*{\scriptstyle v'}!{(.8,0) }*{\scriptstyle g_{v'} = g(v)} !{(.8,-.2) }*{\scriptstyle d_{v'} = 2}
    !{(1,-.5) }*{\bullet}="m"!{(1,-.7) }*{\scriptstyle w'}!{(1,-.9) }*{\scriptstyle g_{w'} = 0}!{(1,-1.1) }*{\scriptstyle d_{w'} = d_{v} - 2}
    !{(-.2,.3) }*{\vdots}
    !{(-.2,-.3) }*{\vdots}
    !{(.1,-.65) }*++++{\vdots}
    !{(-.2,-.8) }*{\vdots}
    "m_2"-@/^.3cm/"r+"
    "m_2"-@/_.3cm/"r_2" 
    "m"-@/_.1cm/"l_4" 
    "m"-@/_.1cm/"l_3" 
    "m"-@/_.1cm/"l_{3.5}"
    "m_2"-@/_.1cm/"l_2"_{2}
    "m"-@/_.1cm/"r_4" 
    "m"-@/_.1cm/"r_3" 
    "m"-@/_.1cm/"r_2" 
    }  \]

    We then vary the genus $g(v)$ vertex $v'$ in a pencil of $|2\tau + f|$ hyperelliptic curves in $\F_{g_{v}}$ as in section \ref{hfamilies}. 

    The resulting test curve will produce a reduction in total excess. We continue using such families to reduce excess to zero.

    \end{proof}

    Finally, we show that the enumeratively relevant coefficients are zero. 

    \begin{prop}\label{lem3}
    The enumeratively relevant coefficients $c(\Delta')$ are zero.
    \end{prop}

    \begin{proof}
    This proposition follows from \autoref{coeffcij}, as we now explain. We note that the Picard groups of the moduli stacks $\tw{\H}_{d,g}, \Hdg$ are the same as that of their coarse spaces after tensoring with $\Q$.

     The moduli spaces $\Hdg$ and $\tw{\H}_{d,g}$ share a common open subspace.  This means there is a rational map of coarse varieties $$G \from \tw{H}_{d,g} \to \overline{H}_{d,g}$$
    where $\tw{H}_{d,g}$ and $\overline{H}_{d,g}$ denote coarse spaces. The birational map $G$ is defined away from a codimension two set in $\tw{H}_{d,g}$, and is one-to-one on the level of points. 

    The boundary divisors in the image of $G$ are precisely the enumeratively relevant divisors.  Therefore, any hypothetical nontrivial relation \ref{hyprelation} would pull back to a nontrivial relation in ${\rm Pic}_{\Q}\,\tw{H}_{d,g}$, which \autoref{indHtilde} prohibits.  Thus our proposition follows from \autoref{indHtilde}.
    \end{proof}

    \theoremstyle{remark}

    \autoref{independence} now follows from \autoref{lem1},  \autoref{lem2}, and  \autoref{lem3}.

\section{One parameter families and \autoref{rigidityextremality}.} \label{ThirdSection}

Throughout this section, we work primarily in the partial compactification $\td \H_{d,g}$.  Our first goal is to understand basic numerical quantities attached to general complete one parameter families in $\td \H_{d,g}$.

Next, we investigate basic families of trigonal, tetragonal, and pentagonal curves. The existence of these families leads to the proof of \autoref{rigidityextremality}, which we naturally break into three parts by the degree. 

Finally, we end by discussing a purely speculative approach to showing extremality/rigidity of the Maroni divisor for higher degrees $d$. 

\subsection{}

Let 
 \begin{equation} \nonumber
\xymatrix @C -1.1pc  {
\cC \ar[rr]^{\alpha} \ar[rd]_f&&
\P \ar[ld]^p \\
& S}
 \end{equation} 
be a family of degree $d$ branched coverings $\alpha_{s} : C_{s} \to \P^{1}$, $s \in S$, where $\P := \P^{1} \times S$.  Suppose $S$ is a complete smooth curve,  and $f$ is generically smooth  with genus $g$ at-worst-nodal fibers. In other words, we have a family $S \to \td \H_{d,g}$ not entirely lying in the boundary.

Our first task is to compute the invariants $\lambda$ and $\delta$ for the family $f \from \cC \to S$ in terms of the Chern classes of the Tschirnhausen bundle $\cE$ and the  bundle of quadrics $\cF$ associated to the branched covering of surfaces $\alpha$. Here, $\lambda = c_{1}(\E^{g})$ where $\E^{g} := f_{*}(\omega_{f})$ is the Hodge bundle, and $\delta$ denotes the number of singular fibers counted with the usual multiplicity.

\subsection{}
 Recall that the bundle $\dual{\cE}$ for the map $\alpha$ is defined by the formula:
\begin{equation}\label{Edirectsum}
\alpha_{*}\cO_{\cC} = \cO_{\P} \oplus \dual{\cE}.
\end{equation} In what follows, $ch$ and $ch_{k}$ will denote the Chern character and the degree $k$ part of the Chern character, respectively.

\begin{prop}\label{lambda:covers} For the family $f \from \cC \to S$ above, we have:
\begin{align}
&\lambda = ch_2(\cE) - \frac{c_{1}^{2}(\cE)}{b}& \label{lambdaexpression}\\ 
&\kappa 
 = d \cdot ch_{2}\cE - ch_{2}\cF + \left(\frac{1}{2} - \frac{8}{b}\right)c_{1}^{2}\cE& \label{kappaexpression}\\
&\delta = (12-d)ch_{2}\cE + ch_{2}\cF  -\left(\frac{1}{2} -  \frac{4}{b}\right)c_{1}^{2}\cE&
\end{align}
where $b = 2g+2d-2$ is the number of branch points.
\end{prop}

\begin{proof}
First we recall that the Grothendieck-Riemann-Roch formula yields: 
\[\lambda = \chi (\cO_{\cC}) - \chi(\cO_{S})(1-g).\]
Since $\chi(\cO_{\cC}) = \chi(\alpha_{*}\cO_{\cC}) = \chi(\cO_{\P}) + \chi(\dual \cE)$ we get: 
\[\lambda = \chi(\cO_{\P}) + \chi(\dual \cE) - (1-g)\chi(\cO_{S}) .\]

We calculate $\chi(\dual \cE)$ by applying the Hirzebruch-Reimann-Roch formula to the bundle $\dual \cE$. This says that 
\[\chi(\dual \cE) = [ch(\dual \cE) \cdot Td(T_{\P})]_{2} = ch_{2}(\dual \cE) + c_1(\cE)(K_{\P}/2) + (d-1)\chi(\cO_{\P}).\] 

Let $\sigma \in \Pic_{\Q} \P$ denote the rational section class $[-\omega_{p}/2]$, so that $\sigma^{2} = 0$.  Since the degree of $\cE$ is $g+d-1$ when restricted to any $\P^{1} \times \{s\} \subset \P$, we may write $c_{1}(\cE) = (g+d-1)\sigma + mf$, where $f$ is the class of a fiber of $p$.  

Squaring gives \[m = \frac{c_{1}^{2}\cE}{b}.\]  The equality $K_{\P}/2 = -\sigma - \chi(\cO_{S})\cdot f$ implies \[c_{1}(\cE) \cdot K_{\P}/2 = -m - (g+d-1) \chi(\cO_{S}).\]

Altogether, we obtain \[\lambda = ch_{2}(\dual \cE) - m + d\chi(\cO_{\P}) - d\chi(\cO_{S}) = ch_{2}(\dual \cE) - \frac{c_{1}^{2}\cE}{b},\] since $\chi(\cO_{\P}) = \chi(\cO_{S})$.  This gives \eqref{lambdaexpression}.

Mumford's relation $12 \lambda = \kappa +\delta$ implies we need only compute $\kappa$. The computation of $\kappa$ uses the Grothendieck-Riemann-Roch formula applied to the line bundle $\omega_{\alpha}^{\otimes 2}$ and the map $\alpha$ as in the proof of \ref{relatecFcE}: 
\[ch(\alpha_{*}\omega_{\alpha}^{\otimes 2})Td(T_{\pi}) = \alpha_{*}(ch(\omega_{\alpha}^{\otimes 2})Td(T_{f})).\]

By using $ch(\alpha_{*}\omega_{\alpha}^{\otimes 2}) = ch(\Sym^{2}\cE) - ch(\cF)$ and taking degree $2$ parts of the above equation, we end up with the equality 
\[ch_{2}(\Sym^2\cE) - ch_{2}(\cE) - ch_{2}(\cF) = c_{1}^{2}(\omega_{\alpha}).\]

Finally, by using $c_{1}(\omega_{f}) = c_{1}(\omega_{\alpha}) + \alpha_{*}c_{1}(\omega_{\pi})$, one checks that $$\kappa = c_{1}^{2}(\omega_{f}) = c_{1}^{2}(\omega_{\alpha}) - \frac{8}{b}c_{1}^{2}(\cE).$$

Putting all of the above together yields \ref{kappaexpression}.

\end{proof}



Let $\sfT \subset S$ and $\sfD \subset S$ denote the divisors of points in $S$ where the map $\alpha_{s} \from C_{s} \to \P^1$ has a triple ramification point and a $(2,2)$ pair of ``doubled" ramification points, respectively.  

\begin{prop}\label{classTD}
We have the following equalities: 
\begin{align}
&\sfT = (3d-12)ch_{2}\cE - 3ch_{2}\cF + \frac{3}{2}c_{1}^{2}\cE,&\\
&\sfD = 4ch_{2}\cF - (4d-12)ch_{2}\cE.&
\end{align}
\end{prop}
\begin{proof}
 Let 
 \begin{equation} 
\xymatrix @C -1.1pc  {
\cC \ar[rr]^{\alpha} \ar[rd]_f&&
\P \ar[ld]^{\pi} \\
& S}
 \end{equation} 
 be a general family of covers in $\widetilde{\H}_{d,g}$ with $S$ being a complete curve. We pick the family general enough so that the ramification divisor $R \subset \cC$ of the finite map $\alpha$ is smooth and maps birationally onto the branch divisor $B \subset \P$, with $B$ having simple nodes and cusps only.  The cusps of $B$ will then count as intersections of our family with the divisor $\sfT$ and the nodes of $B$ will count as intersections with the divisor $\sfD$.  We then see that 
 \begin{equation}\label{comparegenera}
 p_{a}(B) - p_{a}(R) = \sfT + \sfD
 \end{equation}
 By adjunction on $\cC$ and $\P$, this leads to the relation 
 \begin{equation}\label{T+D}
 B^{2} + c_{1}\omega_{\pi}\cdot B - (R^{2} + c_{1}\omega_{f}\cdot R) = B^{2} - 2R^{2} = \sfT + \sfD
 \end{equation}
 
 On the other hand, the ramification divisor $R$ is itself a branched cover of $S$ under the map $f$.  The branch points of the map $f|_{R} \from R \to S$ provide intersections with $\delta$ or with $\sfT$. By adjunction on the surface $\cC$, we get the equality 
 \begin{equation}
 R^{2} + R \cdot c_{1}\omega_{f}  = 2R^{2} + B \cdot c_{1}\omega_{\pi}= \sfT + \delta
 \end{equation}

Recall that $B = 2c_{1}\cE$, so $B^{2} = 4c_{1}^{2}\cE$.  Furthermore, from the proof of \autoref{lambda:covers} we see that $B \cdot \omega_{\pi} = \frac{-2c_{1}^{2}\cE}{b}$.  The current proposition now follows from \autoref{lambda:covers} and the following lemma: 

\begin{lem}\label{R^2}
We have 
\begin{equation}
R^{2} = d \cdot ch_{2}\cE - ch_{2}\cF + \frac{c_{1}^{2}\cE}{2}.
\end{equation}
\end{lem}
\begin{proof}[Proof of \autoref{R^2}]
This follows from applying the Grothendieck-Riemann-Roch formula to the map $\alpha \from \cC \to \P$ and the sheaf $\cO_{\cC}(2R) = \omega_{\alpha}^{\otimes 2}$. 
\end{proof}
\end{proof}

\begin{prop}\label{relationTD} For a family $f \from \cC \to S$ as above, 
\begin{equation}\label{relationTDlambdadelta}
24(b-1)\lambda - 3(b-2)\delta +  6\sfD - (b-10)\sfT = 0
\end{equation}
where $b = 2g+2d-2$ is the number of branch points.
\end{prop}

\begin{proof}
This follows immediately from \autoref{lambda:covers} and \autoref{classTD}.
\end{proof}

\subsection{}

Given the relation in \autoref{relationTD}, we will use $\lambda, \delta$, and $\sfD$ as generators for the $\Q$-vector subspace of ${\Pic}_{\Q}\td \H_{d,g}$ spanned by $\lambda, \delta, \sfT$, and $\sfD$. It can be proved that these three divisor classes are independent in ${\Pic}_{\Q}\td \H_{d,g}$ \cite{dp:pic_345}[Proposition~2.15].  Furthermore, we note that the relation \eqref{relationTDlambdadelta} shows that the analogous divisors (with abuse of notation) $\lambda, \delta, \sfT,$ and $\sfD$ are dependent in ${\Pic}_{\Q}\Hdgn$, modulo higher boundary divisors.  We will choose to use the divisors $\lambda, \delta$ and $\sfD$ as generators of the subspace spanned by $\lambda, \delta, \sfT$, and $\sfD$. 

\subsection{The divisor classes of $\sfM$ and $\sfCE$.}

We now express the classes of the divisors $\sfM$ and $\sfCE$ in terms of $\lambda, \delta$, and $\sfD$ in the Picard group of the partial compactification $\td \H_{d,g}$. 

 Since $\pi \colon \P \longrightarrow \widetilde{\H}_{d,g}$ is a ${\bf P}^1$-bundle, we can use the Bogomolov expressions for the universal bundles $\cE_{univ}$ and $\cF_{univ}$, respectively.  Given any rank $r$ locally free sheaf $\cG$ on $\P$, the Bogomolov expression for $\cG$ is 
\begin{equation}\label{bogomolov}
\text{Bog}(\cG) :=c_2(\cG) -  \frac{(r-1)}{2r}c_1^2(\cG).
\end{equation}
 (It is, up to scaling, the unique linear combination of $c_2$ and $c_1^2$ which is invariant under tensoring by line bundles.) When $\cG$ restricts to a perfectly balanced bundle on the general fiber of $\pi$, the Bogomolov expression $\text{Bog}(\cG)$ detects the change in splitting type of $\cG$ on the fibers of the ${\bf P}^1$-bundle $\P$ \cite{moriwaki98:_relat_bogom}.   Using \autoref{lambda:covers} and \autoref{classTD}, we conclude:

\begin{prop}\label{classofMCE} In ${\rm Pic}_{\Q}\,\widetilde{\H}_{d,g}$, the divisors $\sfM$ and $\sfCE$, when they exist, are expressible in terms of $\lambda, \delta$, and $\sfD$.  These expressions are:
\begin{eqnarray*}
&\sfM = \left(\frac{\frac{10-d}{d-1}b - 8}{b-10}\right)\lambda - \left(\frac{\frac{1}{d-1}b - 2}{b-10}\right)\delta +\left(\frac{\frac{1}{d-1}b-2}{4(b-10)}\right)\sfD &\\
&\sfCE = \left(\frac{(21-d-\frac{54}{d})b-8d+24}{b-10}\right)\lambda - \left(\frac{(2-\frac{6}{d})b - 2d +6}{b-10}\right)\delta + \left(\frac{(1-\frac{6}{d})b - 2d +16}{4(b-10)}\right)\sfD.&
\end{eqnarray*}
Here, as usual, $b = 2d + 2g - 2$.
\end{prop}

\begin{proof} Follows from \autoref{lambda:covers} and \autoref{classTD} and the Bogomolov expression \eqref{bogomolov}.
\end{proof}

\begin{cor}\label{indepMCE}
Let $d \geq 4.$ Then the divisor classes $\sfM$ and $\sfCE$ and boundary divisors of $\o{\H}_{d,g}$ are linearly independent. The same is true when $d=3$ when we omit $\sfCE$ (which is zero).
\end{cor}

\begin{proof}
This follows from the independence of $\lambda, \delta,$ and $\sfD$ and the boundary which is a consequence of \autoref{independence}. 
\end{proof}

\subsection{The effective divisor class $\sfX$.}
Now we notice  in \autoref{classofMCE} that when $d \leq 5$, the coefficient of $\sfD$ is positive in $\sfM$ and negative in $\sfCE$.  We let $X$ be the unique (up to scaling) effective linear combination of $\sfM$ and $\sfCE$ making the coefficient of $\sfD$ zero.  

In ${\rm Pic}_{\Q}\widetilde{\H}_{d,g}$ the divisor class $\sfX$ has the form \[\sfX = a\lambda - b\delta\] for some $a$ and $b$.  The reader can check that the ratio  $a/b$ is $(13g+15)/2g$ and $(31g + 44)/5g$ when $d = 4$ and $d=5$, respectively.  Furthermore, when $d=3$, neither $\sfCE$ nor $\sfD$ exist, and the slope $a/b$ of $\sfM$ is $(7g+6)/g$.  These are precisely the numbers $s(\M^{1}_{d,g})$ occurring in \autoref{slope45}.

\subsection{}
Now we pass to the compactification of Hurwitz space $\Hdgnt$ constructed in \cite{deopurkar13:_compac_hurwit_spaces}.  $\Hdgnt$ parametrizes admissible covers where at most two branch points are allowed to coincide at any given point in the target. See fig. \ref{Deopurkarcovers}. We may  think of $\sfX$ as an effective divisor in $\Hdgnt$  (by taking closures of $\sfM$ and $\sfCE$) and then write 
\begin{equation}\label{divX}
\sfX = a\lambda - b\delta - \sfY   
\end{equation} where $\sfY$ is supported on the higher boundary divisors of the admissible covers compactification. Such an expression is unique, due to \autoref{independence}. 

\begin{figure}\label{Deopurkarcovers}
 \centering
  \includegraphics[width=.75\textwidth]{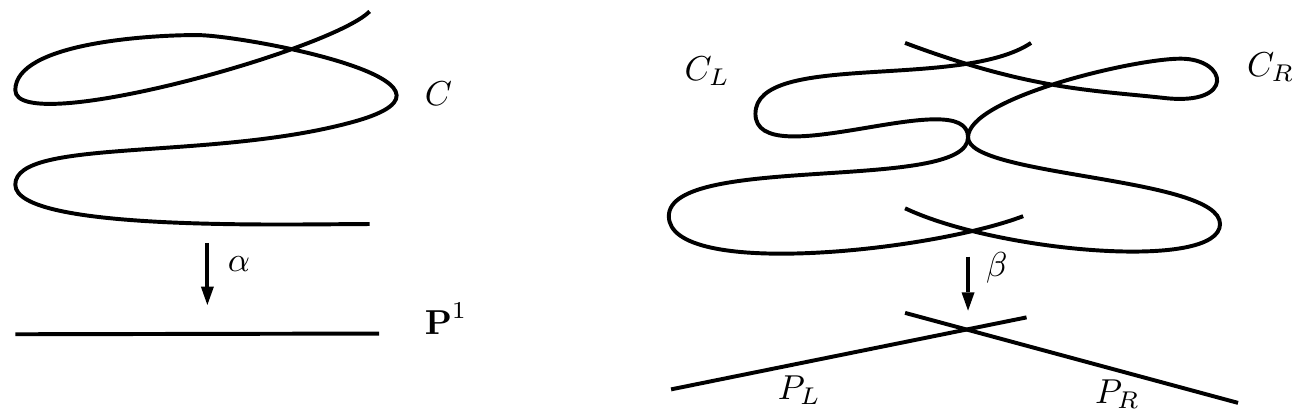}
 \caption{\small Two covers in $\Hdgnt$.  $\alpha$ is also a point in the partial compactification $\widetilde{\H_{d,g}}$. Note that $\alpha$ has $(2,2)$ ramification and also has a node.  $\beta$ is an element in $\Hdgnt$ - it also has $(2,2)$ ramification and a node which does not lie over the node of the target.}
 \end{figure}

\subsection{The key result.}
The following is the key statement allowing us to establish slope bounds 
\autoref{slope45}:

\begin{thm}[Key result for \autoref{slope45}]\label{Yeffective}
Let $d=3, 4$ or $5$, and let $g$ be such that $\sfM$ and $\sfCE$ both exist as  effective divisors in $\H_{d,g}$.  Then the divisor class $\sfY$ appearing in \eqref{divX} is an effective combination of boundary divisors.
\end{thm}

The proof of \autoref{Yeffective} will be the subject of section \ref{slopeproblem}. 

\begin{rmk}
We only want to demonstrate the effectivity of ${\sfY}$ -- we will not calculate the exact class of ${\sfY}$.  (To see an exact expression of the class of ${\sfM}$ in $\overline{\H}_{3,g}$, we refer the reader to \cite{deopurkar_patel:sharp_trigonal}.)
\end{rmk}

\begin{rmk}  The reason we use Deopurkar's compactification $\Hdgnt$ as opposed to the usual space of twisted stable maps $\Hdg$ is clear:  one parameter families of branched covers typically have pairs of branch points colliding.  If we did not use Deopurkar's space, we would have to perform base changes every time two branch points collided, which would be an unnecessary complication.
\end{rmk}

\subsection{The locus $\M^{1}_{d,g}$}
We should justify our passage to the space $\Hdgnt$ when considering the slope problem for the loci $\M^{1}_{d,g} \subset \Mg$.  A sweeping family $S \to \Hdgnt$ can automatically be regarded as a sweeping family $S \to \M^{1}_{d,g}$ having the same slope. The converse is also true, which roughly states that a sweeping family in $\M^{1}_{d,g} \subset \Mg$ can be ``lifted''  to a sweeping family in $\o {\H}_{d,g}^{2}$ with the same slope.  In the cases of interest for us, the inequality $d \leq g/2 +1$ holds, so we assume it in the next lemma: 
\begin{lem}\label{transportsweeping}
Let $f \from S \to \M^{1}_{d,g}$ be a sweeping family parametrized by a smooth proper curve $S$, and suppose $d \leq g/2 +1$. Then there exists a sweeping family $f' \from S' \to \Hdgnt$ parametrized by a smooth proper curve $S'$ such that $s(S') = s(S)$. 
\end{lem}
\begin{proof}
To say that $f$ is a sweeping family is to say there exists a proper flat family of curves \[g \from \cS \to B\] and a dominant morphism $m \from \cS \to \M^{1}_{d,g}$ where $g^{-1}(0)$ = $S$ for some point $0 \in B$.  We may assume $B$ is integral and normal.  

Let \[q \from \Mg \to \o{M}_{g}\] be the map to the coarse space.  Then, by abusing notation, we let $m \from \cS \to \o{M}_{g}$ also denote the composite morphism to the coarse space. We  then take the fiber product 

 \begin{equation} 
\xymatrix @C -1.1pc  {
\cS \times_{\o{M}_g} \o{H}_{d,g}^{(2)} \ar[rr]^{m'} \ar[d]_f&&
\o{H}_{d,g}^{(2)} \ar[d]^F \\
\cS \ar[rr]_{m} && \o{M}_{g}}
 \end{equation} 

The inequality $d \leq g/2 + 1$ implies that the forgetful map $F$ is generically finite onto the $d$-gonal locus.  Since $F$ dominates the $d$-gonal locus, we conclude that $f$ is dominant.  Therefore, we may choose a component $Z \subset \cS \times_{\o{M}_g} \o{H}_{d,g}^{2}$ dominating $B$.  

Now we appeal to the fact that $\Hdgnt$ is a ``global quotient'' stack (see \cite{edidin01:_brauer}), that is to say there exists a finite cover $Z' \to Z$ with a morphism $\widetilde{m}' \from Z' \to \o {\H}_{d,g}^{2}$ lifting the morphism $m'$.  

We replace $Z'$ with its normalization and consider the natural surjective map $t \from Z' \to B$.  Since $Z'$ is normal, the general fiber of $\widetilde{m}'$ will be a smooth curve $S'$, and  will map finitely onto the corresponding fiber of the original family $g \from \cS \to B$.  The lemma follows by considering the family of admissible covers parametrized by $S'$, i.e. $s(S') = s(S)$.  
\end{proof}

\subsection{Constructing families of trigonal curves.} \label{BasicTrigonalFamilies} Let $\F$ denote either the Hirzebruch surface $\F_{0}$ or $\F_{1}$ (depending on the parity of the genus), and let $C$ be a smooth trigonal curve of genus $g_{R}$ on $\F$.  

We choose a general pencil $\P^{1}_{t} \subset |\cO_{\F}(C)|$, and let $p \from Y \to \P^{1}_{t}$ be the total space.  Then we have
\begin{prop}\label{trignum}
The intersection numbers of the pencil $p$ are:
\begin{align*}
\lambda \cdot p &= g_{R} \\
\delta \cdot p &= 7g_{R} + 6
\end{align*}
Furthermore, the pencil $p$ only has irreducible nodal singular fibers.
\end{prop}

\begin{proof}
This is a standard computation and application of Bertini's theorem - we omit it. 
\end{proof}

\subsection{Rigidity and Extremality of $\sfM$ for trigonal curves.}
We are ready to prove our first extremality and rigidity result -- the case of the Maroni divisor $\sfM$ in $\o{\H}^{(2)}_{3,g}$. 

\begin{defn}\label{rigid}
A closed subvariety (or substack) $Y \subset X$ is {\sl rigid} if there exists no nonconstant family $\mathcal{\cY} \subset X \times \Spec k[[t]]$, flat over $\Spec k[[t]]$, with $\cY \cap X \times \{0\} = Y$.
\end{defn}

\begin{thm}\label{Mextremal3}
Let $g \geq 4$ be even.  Then the Maroni divisor $\sfM \subset \o{\H}^{(2)}_{3,g}$ is rigid and extremal in the effective cone of divisors $\Eff (\overline{\H}^{(2)}_{3,g}).$
\end{thm}

First we prove the following proposition:

\begin{prop}\label{prop:flexiblethree}
Let $[\P^{1}_{t}] \in A_{1}(\td{\H}_{3,g} \setminus \sfM)$ be the curve class given by a pencil $\P^{1}_{t} \subset |\cO_{\F_{0}}(C)|$ as in section \ref{BasicTrigonalFamilies}. Then there exists a positive integer $m$ with the following property: There exists an irreducible proper curve in the class $m[\P^{1}_{t}]$ contained entirely in $\td{\H}_{3,g} \setminus \sfM $ which passes through any two prescribed points $\alpha_{1}, \alpha_{2} \in \td{\H}_{3,g} \setminus \sfM$.
\end{prop} 

\begin{proof}
In the linear system $|\cO_{\F_{0}}(C)|$, which is a projective space, we remove the codimension two locus $Z$ of curves which are worse-than-nodal. Since this locus has codimension $2$, there exists an integer $m$ with the property that we may join any two points in the complement $|\cO_{\F_{0}}(C)| \setminus Z$ by a complete curve $S \subset |\cO_{\F_{0}}(C)| \setminus Z$ of degree $m$. 

The resulting curve class $[S] \in A_{1}(\td{\H}_{3,g} \setminus \sfM)$ is clearly $m[\P^{1}_{t}]$, and serves our purposes.
\end{proof}

\begin{proof}[Proof of \autoref{Mextremal3}]
By considering the one-parameter family $S$ from the proof of \autoref{prop:flexiblethree}, we see that there exists a complete one parameter family of twisted admissible covers, which we continue to denote by $S \to \o{\H}^{(2)}_{3,g}$, which avoids $\sfM$, and only intersects $\Delta_{\irr}$ and $\sfT$.

Furthermore, $[S]$ is flexible enough to pass through any two points in the complement $\o{\H}^{(2)}_{3,g} \setminus (\sfM \cup \{\text{higher boundary div.}\})$.  Therefore, $\sfM$ cannot be deformed to a divisor $M'$ since $[S] \cdot \sfM = 0$, while $[S] \cdot M' > 0$.

Moreover, if \[\sfM = \sum_{i} a_{i}E_{i}, \,\,\,\,\,  a_{i}>0\] is an effective combination of cycles $E_{i}$ in $\Pic_{\Q}\o{\H}^{(2)}_{3,g}$, none of which are supported on $\sfM$, we deduce: {\sl The $E_{i}$ must be supported on the higher boundary divisors in $\o{\H}^{(2)}_{3,g}$.} 

But if $a_{i} > 0$, we would obtain a nontrivial relation among $\sfM$ and the higher boundary divisors, which is impossible by \autoref{indepMCE}.
\end{proof}

\begin{rmk}
The method of proof of \autoref{Mextremal3} also shows that $\sfM$ is rigid and extremal when viewed as a divisor in $\widetilde{\H}_{3,g}$ or $\o{\H}_{3,g}$. For the former, we use the same family $S$.  For the later, we pass to the admissible replacement of $S$ (base change required).
\end{rmk}

\subsection{Constructing families of tetragonal curves.}\label{TetragonalFamilies}

For any branched cover $[\alpha \from C \to \P^{1}] \in \td{\H}_{4,g}$, the Casnati-Ekedahl resolution is determined by the map 
\begin{equation}\label{CE4}
\pi^{*}\cF(-2) \to \cO_{\P \cE}
\end{equation}
which in turn is determined, up to scalar multiple, by a global section $\eta \in H^{0}(\P^{1}, \dual{\cF}\otimes \Sym^{2}\cE)$.  For a general tetragonal curve $C$, the bundles $\cF$ and $\cE$ are balanced, and we assume they are throughout this section. 

The locus $Z \subset \P(H^{0}(\P^{1}, \dual{\cF}\otimes \Sym^{2}\cE))$ parametrizing worse-than-nodal schemes in $\P \cE$ is has codimension two, as seen by a straightforward parameter count and Bertini's theorem. 

\subsubsection{Constructing families of tetragonal curves I.}  

Take a general linear pencil $\P^{1}_{t} \subset \P(H^{0}(\P^{1}, \dual{\cF}\otimes \Sym^{2}\cE))$, and consider the resulting family $f \from \cC \to \P^{1}_{t}.$ We have the exact analogue of \autoref{prop:flexiblethree}:

\begin{prop}\label{prop:flexiblefour}
Let $[\P^{1}_{t}] \in A_{1}(\td{\H}_{4,g} \setminus \sfM \cup \sfCE)$ be the curve class given by a pencil $\P^{1}_{t} \subset \P(H^{0}(\P^{1}, \dual{\cF}\otimes \Sym^{2}\cE))$ as above. Then there exists a positive integer $m$ with the following property: There exists an irreducible proper curve of class $m[\P^{1}_{t}]$ contained in $\td{\H}_{4,g} \setminus (\sfM \cup \sfCE)$ which passes through any two prescribed points $\alpha_{1}, \alpha_{2} \in \td{\H}_{4,g} \setminus (\sfM \cup \sfCE)$.  
\end{prop} 

\begin{proof} The proof is the same as the proof of \autoref{prop:flexiblethree}.
\end{proof}

\begin{thm}\label{thm:Extremal4}
The effective divisors $\sfM$ and $\sfCE$ are rigid and extremal in the effective cone of divisors $\Eff (\overline{\H}^{(2)}_{4,g})$ with the exception of $\H_{4,3}$ where the divisor $\sfCE$ has two components $\sfA$ and $\sfB$, both of which are extremal and rigid.
\end{thm}

\begin{proof}
We prove the theorem for $\sfM$; the argument for $\sfCE$ is the same, up to the exception. Suppose we could write $$\sfM = \sum_{i} a_{i}E_{i}$$ where $E_{i} \subset \overline{\H}^{(2)}_{4,g}$ are irreducible effective divisors, none supported on $\sfM$ and $a_{i} > 0$. 

The curve class $m[\P^{1}_{t}]$ clearly has zero intersection with $\sfM$, and therefore must have zero intersection with each irreducible divisor $E_{i}$. \autoref{prop:flexiblefour} then implies that the divisors $E_{i}$ must be supported possibly on $\sfCE$ or the higher boundary divisors. The same proposition also implies that $\sfM$ is rigid: a deformation $M'$ would have to intersect the class $m[\P^{1}]$ nontrivially.

We conclude that the relation $\sfM = \sum a_{i}E_{i}$ then gives a linear dependence among the boundary divisors of $\overline{\H}^{(2)}_{4,g}$ and $\sfCE$. This is impossible by \autoref{indepMCE}.

We now address the exception. The central observation is that the Casnati-Ekedahl divisor $\sfCE$ is {\sl reducible}precisely when $g=3$. In fact, when $g=3$, the Casnati-Ekedahl divisor $\sfCE$ contains $\sfM$ as an irreducible component, and has one residual component. 

The irreducibility of the Casnati-Ekedahl divisor $\sfCE$ in all other cases is proved in \cite{dp:pic_345}[Proposition~4.7] and the reducible example is discussed in \cite{dp:pic_345}[Example~4.8]. 

With this caution, the rest of the argument for $\sfCE$ goes through as in the above argument for $\sfM$.
\end{proof}

\subsubsection{Constructing families of tetragonal curves II.} The Casnati-Ekedahl resolution \ref{CE4} reflects a classical fact about tetragonal curves: A genus $g_R$ cover $[\alpha \colon C \longrightarrow {\bf P}^1]$  rests in its relative canonical embedding $\pi \colon {\bf P}\mathcal{\cE} \longrightarrow {\bf P}^1$ as a complete intersection of two relative conic divisors $Q_u$ and $Q_v$ in the linear systems $|2\zeta -uf|$ and $|2\zeta - vf|$, respectively. Here $\zeta$ and $f$ denote the hyperplane and fiber classes of $\P \cE$.

  The integers $u$ and $v$ are such that $\cF_{\alpha} = \cO(u)\oplus \cO(v)$. Moreover by \autoref{relatecFcE}, $u+v = g_{R}+3.$ We assume that $u\leq v$, so that generically in $\H_{4,g_{R}}$,  $v = \lceil{\frac{g_R+3}{2}\rceil}$.  For a general cover $\alpha$, the line bundle $\cO_{\P\cE}(2\zeta - vf)$ is  globally generated on $\bf{P}\cE$, and $\cO_{\P\cE}(2\zeta - uf)$ is  very ample.  We assume from here on that the bundles $\cE$ and $\cF$ are balanced.

\subsubsection{}
Fix a smooth surface $\F \in |2\zeta - vf|$, and consider a general pencil ${\bf P}^1_t \subset |2\zeta-uf|$ on $\F$.  Let $p \colon Y \longrightarrow {\bf P}^1_t$ denote the total space of the pencil. So $p$ is a family of tetragonal curves, and $Y$ is the blow up of $\F$ at the base locus of the pencil. We  compute the invariants $\delta \cdot p$ and $\lambda \cdot p$ using standard methods. 

\begin{prop}\label{basictetragonal} For a general pencil $p \colon Y \longrightarrow {\bf P}^1_t$ of the type described above,
\begin{eqnarray*}
 \lambda \cdot p &=& g_R,\\
 \delta \cdot p &=& v + 6g_R + 6.\\
 \end{eqnarray*} 
 Furthermore, every singular fiber of $p$ is simply nodal.
\end{prop}
\begin{proof}
The last statement is a straightforward application of Bertini's theorem. 

Since $\chi(\cO_{\sfY}) = 1$ ($Y$ is rational), it follows that \[\lambda \cdot p = \chi(\cO_{\sfY}) - \chi(\cO_{{\bf P}^1_t})\chi(\cO_{C_{\gen}}) = g_R,\] where $C_{\gen}$ denotes the curve in the pencil. 

We now compute $\delta \cdot p$.  We recall the following fact, easily seen using jet bundles: {\sl If $L$ is an ample line bundle on a surface $S$, then a general pencil in $|L|$ has $c_{2}(L + L \otimes \Omega_{S})$ singular elements.}

 Letting $L := \cO_\F (2\zeta-uf)$, get:
\begin{eqnarray*}
\delta \cdot p &=& c_2(L + L \otimes \Omega_\F) = c_1(L) \cdot c_1(L \otimes \Omega_\F)+ c_2(L \otimes \Omega_\F)\\
&=&c_1(L) [K_\F + 2c_1(L)] + c_2(\Omega_\F) + c_1(L) K_\F + c_1^2(L)\\
&=& 3c_1^2(L) + 2c_1(L) \cdot K_\F + c_2(\Omega_{\F})\\
&=& 2(2g_R-2) + c_1^2(L) +  c_2(\Omega_{\F}). \hspace{1 in} 
\end{eqnarray*}

We compute $c_{1}^{2}(L)$:
\begin{eqnarray*}
c_1^2(L) &=& (2\zeta - uf)^2(2\zeta-vf) = [4\zeta^2 - 4u\zeta f][2\zeta - vf]\\
 &=& 8\zeta^3 - 8u - 4v = 8(u+v)-8u-4v = 4v,\\
 \end{eqnarray*}
 which means 
 \begin{eqnarray*}
 \delta \cdot p  &=& 2(2g_R-2) +4v + c_2(\Omega_{\F}).
\end{eqnarray*}
So it suffices to compute $c_2(\Omega_{\F})$. In order to compute this, we use the exact sequence of sheaves on $ {\bf P}\cE$: \[0 \to \cO_{\F}(-2\zeta +vf) \to \Omega_{{\bf P}\cE}|_{\F} \to \Omega_{\F} \to 0\] Put $M :=  \cO_{\F}(-2\zeta +vf)$. Then the Whitney sum formula says:
\begin{eqnarray*}
&&c(\Omega_{\F}) \cdot (1 + c_1(M)t) = c(\Omega_{{\bf P}\cE}|_{\F})\\
\end{eqnarray*} and so we get:
\begin{eqnarray*}
 c_2(\Omega_{\F}) = c_1^2(M) - K_{{\bf P}\cE}|_{\F} \cdot c_1(M) + c_2(\Omega_{{\bf P}\cE})|_{\F}.
\end{eqnarray*}

The canonical class of ${\bf P}\cE$, $K_{{\bf P}\cE}$, is $-3\zeta + (u+v - 2)f$, so the first part of the right hand expression becomes
\begin{eqnarray*}
  -4\zeta^3 + 4u +4v - 8 = -8 .
\end{eqnarray*}
In order to compute  $c_2(\Omega_{{\bf P}\cE})|_{\F}$ , we consider the Euler sequence \[0 \to \Omega_{\pi} \to \pi^{*}\cE(-1) \to \cO \to 0\] and the relative cotangent sequence \[0 \to \cO(-2f) \to \Omega_{{\bf P}\cE} \to  \Omega_{\pi} \to 0.\]
We find that: 
\begin{eqnarray*}
 c_2(\Omega_{{\bf P}\cE})|_{\F} = 3v + 6u - 4g_R.
\end{eqnarray*}
and  \[c_2(\Omega_{\F}) = 3v + 6u - 4g_R - 8.\] 

The proposition now follows by putting these calculations together.
\end{proof}

\subsection{Constructing families of pentagonal curves.} 
The middle map
\begin{equation}
m \from \pi^{*}\cN_{2}(-3) \to \pi^{*}\cN_{1}(-2)
\end{equation}
in the Casnati-Ekedahl resolution of a pentagonal curve $C \subset \P(\cE)$ is well-known \cite{casnati:five} to be  ``skew-symmetric" in the following sense. By duality, $\cN_{2} = \dual{\cN_{1}} \otimes \det \cE$ where $\cN_{1}$ is simply the Casnati-Ekedahl bundle which we have been calling $\cF$.   A homomorphism from $\pi^{*}\dual{\cF}\otimes \det{\cE}(-3)$ to $\pi^{*}\cF(-2)$  can be thought of as an element of the vector space \[H^{0}(\P^{1}, \cF\otimes \cF \otimes \cE \otimes \det \dual{\cE}),\]  and a skew symmetric element is by definition an element of the subspace \[H^{0}(\P^{1}, \bigwedge^{2}\cF \otimes \cE \otimes \det \dual{\cE}).\]

The homomorphism $\sfM$ can be regarded as an element (which we continue to call $\sfM$) in \[H^{0}(\P^{1}, Hom(\bigwedge^{2}\dual{\cF}, \cE \otimes \det \dual{\cE})).\]   It is also easy to check (e.g. for all points $p \in \P^{1}$) that \[m \from \bigwedge^{2}\dual{\cF} \to \cE \otimes \det \dual{\cE}\] is surjective. Let 
\begin{equation}\label{pentseq}
0 \to \cK \to \bigwedge^{2}\dual{\cF} \to \cE \otimes \det \dual{\cE} \to 0
\end{equation}  be the resulting exact sequence of vector bundles.  This sequence gives rise to a chain of inclusions $C \subset \P(\cE) = \P(\cE \otimes \det \dual{\cE}) \subset \P(\bigwedge^{2}\dual{\cF})$, where the first inclusion is the relative canonical embedding of $\alpha \from C \to \P^{1}.$

\subsubsection{Constructing families of pentagonal curves I.}  

Our first  family of pentagonal curves is a generic pencil $$\P^{1}_{t} \subset \P(H^{0}(\P^{1}, \bigwedge^{2}\cF \otimes \cE \otimes \det \dual{\cE})).$$ 

Let $f \from \cC \to \P^{1}_{t}$ be the resulting family of pentagonal curves.  A parameter count shows that the closed locus $Z \subset \P(H^{0}(\P^{1}, \bigwedge^{2}\cF \otimes \cE \otimes \det \dual{\cE}))$  parametrizing subschemes which are worse-than-nodal has codimension two.  This leads to the following proposition analogous to \autoref{prop:flexiblethree} and \autoref{prop:flexiblefour}:

\begin{prop}\label{prop:flexiblefive}
Let $[\P^{1}_{t}] \in A_{1}(\td{\H}_{5,g} \setminus \sfM \cup \sfCE)$ be the curve class given by a pencil $\P^{1}_{t} \subset \P(H^{0}(\P^{1}, \bigwedge^{2}\cF \otimes \cE \otimes \det \dual{\cE}))$ as above. Then there exists a positive integer $m$ with the following property: There exists an irreducible proper curve of class $m[\P^{1}_{t}]$ contained in $\td{\H}_{4,g} \setminus (\sfM \cup \sfCE)$ which passes through any two prescribed points $\alpha_{1}, \alpha_{2} \in \td{\H}_{5,g} \setminus (\sfM \cup \sfCE)$.  
\end{prop} 

\begin{proof} The proof goes exactly like the proof of \autoref{prop:flexiblefour}.
\end{proof}

\begin{thm}\label{thm:Extremal5}
The effective divisors $\sfM$ and $\sfCE$ are rigid and span extremal rays in the effective cone of divisors $\Eff (\overline{\H}^{(2)}_{5,g}).$
\end{thm}

\begin{proof}
The proof is exactly like the proof of \autoref{thm:Extremal4} once we know the irreducibility of $\sfM$ and $\sfCE$. The former is found in part $(4)$ of \autoref{thm:maroniloci}, while the irreducibility is proved in \cite{dp:pic_345}[Proposition~5.2].
\end{proof}

\subsubsection{Constructing families of pentagonal curves II.}\label{pentagonalconstruction}
The geometric description of the embedding $i \from C \into \P(\bigwedge^{2}\dual{\cF})$ is particularly simple and easy to work with, and allows for a second useful construction of families of pentagonal curves.  

Inside the $\P^{9}$-bundle $\P(\bigwedge^{2}\dual{\cF})$ lies is the Grassmannian bundle $\bG \subset \P(\bigwedge^{2}\dual{\cF})$ which parametrizes the decomposable alternating tensors $a \wedge b$.  $\bG$ restricts to the Grassmannian $\bG(1,4) \subset \P^{9}$ in every fiber of the projective bundle $\P(\bigwedge^{2}\dual{\cF})$.   The curve $C$ is simply the intersection of the sub scroll $\P(\cE \otimes \det \dual{\cE})$ with $\bG$ in $\P(\bigwedge^{2}\dual{\cF})$. 

\subsubsection{}
We rephrase this more conveniently for our purposes.  Let $\zeta$ denote the divisor class on $\bG$ associated to the restriction of the natural $\cO(1)$ on $\P(\bigwedge^{2}\dual{\cF})$, and let $f$ denote the fiber of the natural projection $\varphi \from \bG \to \P^{1}$.   Since the kernel $\cK$ in \eqref{pentseq} is a rank $6$ vector bundle on $\P^{1}$, it splits as 
\begin{equation}\label{K}
\cK = \cO_{\P^{1}}(-k_{1}) \oplus ... \oplus \cO_{\P^{1}}(-k_{6}).
\end{equation}  
What this means is that the curve $C \subset \bG$, being the intersection of $\P(\cE \otimes \det \dual{\cE})$ with $\bG$, is a complete intersection of six divisors $H_{i}$  $(i = 1, ... ,  6)$ on $\bG$ with divisor classes $[H_{i}] = \zeta + k_{i}f$.  

From sequence \eqref{pentseq} (and by recalling from \autoref{relatecFcE} that $\deg \cF = 2 \deg \cE = 2(g+4)$), we conclude  \[\deg \cK = - \sum k_{i} = -5(g+4).\] Let us order the $k_{i}$ so  $k_{1} \geq ... \geq k_{6}$.   Then $H_{1}$ is the largest effective divisor among the hyperplane sections $H_{i}$, meaning $H_{1} - H_{j}$ is effective for all $j \geq 2$.  

\subsubsection{}
We consider the surface \[\F := \bigcap_{i = 2}^{6} H_{i}.\] $\F$ will be the surface on which we construct families as in section \ref{construction}.  $\F$ is easily seen to be a genus $1$ fibration over $\P^{1}$ -- we let $\varphi \from \F \to \P^{1} $ denote the natural projection.  Our pencils will lie in the linear system $|H_{1}|$, considered as a system of divisors on $\F$.  We also assume that $\F$ is a smooth surface - this can be verified by Bertini's theorem in all cases we consider.
 
\begin{prop}\label{pentnum}
Let $\P^{1}_{t} \subset |H_{1}|$ be a general pencil of curves on $\F$.  Assume the base locus of $\P^{1}_{t}$ is simple, and that all curves parametrized are at-worst-nodal.  Let $p \from X \to \P^{1}_{t}$ denote the total space of the pencil. 
Then
\begin{align*}
\lambda \cdot {p} & = 2g+3-k_{1},\\
\delta \cdot p & =  13g + 32 - 7k_{1}.
\end{align*}
\end{prop}
\begin{proof}
We first compute the relevant numerical invariants of the surface $\F$.  Adjunction on the relative Grassmannian $\bG \subset \P (\bigwedge^{2} \dual{\cF})$ gives the equality \[\deg K_{\F} = (\sum_{i=2}^{6}(\zeta + k_{i}f) + K_{\bG})|_{\F},\]
so we must compute the canonical class $K_{\bG}$. 

The relative tangent bundle for $g \colon \bG \longrightarrow {\bf P}^1$ is \[T_g = Hom(S^3,Q^2) = \dual{(S^3)} \otimes Q^2.\] Here $S^{3}$ and $Q^{2}$ are the universal sub and quotient bundles parametrized by the relative Grassmannian $\bG$. In other words, $c_{1}(Q^{2}) = \zeta$, while $c_{1}(S^{3}) = -\zeta + g^{*}c_{1}(\dual{\cF})$. Therefore, 
\begin{align*}
K_{\bG} + 2f &= -c_1 ((S^3)^ {\vee} \otimes Q^2)\\
& = - [3\zeta + 2(\zeta - g^{*}c_{1}\dual{\cF})]\\
 & = -5\zeta + 2 c_{1}\dual{\cF} \cdot f\\
\end{align*}

Altogether, the canonical class of $\bG$ is 
\begin{equation}\label{KG}
 K_{\bG} = -5\zeta + (2 c_{1}\dual{\cF} - 2)\cdot f.
\end{equation}

From this, we conclude that the canonical class $K_{\F}$ is 
\begin{align*}
K_{\F} &= (H_{2} + H_{3}+...+H_{6} + K_{\bG}) \\
&= (\sum_{i=2}^{6}k_{i} + 2 c_{1}\dual{\cF} - 2)\cdot f\\
&= (g+2-k_{1})\cdot f
\end{align*}
where we have suppressed the ``restriction" symbols.

We must also compute relevant intersection products on $\bG$.  We first notice that on $\bG$,  \[\zeta^6\cdot f = 5.\] Indeed, the Grassmannian $\bG(1,4) \subset \P^{9}$ has degree $5$. Secondly we note that \[\zeta^7 = a\cdot c_{1}\dual{\cF} + b\] for some fixed constants $a,b$ which do not depend on $\mathcal{F}$. This follows from the fact that $\zeta^7 - (\zeta + l\cdot f)^7$ is a linear function in $l$.  It easily follows that $b=0$ and $a=14$, so that: 
\[\zeta^7 = 14c_{1}\dual{\cF}.\]

On $\F$, we consider the linear series $|\zeta + k_1 f|$ and a general pencil $p \subset |\zeta + k_1 f|$. The pencil $p$ has $B = (\zeta + k_1 f)^2$ basepoints.  This number can now be calculated:

\begin{align*}
B &= (\zeta + k_{1} f)^2\cdot [\F] = (\zeta + k_{1} f) \cdot \prod_{i=1}^{6}(\zeta + k_{i} f)\\
&= \zeta^7+ (2 k_1+k_2+ ... +k_6)\zeta^6\cdot f\\
& = \zeta^7+5(5(g+4) + k_{1})\\
&= 14c_{1}\dual{\cF} + 25(g+4) + 5k_{1}\\
& = 5k_{1}-3(g+4).
\end{align*}
Here we have used $c_{1}(\cF) = 2c_{1}(\cE) = 2(g+4)$ and $\sum_{i}k_{i} = 5(g+4)$.
Since $\F$ is assumed to be smooth, we can determine the topological Euler characteristic $\chi_{\top}(\F)$ by the Noether formula:
\begin{equation}\label{noether}
\chi(\mathcal{O}_\F) = \frac{K_{\F}^2 + \chi_{\top}(\F)}{12} = \frac{\chi_{\top}(\F)}{12}.
\end{equation}

On the other hand, let $g \from \F \to \P^{1}$ be the projection which realizes $\F$ as a family of genus $1$ curves. (These are degree five elliptic normal curves.) Then Grothendieck-Riemann-Roch applied to the family $g$ gives:
\begin{equation}
 \lambda_{g} =  \chi(\mathcal{O}_\F) - \chi(\mathcal{O}_{{\bf P}^1})\chi(\mathcal{O}_{E_{gen}}). 
\end{equation}
which implies
$$\chi(\mathcal{O}_\F) = c_{1}\mathcal{E} - k_{1}$$
and thus
$$ \chi_{\top}(\F) = 12(c_{1}\mathcal{E} - k_{1}).$$

 The total space of the pencil $X$ is the blow up of $\F$ at the $B = 5m_6-3c_1\mathcal{E}$ basepoints, therefore: $$\chi_{\top}(X) = 12(c_{1}\mathcal{E} - k_{1}) + 5k_{1} -3c_1\mathcal{E} = 9c_{1}\mathcal{E} - 7k_{1}$$

This implies  $$\delta \cdot p = 9c_{1}\mathcal{E} - 7k_{1} - 2(2-2g) = 13g - 7k_{1} + 32$$ Furthermore, we know  $$\chi(\mathcal{O}_{\F}) = \chi(\mathcal{O}_{X}) =  c_{1}\mathcal{E} - k_{1}$$ which means $$\lambda \cdot p = 2g - k_{1} +3.$$

\end{proof}

\section{The slope problem.}\label{slopeproblem}
Our objective in this section is to prove \autoref{Yeffective} and hence \autoref{slope45}. The basic strategy is simple enough, although implementing it will involve some care.  

We want to show the divisor class $\sfY$, which is supported on the higher boundary, is an {\sl effective} combination of higher boundary divisors of $\Hdgnt$. 

Given any higher boundary divisor $\Delta$, we will let $c(\Delta, \sfY)$ denote the coefficient of $\Delta$ in the expression of $\sfY$. This is well-defined thanks to \autoref{independence}.

Our task is then: $$\text{\sl To prove the inequality $c(\Delta, \sfY) \geq 0$ for every higher boundary divisor $\Delta$.}$$  We will accomplish this by a sort of induction on the complexity of the dual graph $\Gamma_{\Delta}$ , similar to that found in the proof of \autoref{independence}.  The induction will be provided to us by carefully selected partial pencils $p$ which lie entirely within $\Delta$ and intersect (very few) other boundary divisors $\Delta'$ for which we already know $c(\Delta', \sfY) \geq 0$. From these ``prior'' inequalities, we are able to deduce the inequality for $\Delta$.

For notational convenience, the symbol ``$\Delta$'' will stand for ``the higher boundary divisor currently under investigation.''  It will change per section.  All other boundary divisors arising in the analysis of $c(\Delta, \sfY)$ will be represented by the symbol $\Delta$ with some decorations or subscripts.

The overall method is somewhat tedious and repetitive. Therefore, we have adopted the following plan:  We provide many details in the degree three case, showing the reader exactly how the inductive procedure goes. 

In the degree four and five cases, we skip the demonstration of details which follow from methods found in the degree $3$ case, and focus only on the new challenges which arise when implementing the inductive procedure. 

Throughout this entire section, we assume we are in the natural setting of \autoref{slope45}: $g$ and $d$ are such that $\sfM$ and $\sfCE$ exist as divisors.

\subsection{Degree three analysis.}

\subsubsection{$2$-vertex divisors I.}
Impose three distinct points on a ruling line $F \subset \F$ to be in the base locus of a pencil $p$ as in section \ref{construction}.  

We then attach an unchanging cover $\alpha_{L} \from C_{L} \to P_{L}$ along the three gluing sections to end up with a partial pencil which we continue to denote by $p$.

The partial pencil $p$ lives entirely within a boundary divisor $\Delta$ having dual graph 
$\Gamma_{\Delta}$: 
\[ \xygraph{
!{<0cm,0cm>;<1cm,0cm>:<0cm,1cm>::}
!{(0,0) }*{\bullet}="l"!{(-.3,.2)}{\scriptstyle v_L}!{(-.3,-.2)}{\scriptstyle g_L}
!{(2,0) }*{\bullet}="r"!{(2.3,.2)}{\scriptstyle v_R}!{(2.3,-.2)}{\scriptstyle g_R}
 "l"-@/^.3cm/"r" 
"l"-"r" 
"l"-@/_.3cm/"r" 
}  \] 
The pencil has a unique split fiber, corresponding to an intersection with a new divisor $\Delta_{\spl}$ with dual graph $\Gamma_{\Delta_{\spl}}$:
\[ \xygraph{
!{<0cm,0cm>;<1cm,0cm>:<0cm,1cm>::}
!{(0,0) }*{\bullet}="l"!{(-.3,.2)}{\scriptstyle v'_L}!{(-.3,-.21)}{\scriptstyle g_L+2}
!{(2,0) }*{\bullet}="r"!{(2.3,.2)}{\scriptstyle v'_R}!{(2.3,-.2)}{\scriptstyle g_R - 2}
 "l"-@/^.3cm/"r" 
"l"-"r" 
"l"-@/_.3cm/"r" 
}  \]
As far as numerics go, we have: 
\begin{prop}\label{basictrig}
The partial pencil $p$ constructed above has the following intersection numbers with various divisors: 
\begin{align*}
\lambda \cdot p &= g_{R}\\
\delta \cdot p &= 7g_{R}+3\\
\Delta \cdot p &= -1\\
\Delta_{\spl} \cdot p &= 1\\
\sfX \cdot p &\geq 0
\end{align*}
\end{prop}
\begin{proof}
We explain $\delta \cdot p$ and $\sfX \cdot p$ - the rest are clear.  The original pencil $\P^{1}_{t} \subset |\cO_{\F}(C)|$ had $7g_{R}+6$ nodes occurring in its total space.  We must then add the contribution from the self intersection of the three gluing sections, each of which is $-1$.  This explains the total $7g_{R} +3$.

The inequality $\sfX \cdot p \geq 0$ follows form the observation that the deformations of the pencil $p$ sweep out the entire boundary divisor $\Delta$ in which it lies. Since $\sfM$ and $\sfCE$ are defined by taking closures, no higher boundary divisor can be a component of either.  Therefore, $p$ must intersect both $\sfM$ and $\sfCE$ nonnegatively.
\end{proof} 

\begin{rmk}
The skeptical reader may wonder why the numbers occuring for the partial pencil $p$ in \autoref{basictrig} differ from those of a ``standard pencil'' of genus $g$ as in \autoref{trignum}.  The reason is that the partial pencil $p$ above is not linearly equivalent to the standard pencil, as is evident by the fact that it intersects higher boundary divisors.
\end{rmk}

\subsubsection{Inequalities I.} Let $\Delta$ and $p$ be as in \autoref{basictrig}.  

\begin{lem}
We have the following inequality:
\begin{equation}\label{4.1}
c(\Delta, {\sfY}) \geq c(\Delta_{\spl}, {\sfY}) + 3g-6g_R.
\end{equation}
\end{lem} 

\begin{proof}
We intersect the partial pencil $p$ with the equality $\sfX = a\lambda - b\delta - {\sfY}$ and use \autoref{basictrig}. 
\end{proof}

\begin{cor}\label{induct1}
Keeping the same setting as above, assume $g_R \leq g_L$. Then $c(\Delta_{\spl}, {\sfY}) \geq 0$ implies $c(\Delta, {\sfY}) \geq 0 $.  
\end{cor}

\begin{proof}
This follows from the previous lemma and the observation that $g_{R}+g_{L} = g-2$.
\end{proof}

\begin{prop}\label{unramified1}
Keeping the same setting as the previous corollary, we have: 
$$c(\Delta, \sfY) \geq 0.$$
\end{prop}

\begin{proof}
Suppose $g_{R}$ is even. Then by repeatedly using corollary \ref{induct1}, we see: {\sl In order to deduce $c(\Delta, \sfY) \geq 0$ it suffices to show $c(\Delta_0, {\sfY}) \geq 0$ for the boundary divisor $\Delta_{0}$ which has the dual graph $\Gamma_{\Delta_{0}}$}: 
 \[ \xygraph{
!{<0cm,0cm>;<1cm,0cm>:<0cm,1cm>::}
!{(0,0) }*{\bullet}="l"!{(-.3,.2)}{\scriptstyle v_L}!{(-.3,-.21)}{\scriptstyle g-2}
!{(2,0) }*{\bullet}="r"!{(2.3,.2)}{\scriptstyle v_R}!{(2.3,-.2)}{\scriptstyle 0}
 "l"-@/^.3cm/"r" 
"l"-"r" 
"l"-@/_.3cm/"r" 
}  \]

To do this, we use the same type of pencil $p$ as in \autoref{basictrig}, varying the genus $0$ component.   Fortunately, as is the case for rational partial pencils, the split fiber does not contribute any intersection with further boundary components, because the residual curve $R_0$ is a disjoint union of $(1,0)$ curves on $\F = \P^{1} \times \P^{1}$ which are blown down as in section \ref{ppencils}.  

Therefore, the resulting rational partial pencil family has the following intersection numbers: 
\begin{align*}
\lambda \cdot p &= 0\\
\delta \cdot p &= 3\\
\Delta_0 \cdot p &= -1\\
{\sfX} \cdot p &\geq 0
\end{align*}

By intersecting with $\sfX = a\lambda - b\delta - {\sfY}$ we deduce $c(\Delta_0, {\sfY}) \geq 0$. 

The analysis for $g_R$ odd is completely analogous.  We simply need to show that $c(\Delta_{1}, {\sfY}) \geq 0$ where $\Delta_{1}$ has an elliptic curve as the right hand component. Again, we just use a partial pencil family on on the surface $\F_{1}$. We leave the details to the reader. 
\end{proof}

\subsubsection{$2$-vertex divisors II.}
If we consider modified versions of the partial pencil $p$ in the previous section, we obtain slightly different numbers.  Suppose we impose a subscheme of type $\{2p_{1},p_{2}\}$ in a ruling line $F \subset \F$ to be in the base locus of $\P^{1}_{t} \subset |\cO_{\F}(C)|$ as in section \ref{modified}. Then we will have two special fibers: the split fiber, corresponding to interaction with a divisor $\Delta_{\spl}$, and the reduced ramification fiber corresponding to intersection with a divisor $\Delta_{\ram}$. We denote the resulting partial pencil by $p'$ to distinguish it from the pencil $p$ from proposition \autoref{basictrig}, and let $\Delta$ denote the boundary divisor which contains $p'$.  

We have:
\begin{prop}\label{trig2}
The partial pencil $p'$ has the following intersection numbers: 
\begin{align*}
\lambda \cdot p' &= g_{R}\\
\delta \cdot p' &= 7g_{R}+4\\
\Delta \cdot p' &= -1\\
\Delta_{\spl} \cdot p' &= 1\\
\Delta_{\ram} \cdot p' &= 1\\
\sfX \cdot p' &\geq 0
\end{align*}
\end{prop}
\begin{proof}
The only number needing explanation is $\delta \cdot p'$.  We must subtract $2$ from $7g_{R}+6$ because there are now two gluing sections in the partial pencil, each with self intersection $-1$.
\end{proof}

\subsubsection{Inequalities II.} Let $\Delta$ be a simply ramified $2$-vertex divisor with dual graph $\Gamma_{\Delta}$: 
\[ \xygraph{
!{<0cm,0cm>;<1cm,0cm>:<0cm,1cm>::}
!{(0,0) }*{\bullet}="l"!{(-.3,.2)}{\scriptstyle v_L}!{(-.3,-.2)}{\scriptstyle g_L}
!{(2,0) }*{\bullet}="r"!{(2.3,.2)}{\scriptstyle v_R}!{(2.3,-.2)}{\scriptstyle g_R}
 "l"-@/^.3cm/"r"  
"l"-@/_.2cm/"r"_{\scriptscriptstyle2}
}  \]
 As before, we let $g_R, g_L$ denote the genera of $v_R, v_L$ and assume that $g_R \leq g_L$. We consider the partial pencil $p'$ whose intersection numbers are recorded in \autoref{trig2}.   

\begin{lem}
We have the following inequality: 
\begin{equation}
 c(\Delta, {\sfY}) \geq c(\Delta_{\spl}, {\sfY}) + c(\Delta_{\ram}, {\sfY}) + 4g-6g_R
 \end{equation}
\end{lem}

\begin{proof}
We intersect $p'$ with $\sfX = a\lambda - b\delta - {\sfY}$ and use \autoref{trig2}.
\end{proof}

\begin{cor}\label{induct2}
Keep the setting of the previous lemma. If $c(\Delta_{\spl}, {\sfY}) \geq 0$ and $c(\Delta_{\ram}, {\sfY}) \geq 0$ , then $c(\Delta, {\sfY}) \geq 0 $.
\end{cor}

\begin{prop}
Keeping the same setting as the previous corollary, we have: 
$$c(\Delta, \sfY) \geq 0.$$
\end{prop}

\begin{proof}
The family $p'$ has a split element corresponding to an intersection with the divisor $\Delta_{\spl}$ which has dual graph:

\[ \xygraph{
!{<0cm,0cm>;<1cm,0cm>:<0cm,1cm>::}
!{(0,0) }*{\bullet}="l"!{(-.3,.2)}{\scriptstyle v'_L}!{(-.3,-.21)}{\scriptstyle g_L+1}
!{(2,0) }*{\bullet}="r"!{(2.3,.2)}{\scriptstyle v'_R}!{(2.3,-.2)}{\scriptstyle g_R-2}
 "l"-@/^.3cm/"r" 
"l"-"r" 
"l"-@/_.3cm/"r" 
}  \]The family $p'$ also has an intersection with an {\sl unramified} boundary divisor $\Delta_{\ram}$ which has dual graph $\Gamma_{\Delta_{\ram}}$: \[ \xygraph{
!{<0cm,0cm>;<1cm,0cm>:<0cm,1cm>::}
!{(0,0) }*{\bullet}="l"!{(-.3,.2)}{\scriptstyle v''_L}!{(-.3,-.21)}{\scriptstyle g_L}
!{(2,0) }*{\bullet}="r"!{(2.3,.2)}{\scriptstyle v''_R}!{(2.3,-.2)}{\scriptstyle g_R-1}
 "l"-@/^.3cm/"r" 
"l"-"r" 
"l"-@/_.3cm/"r" 
}  \]
This intersection comes from the reduced ramification fiber.

  We already know from \autoref{unramified1} that $c(\Delta_{\spl}, {\sfY})  \geq 0$ and $c(\Delta_{\ram}, {\sfY}) \geq 0$. We conclude by \autoref{induct2} 

  \end{proof}

\subsubsection{$2$-vertex divisors III.}
Now let us consider the case where a scheme $\{3p_{1}\}$ is in the base locus of $\P^{1}_{t} \subset |\cO_{\F}(C)|$, as in section \ref{modified2}. 

We have two special fibers corresponding to intersections with divisors $\Delta_{\spl}$ and $\Delta_{\ram}$. We let $p''$ denote the resulting partial pencil.

\begin{prop}\label{trig3}
The second modified partial pencil $p''$ has the following intersection numbers: 
\begin{align*}
\lambda \cdot p'' &= g_{R}\\
\delta \cdot p'' &= 7g_{R}+5\\
\Delta \cdot p'' &= -1\\
\Delta_{\spl} \cdot p'' &= 1\\
\Delta_{\ram} \cdot p'' &= 1\\
\sfX \cdot p'' &\geq 0
\end{align*}
\end{prop}
\begin{proof}
Similar to \autoref{trig3}.
\end{proof}

\subsubsection{Inequalities III.} Let $\Delta$ be a triply ramified, $2$-vertex divisor with dual graph $\Gamma_\Delta$: \[   \xygraph{
!{<0cm,0cm>;<1cm,0cm>:<0cm,1cm>::}
!{(0,0) }*{\bullet}="l"!{(-.3,.2)}{\scriptstyle v_L}!{(-.3,-.2)}{\scriptstyle g_L}
!{(2,0) }*{\bullet}="r"!{(2.3,.2)}{\scriptstyle v_R}!{(2.3,-.2)}{\scriptstyle g_R}
"l"-"r" ^{3}
}  \]  
We consider the partial pencil $p''$ constructed in the previous section.   

\begin{lem}
We have the following inequality: 
\begin{equation}
c(\Delta, {\sfY}) \geq c(\Delta_{\spl}, {\sfY}) + c(\Delta_{\ram}, {\sfY}) + 6g-5g_R
\end{equation}
\end{lem}

\begin{proof}
We intersect $p''$ with $\sfX = a\lambda - b\delta - {\sfY}$ and use \autoref{trig3}.
\end{proof}

\begin{cor}\label{induct3}
Keep the setting of the previous lemma. If $c(\Delta_{\spl}, {\sfY}) \geq 0$ and $c(\Delta_{\ram}, {\sfY}) \geq 0,$  then $c(\Delta, {\sfY}) \geq 0 $.
\end{cor}

\begin{proof}
Clear.
\end{proof}

\begin{prop}\label{induct3}
Keeping the same setting as the previous corollary, we have: 
$$c(\Delta, \sfY) \geq 0.$$
\end{prop}

\begin{proof}
The split fiber in $p''$ provides an intersection with the divisor $\Delta_{\spl}$ whose dual graph looks like:
 $\Gamma_{\Delta_{\spl}}$: \[ \xygraph{
!{<0cm,0cm>;<1cm,0cm>:<0cm,1cm>::}
!{(0,0) }*{\bullet}="l"!{(-.3,.2)}{\scriptstyle v'_L}!{(-.3,-.21)}{\scriptstyle g_L}
!{(2,0) }*{\bullet}="r"!{(2.3,.2)}{\scriptstyle v'_R}!{(2.3,-.2)}{\scriptstyle g_R - 2}
 "l"-@/^.3cm/"r" 
"l"-"r" 
"l"-@/_.3cm/"r" 
}  \]  The reduced ramification fiber provides an intersection with the {\sl simply ramified} boundary divisor $\Delta_{\ram}$ which has dual graph $\Gamma_{\Delta_{\ram}}$:\[  \xygraph{
!{<0cm,0cm>;<1cm,0cm>:<0cm,1cm>::}
!{(0,0) }*{\bullet}="l"!{(-.3,.2)}{\scriptstyle v''_L}!{(-.3,-.2)}{\scriptstyle g_L}
!{(2,0) }*{\bullet}="r"!{(2.3,.2)}{\scriptstyle v''_R}!{(2.3,-.2)}{\scriptstyle g_R-1}
 "l"-@/^.3cm/"r"  
"l"-@/_.2cm/"r"_{\scriptscriptstyle2}
}  \]

These two dual graphs were dealt with previously.  So we already know $c(\Delta_{\spl}, \sfY) \geq 0$ and $c(\Delta_{\ram}, \sfY) \geq 0$. Now apply \autoref{induct3}.
 \end{proof}
 
 \subsubsection{$3$ vertices.} Now let us suppose $\Delta$ has a dual graph with more than $2$ vertices.  Then one of the vertices, $v_R$, must have degree $2$, i.e. it must correspond to a hyperelliptic component.  Let us assume $\Gamma_\Delta$ is: \[   \xygraph{
!{<0cm,0cm>;<1cm,0cm>:<0cm,1cm>::}
!{(0,-.25)}*{\bullet}="l"!{(-.3,-.45)}{\scriptstyle g_L}!{(-.3,-.05)}{\scriptstyle v_{L}}
!{(2,0) }*{\bullet}="b"!{(2.3,.2)}{\scriptstyle v_R}!{(2.3,-.2)}{\scriptstyle g_R}
!{(2,-.5) }*{\bullet}="a_1"
 "l"-@/^.3cm/"b"
 "l"-@/^.1cm/"b"
"l"-@/_.1cm/"a_1"
}  \] Clearly, there are two types of $3$-vertex divisors: either $v_R$ is ramified, or not.   We will only consider the analysis for the unramified case - the reader will by now be able to make the necessary minor adjustments for the ramified case.
 
 We vary the genus $g_R$ hyperelliptic curve associated to node $v_R$ in a ${\bf P}^1_{t}$ pencil on ${\bf F}_{g_R}$ with two specified points $p_1, p_2$ on a ruling line  $F$ lying in the base locus.  The resulting partial pencil, $p$, will intersect a new boundary divisor $\Delta_{\spl}$ which has dual graph $\Gamma_{\Delta_{\spl}}$:  \[   \xygraph{
!{<0cm,0cm>;<1cm,0cm>:<0cm,1cm>::}
!{(0,-.25)}*{\bullet}="l"!{(-.3,-.55)}{\scriptstyle g_L+1}!{(-.3,-.05)}{\scriptstyle v'_{L}}
!{(2,0) }*{\bullet}="b"!{(2.3,.2)}{\scriptstyle v'_R}!{(2.3,-.2)}{\scriptstyle g_R-1}
!{(2,-.5) }*{\bullet}="a_1"
 "l"-@/^.3cm/"b"
 "l"-@/^.1cm/"b"
"l"-@/_.1cm/"a_1"
}  \] (Notice that the effect is a reduction of the genus of $v_R$ by one.)  

\begin{prop}\label{3vertex}
The intersection numbers for $p$ are:
  \begin{align*}
\lambda \cdot p &= g_R\\
\delta \cdot p &= 8g_R +2\\
\Delta \cdot p &= -1\\
\Delta_{\spl} \cdot p &= 1\\
{\sfX} \cdot p &\geq 0
\end{align*}
\end{prop}
 
\begin{proof}
Standard - we omit it.
\end{proof}

 \begin{lem} Keeping the setting above, we have:
 \begin{equation}
 c(\Delta,{\sfY}) \geq g(g_R+2)-6g_R + c(\Delta_{\spl}, {\sfY})
 \end{equation}
 \end{lem}
 
\begin{proof}
Intersect $p$ with $\sfX = a\lambda - b\delta - {\sfY}$ and use \autoref{3vertex}.
\end{proof}

\begin{cor}\label{induct4}
Keep the setting of the previous lemma.  If $ c(\Delta_{\spl}, {\sfY}) \geq 0$ then $c(\Delta ,{\sfY}) \geq 0$.
\end{cor}

\begin{proof}
Clear.
\end{proof}

\begin{prop}
Keeping the same setting as the previous corollary, we have: 
$$c(\Delta, \sfY) \geq 0.$$
\end{prop}

\begin{proof}
 By repeatedly use relation \autoref{induct4}, we must only show that $c(\Delta_0, {\sfY}) \geq 0$ for $E_0$ having dual graph $\Gamma_{\Delta_{0}}$: \[   \xygraph{
!{<0cm,0cm>;<1cm,0cm>:<0cm,1cm>::}
!{(0,-.25)}*{\bullet}="l"!{(-.3,-.45)}{\scriptstyle g-1}!{(-.3,-.05)}{\scriptstyle v_{L}}
!{(2,0) }*{\bullet}="b"!{(2.3,.2)}{\scriptstyle v_R}!{(2.3,-.2)}{\scriptstyle 0}
!{(2,-.5) }*{\bullet}="a_1"
 "l"-@/^.3cm/"b"
 "l"-@/^.1cm/"b"
"l"-@/_.1cm/"a_1"
}  \] However, this dual graph is actually "unstable" in $\Hdgnt$ - the corresponding cover in $\Hdgnt$  is just an irreducible nodal curve whose normalization is the curve represented by $v_{L}$.  Therefore, this is an intersection with $\Delta_{\irr}$, which has coefficient $0$ in the divisor class $\sfY$. (Recall that $\sfY$ is supported on the higher boundary divisors.)
\end{proof}

\subsubsection{$4$ vertices.} Now suppose the dual graph for $E$ looks like $\Gamma_{\Delta}$:\[ \xygraph{
!{<0cm,0cm>;<1cm,0cm>:<0cm,1cm>::}
!{(0,-.6)}*{\bullet}="l"!{(-.3,-.8)}{\scriptstyle g_L}!{(-.3,-.4)}{\scriptstyle v_{L}}
!{(0,0)}*{\bullet}="l'"
!{(2,0) }*{\bullet}="b"!{(2.3,.2)}{\scriptstyle v_R}!{(2.3,-.2)}{\scriptstyle g_R}
!{(2,-.6) }*{\bullet}="a_1"
 "l'"-@/^.1cm/"b" 
 "l"-@/^.1cm/"b"
"l"-@/^.1cm/"a_1"
}  \] As usual in such symmetric settings, we assume $g_R \leq g_L$.  Again, vary $v_R$ in a partial pencil of hyperelliptic curves.  We obtain a family of admissible covers $p$ which intersects one other higher boundary divisor $\Delta_{\spl}$ whose dual graph is $\Gamma_{\Delta_{\spl}}$:  \[   \xygraph{
!{<0cm,0cm>;<1cm,0cm>:<0cm,1cm>::}
!{(0,-.25)}*{\bullet}="l"!{(-.3,-.45)}{\scriptstyle g_L}!{(-.3,-.05)}{\scriptstyle v'_{L}}
!{(2,0) }*{\bullet}="b"!{(2.3,.2)}{\scriptstyle v'_R}!{(2.3,-.2)}{\scriptstyle g_R-1}
!{(2,-.5) }*{\bullet}="a_1"
 "l"-@/^.3cm/"b"
 "l"-@/^.1cm/"b"
"l"-@/_.1cm/"a_1"
}  \]In other words, $\Delta_{\spl}$ is of the type discussed in the previous subsection - the left side of the admissible covers represented by $\Delta_{\spl}$ has only one vertex. 

\begin{prop}\label{4vertex}
The intersection numbers for $p$ are: 
 \begin{align*}
\lambda \cdot p &= g_R\\
\delta \cdot p &= 8g_R +3\\
\Delta \cdot p &= -1\\
\Delta_{\spl} \cdot p &= 1\\
{\sfX} \cdot p &\geq 0
\end{align*}

\end{prop}

\begin{proof}
At this point, this is routine.
\end{proof}

\begin{rmk}
The number $8g_{R}+3$ is explained as follows.  A general pencil on $\F_{g_{R}}$ of genus $g_{R}$ hyperelliptic curves has $8g_{R}+4$ nodes.  However, we must subtract $1$ for the gluing section which connects $v_{R}$ with $v_{L}$, and {\sl not} for the gluing section which connects with the rational node, since this node does not add contribute to $\delta$.
\end{rmk}

\begin{lem}
In the setting of the previous proposition, we have:  
\begin{equation}
 c(\Delta,{\sfY}) \geq g(g_R+3)-6g_R + c(\Delta_{\spl}, {\sfY}).
\end{equation}
\end{lem}

\begin{proof}
This follows by intersecting $p$ with $X = a\lambda - b\delta - \sfY$ and using \autoref{4vertex}.
\end{proof}

\begin{cor}\label{induct5}
If $c(\Delta_{\spl},Y) \geq 0$, then $c(\Delta, \sfY) \geq 0$.
\end{cor}

\begin{proof}
When $g_{R} \leq g_{L}$, the expression  $g(g_R+3)-6g_R$ is easily seen to be nonnegative, so we conclude by \autoref{induct5}.
\end{proof}

\subsection{A simplifying observation.}
After seeing the general inductive method in the degree three case, we pause to notice a few things.  The underlying pencils used to construct the partial pencils obviously depended on the degree of the vertex $v_R$ which was being varied: pencils of trigonal curves in ${\bf F}_{0}$ or ${\bf F}_{1}$ for degree $3$ vertices, and pencils of genus $g_R$ hyperelliptic curves in ${\bf F}_{g_R }$ for degree $2$ vertices.   With  the latter hyperelliptic pencils, the quantity $b\delta - a\lambda$ becomes ``much more positive" than it is for pencils of trigonal curves.  This simply reflects the fact that the slope $\frac{8g_{R} +4}{g_{R}}$ of these hyperelliptic families is larger than  $\frac{a}{b} = \frac{7g+6}{g}$.  

Now that we are considering the degree four situation, we notice that a similar phenomenon occurs.  For a divisor $\Delta$ whose dual graph has more than $2$ vertices, we may use the same types of pencils used in the analysis of lower degree: either pencils of trigonal curves in ${\bf F}_{0}$ (or $\F_{1}$, depending on the parity of $g_{R}$), or pencils of hyperelliptic curves of genus $g_R$ in ${\bf F}_{g_R}$.  Because $\frac{a}{b} = \frac{13g+ 15}{2g} $ is  smaller than $\frac{7g_R+6}{g_R}$ and $\frac{8g_R+4}{g_R}$, inductive arguments exactly like those found in the previous section will show that $c(\Delta, \sfY) \geq 0$ for divisors having more than two vertices.

Therefore we arrive at the following simplifying observation: $$\text{\sl It is the $2$-vertex boundary divisors which are of primary interest.}$$  We will need to use the families described in section \ref{TetragonalFamilies} to deal with $2$ vertex divisors.  These same observations hold when we switch to considering pentagonal curves.

\subsection{Degree four analysis.}

We will use the pencils occurring in \autoref{basictetragonal}.  In order to create partial pencils from these, we will impose a $4$-tuple of points in a general (conic) fiber of $\pi \colon \F \longrightarrow {\bf P}^1$ to lie in the base locus of our pencil $p$.  

\subsubsection{Unramified, $2$-vertex $\Delta$.}
Let $\Delta$ be an unramified, $2$-vertex divisor with dual graph $\Gamma_\Delta$:\[  \xygraph{
!{<0cm,0cm>;<1cm,0cm>:<0cm,1cm>::}
!{(0,0) }*{\bullet}="l"!{(-.3,.2)}{\scriptstyle v_L}!{(-.3,-.2)}{\scriptstyle g_L}
!{(2,0) }*{\bullet}="r"!{(2.3,.2)}{\scriptstyle v_R}!{(2.3,-.2)}{\scriptstyle g_R}
 "l"-@/^.6cm/"r"  
"l"-@/^/"r" 
"l"-@/_/"r"
"l"-@/_.6cm/"r" 
}  \]Assume that $g_L \geq g_R$.

By imposing $4$ distinct points $\{p_{1}, p_{2}, p_{3}, p_{4}\}$ in a fiber $F \subset \F$, and then attaching a constant left side cover $\alpha_{L} \from C_{L} \to P_{L}$, we construct a partial pencil $p$ lying entirely inside a boundary divisor $\Delta$. 

The split fiber $\Delta_{\spl}$ has dual graph $\Gamma_{\Delta_{\spl}}$: \[  \xygraph{
!{<0cm,0cm>;<1cm,0cm>:<0cm,1cm>::}
!{(0,0) }*{\bullet}="l"!{(-.3,.2)}{\scriptstyle v'_L}!{(-.3,-.3)}{\scriptstyle g_L+3}
!{(2,0) }*{\bullet}="r"!{(2.3,.2)}{\scriptstyle v'_R}!{(2.3,-.3)}{\scriptstyle g_R-3}
 "l"-@/^.6cm/"r"  
"l"-@/^/"r" 
"l"-@/_/"r"
"l"-@/_.6cm/"r" 
}  \]

The interactions of $p$ are given in the next proposition:

\begin{prop}
The pencil $p$ has the following intersection numbers:
\begin{align*}
\lambda \cdot p &= g_{R}\\
\delta \cdot p &= v + 6g_{R} + 2\\
\Delta \cdot p &= -1\\
\Delta_{\spl} \cdot p &= 1\\
\sfX \cdot p &\geq 0.
\end{align*}
\end{prop} 

\begin{proof}
This follows easily from the calculations in \autoref{basictetragonal}. The last inequality follows from the fact that such partial pencils $p$ pass through the general point of $\Delta$.
\end{proof}

\begin{lem}\label{induct-tetra-1}
In the setting of the previous proposition, we have the following inequality:
$$c(\Delta, \sfY) \geq c(\Delta_{\spl}, {\sfY}) + b(v+6g_R +2) - a(g_R).$$
\end{lem}

\begin{proof}
We intersect $p$ with $\sfX = a\lambda - b\delta - \sfY$.
\end{proof}

\subsubsection{A hitch in the induction.}
Recall that $\frac{a}{b} = \frac{13g+15}{2g}$.   Unfortunately, we may not run the induction procedure directly from \autoref{induct-tetra-1} as we did in the degree three analysis: the quantity \[b(v+6g_R +2) - a(g_R)\] is not positive for all $g_R$ in the assumed range $g_{R} \leq (g-3)/2$.  Luckily, the fraction of $g_R$'s making this quantity negative is quite small, so we can use \autoref{induct-tetra-1} multiple times, reducing genera of the vertex $v_{R}$ by three each time, and add all the resulting inequalities. This strategy ends up working, as we explain next. 

\subsubsection{}

\begin{prop}\label{tetr-ineq1}
In the setting above, we have: $$c(\Delta, \sfY) \geq 0.$$
\end{prop}

\begin{proof}
  Write $g_R = 3i+k$, with $k=0,1,2$ depending on $g_R\mod 3$.  Let $\Delta_l$ denote the unramified, $2$-vertex boundary divisor having dual graph $\Gamma_{\Delta_{l}}$:\[  \xygraph{
!{<0cm,0cm>;<1cm,0cm>:<0cm,1cm>::}
!{(0,0) }*{\bullet}="l"!{(-.3,.2)}{\scriptstyle v}!{(-.3,-.3)}
!{(2,0) }*{\bullet}="r"!{(2.3,.2)}{\scriptstyle w}!{(2.3,-.3)}{\scriptstyle 3l+k}
 "l"-@/^.6cm/"r"  
"l"-@/^/"r" 
"l"-@/_/"r"
"l"-@/_.6cm/"r" 
}  \]  (So our original boundary divisor $\Delta$ is $\Delta_i$.) \autoref{induct-tetra-1} gives: \[c(\Delta_l, {\sfY}) \geq c(\Delta_{l-1}, {\sfY}) + b(v_l + 6(3l+k) +2)-a(3l+k)\] for all $l$.  We may add these relations together for $0 \leq l \leq i$ to obtain the following expression

\begin{equation}\label{5.7}
c(\Delta_{i}, {\sfY}) \geq c(\Delta_{0}, {\sfY}) +  \displaystyle{\sum_{l=1}^{i}(18b-3a)l} +bv_{l} + (6bk+2b-ka)i
\end{equation}

In the above expression,  $v_l := \lceil{\frac{3l+k+3}{2}\rceil}$.  Therefore, if we replace $v_l$ by $\frac{3l+k+3}{2}$ and if the sum in the right hand side of \eqref{5.7} remains positive, we will be able to deduce the inequality $c(\Delta, \sfY) \geq 0$ from the inequality $c(\Delta_{0}, {\sfY}) \geq 0$.  Indeed, substituting $v_l$ with $\frac{3l+k+3}{2}$  gives

\begin{equation}
c(\Delta_{i}, {\sfY}) \geq c(\Delta_{0}, {\sfY}) + 3\left(\frac{13}{2}b-a\right)\cdot {i+1 \choose 2} +  \left(\left(\frac{13}{2}k+\frac{7}{2}\right)b-ka\right)\cdot i  
\end{equation}

From here, it is straightforward to check that \[ 3\left(\frac{13}{2}b-a\right) {i+1 \choose 2} +  \left(\left(\frac{13}{2}k+\frac{7}{2}\right)b-ka\right)i  \geq 0\] for $i$ in the required range, i.e. $3i+k \leq \frac{g-3}{2}$.

Therefore, we must only show that $c(\Delta_{0}, {\sfY}) \geq 0$.  These remaining cases are easily dealt with if we simply interpret the divisor $\Delta_{-1}$ as having more than $2$ vertices. 

 In other words, the split fibers appearing in the partial pencils in the analysis of $c(\Delta_{0},\sfY)$ have disconnected residual curves. However, as we explained earlier, we know $c(\Delta_{-1},0) \geq 0$.  The relation \[c(\Delta_0, {\sfY}) \geq c(\Delta_{-1}, {\sfY}) + b(v_0 + 6k +2)-ak\] then allows us to conclude $c(\Delta_0, {\sfY})\geq 0$, since the quantity $b(v_{0} + 6k +2) - ak$ is positive for these small values for $v_{0}$, as the reader can check.

\end{proof}

\subsubsection{Ramified, $2$-vertex divisors.} The rest of the degree four analysis goes through without any hitches. 

We modify the basic pencil $p$ above by imposing nonreduced base points as we did in the degree three case.  We will demonstrate the analysis in one case, and leave the remaining cases to the reader.

  Let $p'$ be the partial pencil obtained by imposing a $4$-tuple of basepoints of the form $\{2p_1, p_2, p_3\}$ in a general conic fiber of $\pi \colon \F \longrightarrow {\bf P}^1$.  Suppose $\Gamma_{\Delta}$ is:\[   \xygraph{
!{<0cm,0cm>;<1cm,0cm>:<0cm,1cm>::}
!{(0,0) }*{\bullet}="l"!{(-.3,.2)}{\scriptstyle v_L}!{(-.3,-.2)}{\scriptstyle g_L}
!{(2,0) }*{\bullet}="r"!{(2.3,.2)}{\scriptstyle v_R}!{(2.3,-.2)}{\scriptstyle g_R}
 "l"-@/^.5cm/"r"  
"l"-@/^/"r" 
"l"-@/_.3cm/"r" _{\scriptstyle 2}
}  \] 

\begin{prop}\label{tetra2}
Then the intersection numbers for the resulting family $p'$ are: 
\begin{align*}
\lambda \cdot p' &= g_R\\
\delta \cdot p' &= v+6g_R+3\\
\Delta \cdot p' &= -1\\
\Delta_{\spl} \cdot p' &= 1\\
\Delta_{\ram} \cdot p' &= 1\\
{\sfX} \cdot p' &\geq 0
\end{align*}
\end{prop}

\begin{proof}
Clear.
\end{proof}

\begin{prop}
In the setting above, $$c(\Delta, \sfY) \geq 0.$$
\end{prop}

\begin{proof}
The divisor $\Delta_{\spl}$, originating from the split fiber of $p'$, has dual graph $\Gamma_{\Delta_{\spl}}$: \[  \xygraph{
!{<0cm,0cm>;<1cm,0cm>:<0cm,1cm>::}
!{(0,0) }*{\bullet}="l"!{(-.3,.2)}{\scriptstyle v'_L}!{(-.3,-.3)}{\scriptstyle g_L+2}
!{(2,0) }*{\bullet}="r"!{(2.3,.2)}{\scriptstyle v'_R}!{(2.3,-.3)}{\scriptstyle g_R-3}
 "l"-@/^.6cm/"r"  
"l"-@/^/"r" 
"l"-@/_/"r"
"l"-@/_.6cm/"r" 
}  \]and $\Delta_{\ram}$, coming from the reduced tamification fiber, has an {\sl unramified} dual graph $\Gamma_{\Delta_{\ram}}$: \[  \xygraph{
!{<0cm,0cm>;<1cm,0cm>:<0cm,1cm>::}
!{(0,0) }*{\bullet}="l"!{(-.3,.2)}{\scriptstyle v''_L}!{(-.3,-.3)}{\scriptstyle g_L}
!{(2,0) }*{\bullet}="r"!{(2.3,.2)}{\scriptstyle v''_R}!{(2.3,-.3)}{\scriptstyle g_R-1}
 "l"-@/^.6cm/"r"  
"l"-@/^/"r" 
"l"-@/_/"r"
"l"-@/_.6cm/"r" 
}  \]  
\autoref{tetr-ineq1} says $c(\Delta_{\spl}, {\sfY}) \geq 0$ and $c(\Delta_{\ram}, {\sfY}) \geq 0$.

By intersecting $p'$ with $\sfX = a\lambda - b\delta - \sfY$, we get:

\begin{equation}
c(\Delta, \sfY) \geq c(\Delta_{\spl}, {\sfY}) + c(\Delta_{\ram}, {\sfY}) + 9g -15g_R
\end{equation}

Since we assume $g_R \leq \frac{g-3}{2}$, we easily conclude  that $c(\Delta, \sfY) \geq 0$.  
\end{proof}

\subsubsection{Higher ramification.}
For higher ramification, we simply impose more nonreduced points in the base loci of our partial pencils.  The resulting partial pencil family will relate the current ramified divisor $\Delta$ with a divisor $\Delta_{\ram}$ having less ramification.  We omit the details to avoid repetitiveness.

\subsection{Degree five analysis.} 

Our approach to the degree five case of \autoref{Yeffective}, although similar in essence to the previous two, has one added complication.  So far, we have never had to compute the exact intersection number of a partial pencil with the divisors $\sfM$ or $\sfCE$ -- it was always sufficient only to know that these intersections were nonnegative.

In the degree five case, we will have to understand intersections of partial pencils with the Maroni divisor $\sfM$. (Our choice of partial pencil will conveniently circumvent any need to consider $\sfCE$.)

When $d=5$, the reader can check that the coefficients $a$ and $b$ occurring in the divisor class expression for $\sfX$ are: 
\begin{align*}
a &= (31g+44)/10\\
b &= g/2.
\end{align*}

\subsubsection{Unramified, $2$-vertex divisors.}\label{pentagonalunramified} We begin by considering a boundary divisor $\Delta$ whose dual graph $\Gamma_{\Delta}$ is: 

\[  \xygraph{
!{<0cm,0cm>;<1cm,0cm>:<0cm,1cm>::}
!{(0,0) }*{\bullet}="l"!{(-.3,.2)}{\scriptstyle v_L}!{(-.3,-.2)}{\scriptstyle g_L}
!{(2,0) }*{\bullet}="r"!{(2.3,.2)}{\scriptstyle v_R}!{(2.3,-.2)}{\scriptstyle g_R}
 "l"-@/^.6cm/"r"  
"l"-@/^/"r" 
"l"-"r" 
"l"-@/_/"r"
"l"-@/_.6cm/"r" 
}  \]

\subsubsection{}
Recall the construction of pentagonal families found in section \ref{pentagonalconstruction}.  We had the genus $1$ fibration $$\F \subset \P(\bigwedge^{2} \dual{\cF}) \to \P^{1}_{s}$$ on which we considered the linear system $|H_{1}|$. 

The surface $\F$ is a complete intersection of five relative hyperplane divisors $H_{2}, ... ,H_{6}$ having divisor classes $\zeta + k_{2}f, ... , \zeta + k_{6}f$ in the scroll $\P(\bigwedge^{2} \dual{\cF}) \to \P^{1}_{s}$.  Therefore $\F$ spans a $\P^{4}$-bundle subscroll $\P(\cE_{\F})$, in which the fibers of $\F$ are degree five elliptic normal curves in $\P^{4}$. 

\subsubsection{}
The rank $5$ vector bundle $\cE_{\F}$ is defined by a sequence on $\P^{1}_{s}$: 
$$0 \to \bigoplus_{i = 2}^{6}\cO(-k_{i}) \to \bigwedge^{2}\dual{\cF} \to \cE_{\F} \to 0.$$

Picking a general pencil $\P^{1}_{t} \subset |H_{1}|$ on the surface $\cF$ is then equivalent to varying an inclusion $\cO(-k_{1}) \hookrightarrow \cE_{\F}$ with the parameter $t \in \P^{1}_{t}$. 

\subsubsection{}
Put differently, a general pencil $\P^{1}_{t} \subset |H_{1}|$ is the same as a sequence of vector bundles on $\P^{1}_{t} \times \P^{1}_{s}$: 
 \begin{equation}
0 \to \cO(-k_{1}R_s-R_t) \longrightarrow f_{s}^*\cE_{\F} \longrightarrow \cW \longrightarrow 0
 \end{equation}
 where $f_{s}, f_{t}$ are respective projections to $\P^{1}_{t}$ and $\P^{1}_{s}$ and $R_{s}$ and $R_{t}$ are ruling lines parametrized by $s$ and $t$.

 The vector bundle $\cW$ restricts for all $t$ to $\cW_{t} = \cE_{\alpha_{t}} \otimes \det \dual{\cE_{\alpha_{t}}}$, where $\alpha_{t} \from C_{t} \to \P^{1}_{s}$ is the degree five branched cover parametrized by $t \in \P^{1}_{t}$.

 Note that $k_{1} = \lceil \frac{5(g_{R}+4)}{6} \rceil$.

\subsubsection{}\label{pentagonalpencil} Let $\P^{1}_{t}$ be a general genus $g_{R}$ pencil $p$ in $|H_{1}|$ with $5$ distinct basepoints imposed in a fiber $F \subset \F$.  We denote by $$\alpha \from \cC_{L} \cup \cC_{R} \to \cP_{L} \cup \cP_{R}$$ the ensuing partial pencil whose variation is given by $\P^{1}_{t} \subset |H_{1}|$.

The difficulty with using a pencil $p$ as constructed above to prove \autoref{Yeffective} is that we need to calculate the intersection of $p$ with the Maroni divisor $\sfM$.

\subsubsection{Maroni special curves in $\Delta$.}\label{Maronispecialanalysis} In this section, let $\beta \from C_{L} \cup C_{R} \to P_{L} \cup P_{R}$ be a curve in $\Delta$.  In particular, there is no ramification occurring above the node $x = P_{L} \cap P_{R}$. 

The Tschirnhausen bundles $\cE_{L}$ and $\cE_{R}$ glue together to form $\cE_{\beta}$.  In particular, the scroll $\P\cE_{\beta}$ is created by gluing $\P\cE_{L}$ to $\P\cE_{R}$ along the $\P^{3}$ lying over the node $x$.

\begin{defn}\label{directrix}
Let $\cA = \cO(k)^{\oplus m} \oplus \cO(k+1)^{\oplus n}$ be a balanced vector bundle on $\P^{1}$. The {\sl directrix} of $\P \cA$ is the subscroll associated to the surjection $\cA \onto \cO(k)^{\oplus m}$.
\end{defn}

Now assume $\cE_{L}$ and $\cE_{R}$ are each balanced. Then each scroll $\P\cE_{L}$ and $\P\cE_{R}$ has its directrix subscroll $\Sigma_{L}$ and $\Sigma_{R}$ which, due to the assumption $4 \mid g$, have ``complementary'' rank, i.e. the dimension of the fibers of $\Sigma_{L}$ and $\Sigma_{R}$ add to $2$.

\begin{thm}\label{directrixmeet}
Let $\beta \in \Delta$ be such that $\cE_{L}$ and $\cE_{R}$ are both balanced. Then $\beta \in \sfM$ if and only if the directrices $\Sigma_{L}$ and $\Sigma_{R}$ intersect nontrivially. 

More generally, let $\beta_{t} \from \cC_{L,t} \cup \cC_{R,t} \to \cP_{L,t} \cup \cP_{R,t}$ be a family of $2$-component covers parameterized by a map $\varphi \from \Spec k[[t]] \to \Delta$, with $x(t) = \cP_{L,t} \cap \cP_{R,t}$ the family of nodes, such that
\begin{enumerate}
 \item $\cE_{L,t}$ and $\cE_{R,t}$, the two components of $\cE_{\beta_{t}},$ are balanced for all $t$.
 \item The total spaces of the family of directrices $\Sigma_{L,t} \subset \P \cE_{L,t}$ and $\Sigma_{R,t} \subset \P \cE_{R,t}$ meet at a zero dimensional scheme supported above $t=0$.
\end{enumerate}
Then: $$\deg \varphi^{-1}(\sfM) = \length \left(\Sigma_{L,t} \cap \Sigma_{R,t} \cap \P (\cE_{\beta_{t}}|_{x(t)})\right).$$
\end{thm}

\begin{rmk}
The theorem's content comes from the fact that the Maroni divisor $\sfM \subset \Hdgnt$ is defined by taking a {\sl closure}, which means that one must be careful when analyzing it in the boundary. Furthermore, the theorem is true for any degree $d$, as will be apparent in the proof.
\end{rmk}

\begin{proof}

The first statement is straightforward: one checks that $h^{1}(\End \cE) = 0$ if and only if the directrices do not meet. Then we conclude by semicontinuity.

The second statement is not immediate.  We must understand how to interpret the Maroni divisor $\sfM$, which is defined by taking a closure, near the point $[\beta \from C_{L} \cup C_{R} \to P_{L} \cup P_{R}]$ in the boundary $\Delta$. 
For this, let $(B, 0)$ denote a versal deformation space of $[\beta] \in \Hdgnt$, and let $\pi \from \cP_{B} \to B$ denote the versal family of targets. So $\pi^{-1}(0) = P_{L} \cup P_{R}$. 

We continue to let $\Delta \subset B$ and $\sfM \subset B$ denote the local boundary divisor and Maroni divisor in the versal space $B$. 

Over $\cP_{B}$ we have the versal Tschirnhausen bundle $\cE_{\rm versal}$. Pick smooth points $s_{L} \in P_{L}$ \and $s_{R} \in P_{R}$ and extend them to sections $\sigma_{L}, \sigma_{R} \subset \cP_{B}$. (There are no obstructions to extending.) By assumption $\cE_{\beta}$ restricts as $\cO(a)^{\oplus r} \oplus \cO(a-1)^{\oplus 4-r}$ on $P_{L}$ and $\cO(b)^{\oplus 4-r} \oplus \cO(b-1)^{\oplus r}$ on $P_{R}$.

Now we twist $\cE_{\rm versal}$ down by the line bundle $\cO_{\cP_{B}}(a\sigma_{L} + b\sigma_{R})$ to get a bundle $\cE'_{{\rm versal}}$ such that $\cE'_{{\rm versal}}|_{0} := \cE'_{{\rm versal}}|_{P_{L}\cup P_{R}}$ restricts to $\cA_{L} := \cO^{\oplus r} \oplus \cO(-1)^{\oplus 4-r}$  and $\cA_{R} := \cO^{\oplus 4-r} \oplus \cO(-1)^{\oplus r}$ on $P_{L}$ and $P_{R}$ respectively.

Take the exact sequence on $\cP_{B}$ \[0 \to \cE'_{{\rm versal}} \to \cE'_{{\rm versal}}(+\sigma_{L}) \to \cE'_{{\rm versal}}(+\sigma_{L})|_{\sigma_{L}} \to 0\] and push it down via $\pi$.  Since $\cE'_{{\rm versal}}$ restricts on the generic fiber of $\pi$ to $\cO(-1)^{\oplus 4}$, we conclude that $\pi_{*}(\cE'_{{\rm versal}}) = 0$, and that $R^{1}\pi_{*}(\cE'_{{\rm versal}})$ is supported on a codimension one subvariety containing the Maroni divisor $\sfM$. Furthermore, the Bogomolov expression for $\cE'_{\rm versal}$ comes from computing $c_{1}R^{1}\pi_{*}(\cE'_{{\rm versal}})$ over the locus $B \setminus \Delta$, so the Maroni divisor $\sfM$ occurs with multiplicity one as a component of the support of $R^{1}\pi_{*}(\cE'_{{\rm versal}})$.

This also means that any other component of the support of $R^{1}\pi_{*}(\cE'_{{\rm versal}})$ must be supported on $\Delta$.  However, it is easy to check that for a general point in $\Delta$, the directrices $\Sigma_{L}$ and $\Sigma_{R}$ do not meet, and that this translates into vanishing of $R^{1}\pi_{*}(\cE'_{{\rm versal}})$ at such a point. 

Therefore, {\sl the Maroni divisor $\sfM$ is simply the codimension one support of $R^{1}\pi_{*}(\cE'_{{\rm versal}})$.}

By pushing forward, we end up with a sequence of $\cO_{B}$--modules:
\begin{equation}\label{exactseqMaroni}
0 \to \pi_{*}\cE'_{{\rm versal}}(+\sigma_{L}) \to \cE'_{{\rm versal}}(+\sigma_{L})|_{\sigma_{L}} \to R^{1}\pi_{*}(\cE'_{{\rm versal}}) \to 0.
\end{equation}

The second sheaf is clearly locally free, and so is the first.  Indeed, it suffices to show that $h^{0}(\cE'_{{\rm versal}}(+\sigma_{L})|_{P_{L}\cup P_{R}}) = d-1$ regardless of whether the directrices meet or not. This is straightforward: it follows from the fact that $\cE'_{{\rm versal}}(+\sigma_{L})|_{P_{L}\cup P_{R}}$ restricts to a globally generated bundle on $P_{L}$ and has trivial $h^{1}$ on either. component.  We conclude that the support of $R^{1}\pi_{*}(\cE'_{{\rm versal}})$ is pure codimension one.

Now we specialize our attention to the type of family parameterized by $\Spec k[[t]]$ found in the statement of the theorem. Thus, we have a map $\varphi \from \Spec k[[t]] \to \Delta \subset B$ whose closed point maps to $0 \in B$, and whose generic point does not lie in $\sfM$. 

We let $\beta_{t}$ denote this family of covers, and we see that we must compute the power of $t$ which annihilates the $k[[t]]$-module $\varphi^{*} R^{1}\pi_{*}(\cE'_{{\rm versal}})$. (We need not worry about cohomology and base change issues, thanks to the free resolution \ref{exactseqMaroni}. Also, we will suppress the ``versal'' and pullbacks in our notation and only write $R^{1}\pi_{*}(\cE'_{\beta_{t}}).$)

In fact, we will compute this annihilator by using a different resolution of the module $\varphi^{*} R^{1}\pi_{*}(\cE'_{{\rm versal}})$ which is available to us precisely because we have a family of two-component admissible covers.  

Let $\nu \from \cP_{L,t} \sqcup \cP_{R,t} \to \cP_{L,t} \cup \cP_{R,t}$ denote the normalization. Then we use consider the following exact sequence on $\cP_{L,t} \cup \cP_{R,t}$: 

\begin{equation}\label{exactseq2component}
0 \to \cE'_{\beta_{t}} \to \nu_{*}\nu^{*}\cE'_{\beta_{t}} \to \cE'_{\beta_{t}}|_{x(t)} \to 0,
\end{equation} 
where $x(t)$ is the family of nodes. We therefore seek the length of the cokernel of the restriction map $$\rho \from H^{0}(\nu_{*}\nu^{*}\cE'_{\beta_{t}}) \to H^{0}(\cE'_{\beta_{t}}|_{x(t)}).$$

The latter is clearly a free $k[[t]]$-module of rank $4$.

 By assumption, the bundle $\cE'_{\beta_{t}}$ restricts to $\cA_{L,t}$ and $\cA_{R,t}$ on $\cP_{L,t}$ and $\cP_{R,t}$ respectively, where $$\cA_{L,t} = \cO_{\cP_{L,t}}^{\oplus r} \oplus \cO_{\cP_{L,t}}(-1)^{\oplus 4-r}$$ and $$\cA_{R,t} = \cO_{\cP_{R,t}}^{\oplus 4-r} \oplus \cO_{\cP_{R,t}}(-1)^{\oplus r}$$ therefore we get \[H^{0}(\nu_{*}\nu^{*}\cE'_{\beta_{t}}) = k[[t]]\langle e_{1}, ... , e_{r} \rangle \oplus k[[t]]\langle f_{1}, ... , f_{4-r} \rangle \] where $e_{i}$ and $f_{j}$ are bases for the free $k[[t]]$--modules $H^{0}(\cA_{L,t})$ and $H^{0}(\cA_{R,t})$ respectively. 

Hence the length of the cokernel of $\rho$ is given by order of vanishing of the determinant of the matrix $\rho$ at $t=0$. 

Now we simply notice that in the projective bundle $\P \cE'_{\beta_{t}}|_{x(t)}$ over $\Spec k[[t]]$, the equations $\rho(e_{1}) = ... \rho(e_{r}) = 0$ cut out the scheme $\Sigma_{L,t} \cap \P \cE'_{\beta_{t}}|_{x(t)}$ and similarly $\rho(f_{1}) = ... \rho(f_{4-r}) = 0$ cut out the scheme $\Sigma_{R,t} \cap \P \cE'_{\beta_{t}}|_{x(t)}$.  The theorem now follows by the simple observation that the scheme theoretic intersection of these two linear subschemes is given by the vanishing of the determinant of $\rho$. 

\end{proof}

\subsubsection{The partial pencil $p$ avoids $\sfCE$.} 
\begin{prop}\label{avoidCE}
A general partial pencil family $p$ as in section \ref{pentagonalpencil} avoids the Casnati-Ekedahl divisor $\sfCE$.
\end{prop}

\begin{proof}
The partial pencil $p$ varies with the parameter $t \in \P^{1}_{t}$, and $t= \infty$ gives the split fiber. Let $\alpha_{t}$ denote the admissible cover in $p$ parameterized by $t$.

First, consider $t \neq \infty.$  The right side $C_{R,t} \subset \F$ intersects every degree $5$ elliptic normal curve fiber $F_{s} \subset \F \subset \P \cE_{\F}$ as a hyperplane section. Let $\cF_{R,t}$ denote the Casnati-Ekedahl bundle for the pentagonal curve $C_{R,t}$.

 Consider the vector bundle defined by the formula: 
$$\td{\cF}|_{s} := \{\text{Quadrics in $\P(\cE_{\F})|_{s}$ containing $F_{s}$}\}.$$ As explained in the proof of \autoref{EF01}, part $(d)$, the natural restriction map 
$$\res \from \td{\cF} \to \cF_{R,t}$$ is an isomorphism for all $t$. 

This means that the family of Casnati-Ekedahl bundles $\cF_{R,t}$, $t \in \A^{1}_{t}$, regarded as a vector bundle on $\A^{1}_{t} \times \P^{1}_{s}$ is the pullback of the bundle $\td{\cF}$ under the projection to $\P^{1}_{s}$. 

Since the left sides $\C_{L,t}$ of the partial pencil $p$ are constant (including points of attachment), we conclude: $$\text{\sl The family of Casnati-Ekedahl bundles $\cF_{\alpha_{t}}$ is constant for $t \neq \infty$.}$$

Therefore, by a similar analysis as in the proof of \autoref{directrixmeet}, we conclude that by choosing the pencil $p$, $C_{L}$, and the points of attachment on $C_{L}$ generically, we may avoid $\sfCE$ for $t \neq \infty$. 

It remains to understand why the split fiber at $t=\infty$ is not in $\sfCE$. For this, we will refer again back to \autoref{EF01} part $(d)$, but for slightly different reasons.

The admissible cover parameterized by $t = \infty$ $$\alpha_{\infty} \from C_{L} \cup C_{M} \cup C'_{R} \to P_{L} \cup P_{M} \cup P_{R}$$ has three components, with the left curve $C_{L}$ being the fixed curve we have glued to the varying right side.  The curve $C_{M}$ is isomorphic to the fiber $F \subset \F$ which split off in the original pencil $\P^{1}_{t} \subset |H_{1}|,$ and the curve $C'_{R}$ is the residual curve in the divisor class $H_{1} - F$. 

By \autoref{EF01}, we know that the Casnati-Ekedahl bundle $\cF_{M}$ for the middle map $$\alpha_{\infty,M} \from C_{M} \to P_{M}$$ is perfectly balanced, isomorphic to $\cO(2)^{\oplus 5}$. We furthermore know from arguments earlier in this proof, we know that the Casnati-Ekedahl bundle, up to twisting by a line bundle on $\P^{1}_{s}$, is isomorphic to the vector bundle $\td{\cF}$. 

With this description of $\cF_{\alpha_{\infty}}$, and by similar arguments to the proof of \autoref{directrixmeet}, we see that for a general choice of residual curve $C'_{R}$, the bundle $\cF_{\alpha_{\infty}}$ obeys $h^{1}(\End \cF_{\alpha_{\infty}}) = 0$, which implies that $\alpha_{\infty}$ is Casnati-Ekedahl general.

\end{proof}

\begin{rmk}
The fact that $\sfCE \cdot p = 0$ was the original motivation in considering this particular construction of partial pencil $p$. 
\end{rmk}

\subsection{A proposition on ``rotating directrices.''} Given \autoref{directrixmeet}, we are now ready to compute the intersection of $p$ with the divisor $\sfM$. We see that the computation translates into a question about the variation of directrices in a family of balanced scrolls.

\subsubsection{}
We begin with the following general setup. Let $\cV$ be a vector bundle on ${\bf P}^1_{s}$ of rank $N$.  On ${\bf P}^1_{t} \times {\bf P}^1_{s}$, consider an exact sequence of the form \[0 \longrightarrow \mathcal{O}(-aR_s-R_t) \longrightarrow f_{s}^*\cV \longrightarrow \cW \longrightarrow 0\] where $R_{s}$ and $R_{t}$ denote the ruling line classes parametrized by $s$ and $t$, respectively. We let $f_{s}$ and $f_{t}$ denote the projections of $\P^{1}_{t} \times \P^{1}_{s}$ to the respective factors, and $\pi_{s}$ and $\pi_{t}$ denote the respective maps from $\P \cV $ to $\P^{1}_{s}$ and $\P^{1}_{t}$.

 We further assume that $\cW$ splits as \[\cW_{t} = \mathcal{O}(l)^{\oplus r} \oplus \mathcal{O}(l+1)^{\oplus N-1-r}\] for some $1 \leq r \leq N-2$ and all $t \in \P^{1}_{t}$. In other words, we assume that $\cW_{t}$ is balanced, though not perfectly balanced, for all $t$. 

\subsubsection{} 
Pick a point $s = 0 \in \P^{1}_{s}$ and consider the preimage $X := \pi_{s}^{-1}(0) \subset \P \cV.$  Then $X$ is isomorphic to the product $\P^{N-1} \times  \P^{1}_{t}$. If we intersect $\P \cD_{t}$ with $X$, we get a family of $\P^{r-1}$'s in $\P^{N-1}$ parametrized by $t$. Letting $Y \subset X$ denote the total space of this family,  our intention is to compute the ``degree'' of $Y$, viewed as a one parameter family of $(r-1)$-planes in $\P^{N-1}$. 

\begin{rmk}
If we imagine the scrolls $\P\cW_{t}$ as varying with time, we may imagine their directrices $\P \cD_{t}$ as ``rotating'' in the total space $\P\cW$, and this amount of rotation is quantified by the degree mentioned above.  
\end{rmk} 

\begin{prop}[Rotating directrices]\label{a+l+1}
Let $H$ and $F$ denote the pullbacks of the hyperplane class of $\P^{N-1}$ and $\P^{1}_{t}$, respectively, to $X$. Then the class of the total space of $Y$ in the Chow ring of $X$ is given by 
\begin{equation}
[Y] = H^{N-r} + (a+l+1)H^{N-r-1}\cdot F.
\end{equation}
\end{prop}
\begin{proof}
Let $\cA$ be any vector bundle on ${\bf P}^1_{t} \times {\bf P}^1_{s}$ such that $\cA_{t}$ has vanishing higher cohomology for all $t \in \P^{1}_{t}$.  We can compute the degree of the pushforward $f_{t*}\cA$ using the Grothendieck-Riemann-Roch formula; we get:
\begin{equation}\label{GRR}
\deg \, (f_{t*}\cA) = c_1\cA\cdot R_s + \frac{c_1^2\cA}{2}-c_2\cA.
\end{equation}

We apply this formula to $\cA = \cW(-(l+1)R_s)$.  For this, we first note that 
\begin{eqnarray*}
&&c_1\cW = c_1\cV + aR_s+R_t = \deg \cV\cdot R_s +a R_s +R_t,\\
&&c_2\cW = \deg \cV + 2a,
\end{eqnarray*}
and so 
\begin{eqnarray*}
c_1\cW(-(l+1)R_s) &=& \deg \cV\cdot R_s +a R_s +R_t - (N-1)(l+1)R_s\\
c_2\cW(-(l+1)R_s) &=& \deg \cV + 2a + (N-2)(-(l+1)R_s)c_1 \cW\\
&=& \deg \cV + 2a - (N-2)(l+1).
\end{eqnarray*}

Plugging this into \eqref{GRR} gives:
\begin{eqnarray*}
\deg \, f_{t*}\cW(-(l+1)R_s) &=& 1 + [\deg \cV + a - (N-1)(l+1)] - [\deg \cV + 2a - (N-2)(l+1)]\\
&=& -a-l.
\end{eqnarray*}

Now consider the sequence which gives rise to the vector bundle $\cD$: \[0 \longrightarrow f_{t}^*f_{t*}\cW(-(l+1)R_s) \longrightarrow \cW(-(l+1)R_s) \longrightarrow \cD \longrightarrow 0\] The family of bundles ${\bf P}\cD_{t}$, $t \in \P^{1}_{t}$, is the family of directrices of ${\bf P}\cW_t$. If we let $\eta$ be the hyperplane class associated to the $\cO(1)$ on ${\bf P}\cW(-(l+1)R_s)$, we get the equality \[[{\bf P}\cD] = \eta^{N-1-r} +(a+l)R_t\cdot \eta^{N-2-r}\] in the Chow ring of ${\bf P}\cW(-(l+1)R_s)$.  If we denote by $\zeta$ the natural hyperplane class for ${\bf P}\cW$, we see that $\eta = \zeta - (l+1)R_s$, and so in terms of $\zeta$, \[[{\bf P}\cD] = \left( \zeta - (l+1)R_s\right)^{N-1-r} + (a+l)R_t\cdot ( \zeta - (l+1)R_s)^{N-2-r}.\] In the Chow ring of the ambient space ${\bf P}\cV$,\[[{\bf P}\cD] = (\zeta + aR_s+R_t)\cdot \left( \zeta - (l+1)R_s\right)^{N-1-r} + (a+l)R_t\cdot ( \zeta - (l+1)R_s)^{N-2-r}\]

Now, we want to understand ${\bf P}\cD \cap R_s$. In particular, we want to know its class in the Chow ring of $X = {\bf P}^{N-1}_{\{s=0\}} \times {\bf P}^1_t$. So we intersect with $R_{s}$ to obtain
\begin{eqnarray*}
[{\bf P}\cD] \cdot R_s = [\left( \zeta - (l+1)R_s\right)^{N-1-r} + (a+l)R_t\cdot ( \zeta - (l+1)R_s)^{N-2-r}]\cdot  [\zeta + aR_s+R_t]\cdot R_s
\end{eqnarray*}

Using the fact that $R_s^2 = R_t^2=0$, we can easily calculate this product. If we let $H$ and $F$ be the pullbacks of the hyperplane classes in ${\bf P}^{N-1}_{\{s=0\}} \times {\bf P}^1_{t}$, then the class of the scroll swept out by the directrices of $\P \cW$ is \[ [{\bf P}\cD \cap \{s=0\}] = H^{N-r} + (a+l+1)H^{N-r-1}\cdot F,\] which is what the proposition claims.
\end{proof}
Therefore, the ``degree" of this scroll swept out by the directrices in ${\bf P}^{N-1}_{\{s=0\}}$ is $a+l+1$.  By this we mean the following:  If $\Lambda^{N-r-1} \subset \P^{N-1}_{\{s=0\}}$ is a general $N-r-1$ dimensional subspace, then the family of $r-1$ dimensional planes $\P \cD_{t}$, $t \in \P^{1}_{t}$, viewed as subspaces of $\P^{N-1}_{\{s=0\}}$, will intersect $\Lambda^{N-r-1}$ at $a+l+1$ values of $t$.  

\subsubsection{} 
Let $k_{R} := \lceil 5(g_{R}+4)/6 \rceil$ and $m_{R} :=  \lceil -3(g_{R}+4)/4 \rceil$.  Then we have:

\begin{prop}\label{Mp}
Maintain the setting of section \ref{pentagonalpencil}.  The intersection number of the partial pencil $p$ with the Maroni divisor $\sfM$ is:
\[\sfM \cdot p = k_{R} + m_{R}\]
\end{prop}

\begin{proof}
This follows from \autoref{a+l+1} with the substitutions $\cV := \cE_{\F}, a = k_{R},$ and $l+1 = m_{R}$. 
\end{proof}

\begin{rmk}
There remains to understand the case when the splitting type of $\cW_{t} = \cO(l+1)^{N-1}$, i.e. the generic scroll is perfectly balanced. 

Then for finitely many $t \in \P^{1}_{t}$, the scroll $\P \cW_{t}$ will change splitting type to $\cO(l)\oplus \cO(l+1)^{\oplus N-2} \oplus \cO(l+2)$. The number of times this occurs is easily calculated, and the result is still $a+l+1$.
\end{rmk}

\subsubsection{Intersection numbers for the partial pencil $p$.}
Let $p$ be as in section \ref{pentagonalpencil}, and let $k_R$ and $m_{R}$ be as in the previous section. 

The split element in $p$ gives an intersection with a new divisor $\Delta_{\spl}$ with dual graph $\Gamma_{\Delta_{\spl}}$:

\[  \xygraph{
!{<0cm,0cm>;<1cm,0cm>:<0cm,1cm>::}
!{(0,0) }*{\bullet}="l"!{(-.3,.2)}{\scriptstyle v'_L}!{(-.3,-.3)}{\scriptstyle g_L+4}
!{(2,0) }*{\bullet}="r"!{(2.3,.2)}{\scriptstyle v'_R}!{(2.3,-.3)}{\scriptstyle g_R-4}
 "l"-@/^.6cm/"r"  
"l"-@/^/"r" 
"l"-"r"
"l"-@/_/"r"
"l"-@/_.6cm/"r" 
}  \]

We now collect the relevant intersection numbers of $p$. 
\begin{prop}\label{pentagonalpencilnumbers}
The partial pencil $p$ satisfies: 
\begin{align*}
\lambda \cdot p &= 2g_{R} + 3 -  k_{R}\\
\delta \cdot p &= 31g_{R} + 27 - 7k_{R} \\
\Delta \cdot p &= -1\\
\Delta_{\spl} \cdot p &= 1\\
\sfX \cdot p &= ((2g-22)/5)(k_{R} + m_{R})
\end{align*}
\end{prop}
\begin{proof}
The only number which needs explanation is $\sfX \cdot p$.   When $\sfX$ is written as a combination of $\sfM$ and $\sfCE$, the coefficient of $\sfM$ is $(2g-22)/5$.  Then we use \autoref{Mp}.
\end{proof}

When we intersect the partial pencil $p$ with the divisor class $\sfX$, we obtain the following relation: 
\begin{equation}\label{pentrelation}
c(\Delta, \sfY) = c(\Delta_{\spl},Y) + b(13g_{R} -7k_{R} +27) - a(2g_{R} -k_{R} +3)  + ((2g-22)/5)(k_{R} + m_{R}).
\end{equation}

\begin{lem}\label{pos}
Maintaining the notation above, we have the following inequality: 
\begin{equation}\label{ineq}
b(13g_{R} -7k_{R} +27) - a(2g_{R} -k_{R} +3)  + ((2g-22)/5)(k_{R} + m_{R}) \geq 3g - \frac{11}{2}g_{R}.
\end{equation}
\end{lem}
\begin{proof}
This is is straightforward - it follows by simplifying the left hand side after replacing $m_{R}$ by the quantity $\frac{-3(g_{R} + 4)}{4}$.
\end{proof}

The rest of the inductive procedure carries through as in the degree $3$ and $4$ cases.

\begin{prop}
Let $\Delta$ be as in section \ref{pentagonalunramified}. Then \[c(\Delta, \sfY) \geq 0.\] 
\end{prop}

\begin{proof}
Recall that as usual, we assume $g_{R} \leq g_{L}$.  This means that the quantity $3g-\frac{11}{2}g_{R}$ is nonnegative. We conclude by inductively using \autoref{pentrelation}.
\end{proof}

\subsubsection{Ramified $2$-vertex $\Delta'$.} (In this section, we ask the reader to pay attention to the prime in $\Delta'$.) It remains to establish $c(\Delta', \sfY) \geq 0$ for $\Delta'$ a ramified, $2$-vertex divisor.  For this, we do not proceed as in previous cases by imposing nonreduced base loci in the pencil $\P^{1}_{t} \subset |H_{1}|$, since doing so changes the analysis of $\sfM \cdot p$. 

Therefore, we take a different approach.  We use the families constructed in section \ref{modified2} where the surface $\F$ is the genus $1$ fibration used in previous sections. 

We begin with a {\sl general} pencil $\P^{1}_{t} \subset |H_{1}|$ of genus $g_{R}$ pentagonal curves, and fix a general fiber $F \subset \F$. 

Let $p$ denote the partial pencil family obtained after performing the necessary degree $5!$ base change $\beta : D \to \P^{1}_{t}$.  If $z \in \P^{1}_{t}$ is a branch point of the map $F \to \P^{1}_{t}$, then recall from section \ref{modified2} the set $A_{z} := \beta^{-1}(z)$.  The partial pencil $p$ is contained in a boundary divisor $\Delta$, and interacts with the higher boundary divisor $\Delta'$ corresponding to the union of the sets $A_{z} \subset D$ as $z$ varies over all branch points of $F \to \P^{1}_{t}$.  The number of branch points of $F \to \P^{1}_{t}$ is $10$, as follows from the Riemann-Hurwitz formula, after recalling that $F$ is an elliptic curve.  The sets $A_{z}$ have $5!/2$ points.

\begin{prop}\label{pentagonalbasechangefamily}
The family $p$ has the following intersection numbers:  
\begin{align*}
\lambda \cdot p &= 5! \cdot (2g_{R} + 3 -  k_{R})\\
\delta \cdot p &= 5! \cdot (31g_{R} + 32 - 7k_{R}) - 2400 \\
\Delta \cdot p &= -1080\\
\Delta' \cdot p &= 600\\
\sfX \cdot p &= 5! \cdot ((2g-22)/5)(k_{R} + m_{R})
\end{align*}

\end{prop}

\begin{proof}
Recall from section \ref{modified2} that the base change $D \to \P^{1}_{t}$ of degree $5!$ which is required in order to create the pencil $p$. If we let $X = Bl_{B}\F$ be the total space of the pencil $\P^{1}_{t}$ and $X_{D}$ denote the base change over $D$, then we end up with $5$ sections $\sigma_{1}, ... ,\sigma_{5}$.  (The reader should refer to section \ref{modified2}.) 

We first note that the sections $\sigma_{i}$ satisfy $(\sigma_{1} + ... + \sigma_{5})^{2} = 0$, since $F^{2}=0$.  Furthermore, $\sigma_{i}^{2}$  is independent of $i$, and $\sigma_{i}\sigma_{j}$ is independent of the subset $\{i,j\}$.

Furthermore, the total intersection of sections, $\sum_{i<j}\sigma_{i}\sigma_{j}$ is just the total number of points in $A_{z}$ as $z$ varies over the ten branch points of $F \to \P^{1}_{t}$.  This totals $600$.  Dividing by $5$ gives $\sigma_{i} \sigma_{j} = 600/10 = 60.$

From here, we conclude that $\sigma_{i}^{2} = -120.$ 

Now we must blow up $\sigma_{i}$ every time it meets another section $\sigma_{j}$.  This happens $4 \times 60 = 240$ times for a fixed $i$. Moreover, in order to create a family of admissible covers, the section $\sigma_{i}$ must be blown up {\sl twice} each time  two other sections meet, which happens ${4 \choose 2} \times 60 = 360$ times.  

Therefore, in the ultimate blown up surface, the proper transform $\td{\sigma_{i}}$ has self intersection  $\td{\sigma_{i}}^{2} = \sigma^{2} - 240 - 720 = -1080.$  This explains $\Delta \cdot p$.

Furthermore, the intersection number $\delta \cdot p$ is calculated in a similar way:  One must take a little bit of care however. Since $\delta$ is pulled back from $\Mg$, we must stabilize the family $p$ first, and then compute $\delta$.  This has the effect is to not blow up $\sigma_{i}$ twice at points where two other sections meet.  We conclude the calculation of $\delta$ from \autoref{pentnum}.

The calculation of $\lambda$ follows clearly from \autoref{pentnum}.  The calculation of $\Delta'$ comes from the observation that there are $600$ points in the union of all sets $A_{z}$.  The intersection calculation with the Maroni and Casnati-Ekedahl divisors proceeds as in section \ref{Maronispecialanalysis} -- at this point we leave these details to the reader.
\end{proof}

\begin{cor}
Maintaining the setting above,  we have $c(\Delta', \sfY) \geq 0$. 
\end{cor}
\begin{proof}
This follows as usual by intersecting the pencil $p$ with $\sfX = a\lambda - b\delta - \sfY$ and using \autoref{pentagonalbasechangefamily}.
\end{proof}

\subsubsection{Higher ramification $2$-vertex divisors $\Delta'$} The analysis for higher ramification, two vertex divisors proceeds just as in the previous case: we use the same families, except we impose that the pencil $\P^{1}_{t}$ has a profile of ramification $(m_{1}, ... ,m_{k})$ as in section \ref{branchingprofilepentagonalfamily}.  We leave it to the reader to check using the same strategy as in the proof of \autoref{pentagonalbasechangefamily} that the resulting partial pencil family shows that $c(\Delta', \sfY) \geq  0$.

\subsubsection{The proof of \autoref{Yeffective} and \autoref{slope45}.} We now have all the relevant ingredients for the proof of \autoref{Yeffective}. 

\begin{proof}[Proof of \autoref{Yeffective}]
The inequalities presented in this section allow us to conclude \autoref{Yeffective}.
\end{proof}

\begin{proof}[Proof of \autoref{slope45}]\label{proofslope45}

The congruence conditions in \autoref{slope45} precisely correspond to pairs $(d,g)$ where both $\sfM$ and $\sfCE$ exist as divisors.

Furthermore, the sweeping one-parameter families achieving the slope bounds are precisely the families found in the three parts of the proof of \autoref{rigidityextremality}, since they avoid all higher boundary and do not intersect $\sfM$ and $\sfCE$.

\autoref{slope45} follows immediately from \autoref{Yeffective}, which we just proved.  One point of care:  The statement of \autoref{slope45} involves the standard compactification of twisted stable maps $\Hdg$.  We have technically only proven the result for the spaces $\Hdgnt$.  The modification for the proof in the case of $\Hdg$ is simple: There is a stabilization map $\Hdg \to \Hdgnt$ which is finite and bijective on points. Thus, the coefficients of $\sfY$ in both spaces have the same sign.

\end{proof} 

\begin{cor}
The slope bounds in \autoref{slope45} also apply to the respective loci $\cM^{1}_{d,g} \subset \Mg$.
\end{cor}

\begin{proof}
This follows from \autoref{transportsweeping}.
\end{proof}

\section{Further directions and open questions.} 

Our purpose in this paper was to introduce the relevance of the loci of covers possessing atypicial behaviour in their Casnati-Ekedahl resolutions. Many questions are left unanswered, and we present some which intrigue us in this section.

For all $i$, let $$Z_{i} := \{\alpha \in \H_{d,g} \mid \text{$\cN_{i}$ is unbalanced}\}.$$
\begin{question}[Generic balancedness] \label{question:Nibalanced}
Is $Z_{i}$ a proper subset of $\H_{d,g}$?
\end{question}

As mentioned already, the question is already interesting for the bundle of quadrics $\cN_{1}$, see \cite{bujokas-patel} and \cite{bopp-hoff}.

The next question concerns the expected dimension of the locus $Z_{i}$. 

\begin{question}[Expected dimension]\label{question:Ziexpecteddimension}
Are the loci $Z_{i}$ of expected dimension? When is $Z_{i}$ divisorial?  
\end{question}

We do not know the answer to \autoref{question:Ziexpecteddimension} even in the low degree regime.  Specifically, we know all Casnati-Ekedahl loci have expected dimension when $d=4$ (as a simple consequence of the {\sl trigonal construction}), but we do not know this beginning with $d=5$.

\begin{question}
Which bundles $\cN_{i}$ are realized in the Casnati-Ekedahl resolution of a cover? 
\end{question}

This question seems extremely difficult. The general problem of determining which invariants are realizable by a cover is already wide open for the bundle $\cN_{1}$.  The same holds for the Tschirnhausen bundle $\cE$.

\begin{question}\label{contract}
Is it possible to contract the Maroni divisor $\sfM$ or the Casnati-Ekedahl divisor $\sfCE$?
\end{question}

A candidate for the contraction of $\sfM$ is a GIT quotient of a Hilbet scheme. The Maroni divisor $\sfM$ exists precisely when the scroll $\P \cE$ associated to a general branched cover is trivial, i.e. $\P \cE = \P^{d-2} \times \P^{1}$. 

We let $V_{d,g}$ denote the closure of the scheme of relative canonical at-worst-nodal curves in $\P^{d-1}\times \P^{1}$ of degree genus $g$. $V$ is a closed subscheme of the appropriate Hilbert scheme. Let $G = \PGL_{d-1}(k) \times \PGL_{2}(k)$.

There is then a birational map $$\varphi : \o{H}_{d,g} \dashrightarrow V_{d,g}//G$$
where the latter space is a GIT quotient. (It is unclear what linearization we would want.)

In all understood cases, which is not many, a aparameter count suggests that the Maroni divisor $\sfM$ experiences a contraction under $\varphi$. If $\varphi$ were a regular morphism, a contraction of $\sfM$ would be enough to conclude rigidity and extremality. 

However, since $\varphi$ is most probably not a morphism, we would need to more carefully analyze the resolution of its indeterminacy (which is contained in $\sfM$).  This analysis currently seems somewhat daunting, even in the case $d=4$.

\autoref{contract} is related to establishing the rigidity of the respective loci.  In fact, it could be the case that some multiple of the Casnati-Ekedahl divisor $\sfCE$ may not be rigid.  Together with Anand Deopurkar, the author has computed the Bogomolov expressions of the higher syzygy bundles, and the result is surprising:  All Bogomolov expressions are multiples of the divisor class of $\sfCE$.  As pointed out to the author by Deopurkar, this is very similar to the story of the Brill-Noether divisors in $\M_{g}$.  

Assuming the supports of the divisorial loci $Z_{i}$ are distinct, we would conclude that some multiple of $\sfCE$ would not be rigid.  This does not rule out extremality, however.

Even in the low degree regime, interesting questions remain. Recall that when $d=3$ and the genus $g$ is odd, there is the {\sl tangency divisor} ${\rm Tan}$ which generically parmaterizes covers which are tangent to the directrix in $\F_{1}$.  This divisor is extremal, and no multiple of it deforms, as can be proved using the same technique as in the proof of \autoref{Mextremal3}. It also plays the key role in the slope problem for trigonal curves of odd genus \cite{deopurkar_patel:sharp_trigonal}.

\begin{question}
What are the higher degree analogues of ${\rm Tan}$?
\end{question}

More generally, 

\begin{question}[Effective cone]
What are all extremal rays of the effective cone $\Eff \Hdgnt$?
\end{question}

We do not know the answer to this question even when $d=3$, although we suspect that the Maroni divisor or the tangency divisor, along with boundary components generate the effective cone. 

One consequence of \autoref{slope45} is that the pentagonal locus $\cM^{1}_{5,g} \subset \Mg$ cannot be swept out by $K_{\Mg}$--negative curves, because such a curve would have slope $\geq 6 \frac{1}{2}$, while \autoref{slope45} does not allow for this. (Of course we are only considering the appropriate genera occurring in \autoref{slope45}.)

Unfortunately, one cannot immediately conclude from this that $\cM^{1}_{5,g}$ does not lie in the base locus of the canonical divisor $K_{\Mg}$, but we still ask the question:

\begin{question}
Does the pentagonal locus $\cM^{1}_{5,g}$ play a fundamentally different role from the trigonal or tetragonal locus in the log MMP program for $\Mg$?
\end{question}

Finally, one can ask about generalizing the ideas in this paper to Hurwitz spaces of covers of higher genus curves.  In a forthcoming paper with Gabriel Bujokas, we will investigate the analogues of Maroni and Casnati-Ekedahl divisors in spaces of branched covers of elliptic curves.

   \bibliographystyle{amsalpha}
   \bibliography{CommonMath}

\newcommand{\Addresses}{{
  \bigskip
  \footnotesize

 \textsc{Department of Mathematics, Boston College, Chestnut Hill, MA 02467}
 \par\nopagebreak
  \textit{E-mail address}: \texttt{anand.patel@bc.edu}
}}

\Addresses

  \end{document}